\documentclass[11pt,reqno]{amsart}

\pdfoutput=1

\usepackage{amscd,amsthm,amsmath,amssymb,mathtools}
\usepackage{verbatim}
\usepackage{color}
\usepackage{url}
\usepackage{enumerate,enumitem}
\usepackage{graphicx}
\usepackage{epsfig,subfigure}
\usepackage{curve2e}

\theoremstyle{plain}
\newtheorem{theorem}{Theorem}[section]

\newtheorem{lemma}[theorem]{Lemma}
\newtheorem{corollary}[theorem]{Corollary}
\newtheorem{problem}[theorem]{Problem}

\theoremstyle{definition}
\newtheorem{definition}[theorem]{Definition}
\newtheorem{remark}[theorem]{Remark}

\usepackage{geometry}

\newcommand{\linspan}{\mathop{\rm span}\nolimits}

\newcommand{\supp}{\mathop{\rm supp}\nolimits}
\newcommand{\diver}{\mathop{\rm div}\nolimits}

\newcommand{\rest}{\left.\kern-2\nulldelimiterspace\right|_}
\newcommand{\norm}[2]{\left|#1\right|_{#2}}

\newcommand{\ex}{\mathrm{e}}
\newcommand{\p}{\partial}

\newcommand{\ed}{\mathrm d}

\newcommand{\N}{{\mathbb N}}
\newcommand{\Z}{{\mathbb Z}}
\newcommand{\R}{{\mathbb R}}
\newcommand{\T}{{\mathbb T}}

\newcommand{\E}{{\mathbb E}}

\newcommand{\D}{{\mathrm D}}

\newcommand{\AAA}{{\mathbf A}}
\newcommand{\BBB}{{\mathbf B}}
\newcommand{\CCC}{{\mathbf C}}
\newcommand{\RRR}{{\mathbf R}}

\newcommand{\nnn}{\mathbf n}
\newcommand{\ttt}{\mathbf t}

\newcommand{\aA}{{\mathcal A}}
\newcommand{\BB}{{\mathcal B}}

\newcommand{\EE}{{\mathcal E}}
\newcommand{\FF}{{\mathcal F}}
\newcommand{\GG}{{\mathcal G}}
\newcommand{\HH}{{\mathcal H}}
\newcommand{\II}{{\mathcal I}}

\newcommand{\KK}{{\mathcal K}}
\newcommand{\LL}{{\mathcal L}}
\newcommand{\MM}{{\mathcal M}}
\newcommand{\NN}{{\mathcal N}}
\newcommand{\OO}{{\mathcal O}}

\newcommand{\RR}{{\mathcal R}}
\newcommand{\sS}{{\mathcal S}}

\newcommand{\WW}{{\mathcal W}}
\newcommand{\XX}{{\mathcal X}}
\newcommand{\YY}{{\mathcal Y}}
\newcommand{\ZZ}{{\mathcal Z}}

\newcommand{\eqrmref}[1]{(\roman{rmnum})} 
      
\newcommand*{\Bigcdot}{\raisebox{-.25ex}{\scalebox{1.25}{$\cdot$}}}
 
\begin{document}
\title{Feedback boundary stabilization to trajectories
for 3D Navier--Stokes equations}
\author{S\'ergio S.~Rodrigues$^*$}
\thanks{$^*$Johann Radon Institute for Computational and Applied Mathematics,
\"OAW, Altenbergerstra{\normalfont\ss}e 69, A-4040 Linz, Austria.\quad Fax: +43 732 2468 5212.\\
E-mail: {\small\tt sergio.rodrigues@oeaw.ac.at},\quad Tel: +43 732 2468 5241.\hfill \today}


\begin{abstract}
Given a nonstationary trajectory
of the Navier--Stokes system, a finite-dimensional feedback boundary
controller stabilizing locally the system to the given trajectory is derived. Moreover the controller is
supported in a given open subset of the boundary
of the domain containing the fluid.

In a first step a controller is
derived that stabilizes the linear Oseen--Stokes system ``around the given trajectory'' to zero;
for that a corollary of a suitable truncated boundary observability inequality, the
regularizing property for the system, and some
standard techniques of the optimal control theory are used. Then it is shown 
that the same controller also stabilizes,
locally, the Navier--Stokes system to the given trajectory.
\end{abstract}

\maketitle



\pagestyle{myheadings} \markboth{\sc S.~S.~Rodrigues}{\scriptsize\sc Boundary stabilization for {3D} Navier--Stokes equations}

{\footnotesize
\tableofcontents
}
%
%
%
%
%
%
%
%
%
%


\section{Introduction}
Let $\Omega\subset\R^3$ be a connected open bounded subset located locally on
one side of its smooth boundary $\Gamma=\p\Omega$, and let $I\subseteq\R$ be a nonempty open interval. The
Navier--Stokes system, in~$I\times\Omega$, controlled through the boundary reads
\begin{equation}\label{sys-u-bdry}
 \p_t u+\langle u\cdot\nabla\rangle u-\nu\Delta u+\nabla p_u+h=0, \quad \diver u
 =0,\quad u\rest \Gamma =\gamma+\zeta
\end{equation}
where~$\zeta$ is a control
taking values in a suitable subspace of square-integrable functions
in~$\Gamma$ whose support in~$x$ is contained in a given open
subset $\Gamma_{\rm c}\subseteq\Gamma$. Furthermore, as usual, $u=(u_1,u_2,u_3)$ and~$p_u$,
defined for $(t,\,x_1,\,x_2,\,x_3)\in I\times\Omega$, are the unknown velocity field and
pressure of the fluid, $\nu>0$ is the viscosity, the operators $\nabla$ and $\Delta$ are respectively the
well known gradient and Laplacian in the space variables $(x_1,\,x_2,\,x_3)$, $\langle
u\cdot\nabla\rangle v$ stands for $(u\cdot\nabla v_1,\,u\cdot\nabla v_2,\,u\cdot\nabla v_3)$,
$\diver u\coloneqq \p_{x_1}u_1+\p_{x_2}u_2+\p_{x_3}u_3$, and $h$ and $\gamma$~are fixed external forces.

Suppose we are given a targeted (reference, desired) solution $\hat u(t)=\hat u(t,\,x)$ of~\eqref{sys-u-bdry}
with $I=(0,\,+\infty)$ and $\zeta=0$. If $\hat u$ is stationary, $\hat u(t)=\hat u(0,\,x)=\hat u_0$, then the problem of
stabilization to $\hat u_0$ is now quite well understood. Namely, it was proven
that, for any initial
function~$u_0$ sufficiently close to~$\hat u_0$ one can find  a
square integrable control $\zeta\in L^2((0,\,+\infty),\,L^2(\Gamma,\,\R^3))$, such that
the corresponding solution $u(t)$, supplemented with the initial
condition
\begin{equation} \label{ini-u}
u(0,x)=u_0(x)
\end{equation}
is defined on~$[0,+\infty)$ and $u(t)$ goes to $\hat u_0$ exponentially as time $t$ goes to $+\infty$;
we refer the reader to the
works~\cite{Badra09-cocv,Barbu12,BarbuLas12,BadTakah11,Fursikov01,Fursikov04,Raymond06,Raymond07,RaymThev10}.

Again for a stationary targeted trajectory $\hat u=\hat u_0$, the analogous result hold also in the case of an internal control under
Dirichlet boundary conditions:
\begin{equation}\label{sys-u-int}
 \p_t u+\langle u\cdot\nabla\rangle u-\nu\Delta u+\nabla p_u+h+\eta=0, \quad \diver u =0,
\qquad u\rest \Gamma =0.
\end{equation}
where now $\eta$ is a control supported in a given open subset
$w\subseteq\Omega$. For details we refer
to~\cite{Barbu03,BarbuLasTri06,BarbuTri04}.

Here, we are particularly interested in the case where the targeted trajectory $\hat u$ is nonstationary
(i.e., $\hat u=\hat u(t)$ depends on time),
a situation that often can occur in real world applications, as in the case suitable (say non-gradient)
external forces ($h$ and $\gamma$)
depend on time. Also, since they are important and often required in applications,
we look for controls obeying some general constraints like to be
given in feedback form, finite-dimensional, and supported in a given (small) open subset.

In~\cite{BarRodShi11}, an {\em internal} stabilizing
finite-dimensional feedback controller was found for the case of
nonstationary targeted solutions. Then, one question arises: can we find a similar
{\em boundary} controller?  The methods used in the particular case of a stationary
targeted solution,
use some (spectral-like) properties of the (time-independent) Oseen--Stokes operator
$u\mapsto \nu\Delta u-\BB(\hat u_0)u+\nabla p_u$ and/or of its ``adjoint''
$v\mapsto \nu\Delta v-\BB^*(\hat u_0)v+\nabla p_v$, which seem to give us no hint for the nonstationary case. Here
$\BB(\hat u_0)v\coloneqq\langle \hat u_0\cdot\nabla\rangle v+\langle v\cdot\nabla\rangle \hat u_0$ and~$\BB^*(\hat u_0)$ is the formal adjoint of~$\BB(\hat u_0)$.

Also the constraints on the boundary control, imply that the procedure in~\cite{FurIma99},
that allow to derive suitable boundary results from internal ones is (or may be) no longer sufficient to
derive the wanted boundary stabilization result.

Departing from an exact controllability result in~\cite{Rod14-na}, suitable truncated boundary observability inequalities have been derived for the
(linear) Oseen--Stokes system in~\cite{Rod15-cocv}. These results
will enable us to follow the procedure in~\cite{BarRodShi11} in order
to construct a boundary stabilizing finite-dimensional controller to a given nonstationary targeted solution.
To prove that the control can be taken in feedback form, and to find the feedback rule,
we will need to overcome some technical regularity/compatibility issues.

More precisely, we will find the feedback rule by considering, at a first step, the linear Oseen--Stokes system
\begin{subequations}\label{sys-oseen}
\begin{align}
 &\p_t v+\BB(\hat u)v-\nu\Delta v+\nabla p_v=0, \label{sys-oseen-dyn}\\
 &\diver v =0, \qquad v\rest \Gamma =\zeta,\label{sys-oseen-div-bd}
\end{align}
\end{subequations}
and we look for the control~$\zeta$ that minimizes a suitable cost~$J(v,\,\zeta)$; by the Karush-Kuhn-Tucker Theorem and the dynamical programming principle we will conclude that the control is given in feedback form.
Though this is a standard procedure, also used in~\cite{BarRodShi11}, in the boundary control case we meet some nontrivial regularity/compatibility issues.
Roughly speaking in the internal case to have the feedback rule at time $t$ we need to know~$q(t)\in L^2(\Omega,\,\R^3)$
where~$q$ is a suitable Lagrange multiplier, associated with the constraint~\eqref{sys-oseen-dyn}, that solves a system ``adjoint'' to~\eqref{sys-oseen},
while in the boundary case we will need to know~$\nnn\cdot\nabla q(t)+p_{q(t)}\nnn$ on the boundary~$\Gamma$, where~$\nnn$ is the unit outward normal vector to the boundary~$\Gamma$ and~$p_{q(t)}$ is a function
depending on~$q(t)$. Thus, we will need more regularity for~$q$. 
One possible way to get more regularity for~$q$ is to consider a cost functional~$J(v,\,\zeta)$ that penalizes~$v$ in a larger (i.e., less regular) space, but we need to keep enough regularity for the optimal
solution~$v$ to give a meaning to~$v(t)$, because we recall we want the control~$\zeta$ in feedback form, that is,
we want~$\zeta$, in~\eqref{sys-oseen-div-bd}, as a function of~$v(t)$: $\zeta(t)=K(t,\,v(t))$. We shall guarantee enough
regularity for both~$q$ and~$v$ by considering an appropriate cost functional and an appropriate auxiliary extended system (cf.~\cite{Badra09-cocv}).

We shall prove the following Theorem, whose exact formulation is given in
Section~\ref{S:nonlinear}.

\noindent {\sc Main Theorem:} {\it Let~$(\hat u,p_{\hat u})$ be a
global smooth solution for problem~\eqref{sys-u-bdry}, with
$\zeta=0$ and $t\in\R_0=(0,\,+\infty)$, such that
\[
|\hat u|_{L^\infty(\R_0\times\Omega,\,\R^3)}+\sup_{\tau\in[0,\,+\infty)}|\p_t\hat
u|_{L^2((\tau,\,\tau+1),\,L^\sigma(\Omega,\,\R^3))}+\sup_{\tau\in[0,\,+\infty)}|\nabla\hat
u|_{L^2((\tau,\,\tau+1),\,L^3(\Omega,\,\R^9))}\le R
\]
where $R>0$ and $\sigma>\frac{6}{5}$ are constants.

Then for
any $\lambda>0$ and any open subset $\Gamma_{\rm c}\subseteq\Gamma=\p\Omega$
there are an integer $M=M(R,\lambda,\Gamma_{\rm c})\ge1$, an
$M$-dimensional space $\EE_M\subset \{f\in C^2(\Gamma,\R^3)\mid f\rest{\Gamma\setminus\Gamma_{\rm c}}=0\}$, and
a family of continuous linear operators $\mathbf K_{\hat u}^{\lambda,\,t}$, $t\ge0$,
from a suitable subset of $L^2(\Omega,\R^3)$ into $\EE_M$, such that the following
assertions hold.
\begin{itemize}
\item[\rm(a)] The function $t\mapsto \mathbf K_{\hat u}^{\lambda,\,t}$ is continuous in
the weak operator topology, and its operator norm is bounded by a
constant depending only on~$R$, $\lambda$, and~$\Gamma_{\rm c}$.
\item[\rm(b)] For any divergence free function $u_0\in
H^1(\Omega,\R^3)$ such that the difference~$u_0-\hat u(0)$ is
sufficiently small in the $H^1(\Omega,\R^3)$-norm and $(u_0-\hat u(0))\rest\Gamma\in \EE_M$,
problem~\eqref{sys-u-bdry}--\eqref{ini-u} with~$\zeta(t)=(u_0-\hat u(0))\rest\Gamma+\int_0^t\mathbf K_{\hat u}^{\lambda,\,r}(u-\hat u(r))\,\ed r$
has a unique global solution $(u,p_u)$,
which satisfies the inequality
\begin{equation*}
|u(t)-\hat u(t)|_{H^1(\Omega,\R^3)}^2\le C\ex^{-\lambda t}|u_0-\hat
u(0)|_{H^1(\Omega,\R^3)}^2, \quad t\ge0.
\end{equation*}
\end{itemize}
}

At this point we should say that this Theorem remains true for the two-dimensional (2D) case.
Though we will focus on the three-dimensional (3D) case, the procedure is still valid for the two-dimensional one.

The fact that the control appears in integral form is meaningful from the physical point of view; indeed, since $u$ is the velocity of the
fluid then, roughly, the integral form means that we are accelerating (or forcing) the fluid particles through the boundary.
From the physical and practical point of view this is more natural than instantaneously imposing the velocity of the boundary particles.

Though the integral feedback form of the controller, we will also show that the control
is defined pointwise in time, that is, the control $\zeta(t)$ at time $t>0$ depends only on $u(t)-\hat u(t)$, and not on
the trace $(u_0-\hat u(0))\rest\Gamma$
as the integral feedback form could suggest.

The feedback control we are going to construct will have both tangent and normal components. In some particular cases,
in the case the targeted trajectory $\hat u$ is stationary, the stabilization
of the Navier--Stokes system by normal boundary controls is proven in~\cite{BalLiuKrstic01,Barbu07,Munteanu12,Munteanu3D12,VazquezKrstic05}.
In the general case, stabilization to a stationary solution in the 2D case has been achieved in~\cite{Barbu12}, under some general conditions,
by means of oblique controls. This oblique stabilization result also holds in the 3D case for the linear Oseen--Stokes system.
We would like to refer also to the work~\cite{Barbu_TAC13} where the idea in~\cite{Barbu12} is used for
boundary stabilization to a stationary solution of parabolic equations, and leads to a simple algorithm to construct the
stabilizing controller. Finally for the stabilization to a stationary trajectory by means of tangential controls we refer
to~\cite{BarbuLasTri06,BarbuLas12}.

\medskip
The rest of the paper is organized as follows. In
Section~\ref{S:prelim}, we introduce some functional spaces arising
in the theory of the Navier--Stokes equations and recall some
well-known facts. Sections~\ref{S:linear-exist} and~\ref{S:linear-feed} are devoted
to studying the linearized problem, that is, the Oseen--Stokes system; in Section~\ref{S:linear-exist} we prove the existence of a stabilizing control and
in Section~\ref{S:linear-feed} we prove that the control can be taken in feedback form. Finally in Section~\ref{S:nonlinear}
we establish the main result of the paper on local exponential
stabilization of the Navier--Stokes system. The Appendix gathers a few more remarks concerning some points in the main text.

\medskip\noindent
{\bf Notation.}
We write~$\R$, $\Z$, and~$\N$ for the sets of real, integer, and nonnegative
integer numbers, respectively, and we define $\R_a\coloneqq (a,\,+\infty)$ for all $a\in\R$, and $\mathbb
N_0\coloneqq \mathbb N\setminus\{0\}$. We denote by $\Omega\subset\R^3$ a
bounded domain with a smooth boundary $\Gamma=\p\Omega$.
Given a vector function $v\colon (t,x_1,\,x_2,\,x_3)\mapsto v(t,x_1,\,x_2,\,x_3)\in\R^k$, $k\in\N_0$, defined in
an open subset of $\R\times\Omega$, its partial
time derivative~$\frac{\p v}{\p t}$ will be denoted by $\p_tv$. Also the spatial partial
derivatives~$\frac{\p v}{\p x_i}$ will be denoted by $\p_{x_i}v$.

Given an open interval $I\subseteq\R$, then we write
$W(I,\,X,\,Y)\coloneqq \{f\in L^2(I,\,X)\mid \p_tf\in L^2(I,\,Y)\}$, 
where the derivative $\p_tf$ is taken in the sense of
distributions. This space is endowed with the natural norm
$|f|_{W(I,\,X,\,Y)}\coloneqq \bigl(|f|_{L^2(I,\,X)}^2+|\p_tf|_{L^2(I,\,Y)}^2\bigr)^{1/2}$. In the case $X=Y$ we write $H^1(I,\,X)\coloneqq W(I,\,X,\,X)$.
Again, if $X$ and $Y$ are endowed
with a scalar product, then also $W(I,\,X,\,Y)$ is.
The space of continuous linear mappings from $X$ into $Y$ will be denoted by $\LL(X\to Y)$. In the case $X=Y$ we write simply~$\LL(X)$. 

$\overline C_{\left[a_1,\dots,a_k\right]}$ denotes a function of nonnegative variables~$a_j$ that
increases in each of its arguments.

$C,\,C_i$, $i=1,2,\dots$, stand for unessential positive constants.

\section{Preliminaries}\label{S:prelim}

\subsection{Functional spaces}\label{sS:funspaces}

Let $\Omega\subset \R^3$ be a connected bounded domain of class $C^\infty$
located locally on one side of its boundary
$\Gamma=\partial \Omega$, with~$\int_\Gamma\ed\Gamma<+\infty$.

We recall some spaces appearing in the study of the
system~\eqref{sys-u-bdry} (cf.~\cite{Rod14-na,Rod15-cocv}). We start by the Lebesgue and Sobolev subspaces
\begin{align*}
L^r_{\diver}(\Omega,\,\R^3)&\coloneqq\{u\in L^r(\Omega,\R^3)\mid \diver u=0\text{ in
}\Omega\},\quad 1\leq r\leq+\infty,\\
H^s_{\diver}(\Omega,\,\R^3)&\coloneqq\{u\in H^s(\Omega,\R^3)\mid \diver u=0\text{ in
}\Omega\},\quad s\geq0.
\end{align*}

The incompressibility condition allows us to define the trace of
$u\cdot\nnn$ on the boundary $\Gamma=\p\Omega$, where~$\nnn$ is the unit outward
normal vector to the boundary~$\Gamma$, and then to write
\begin{equation*}
H\coloneqq\{u\in
L^2_{\diver}(\Omega,\,\R^3)\mid u\cdot\nnn=0 \text{ on }\Gamma\},\quad
H_{\rm c}\coloneqq\{u\in L^2_{\diver}(\Omega,\,\R^3)\mid  u\cdot\nnn=0 \text{ on
}\Gamma\setminus\overline{\Gamma_{\rm c}}\},
\end{equation*}
where $\Gamma_{\rm c}$ is an open subset of $\Gamma$. Some spaces of more regular vector fields we find throughout the
paper are
\begin{align}
V\coloneqq\{u\in H^1_{\diver}(\Omega,\,\R^3) \mid  u=0 \text{ on }\Gamma\},&\quad
V_{\rm c}\coloneqq\{u\in H^1_{\diver}(\Omega,\,\R^3) \mid  u=0 \text{ on
}\Gamma\setminus\overline{\Gamma_{\rm c}}\},\notag\\
\D(L)&\coloneqq V\cap H^2(\Omega,\R^3).\label{DL}
\end{align}

The spaces $H^s_{\diver}(\Omega,\,\R^3)$ are endowed with the scalar
product inherited from $H^s(\Omega,\R^3)$;
the spaces~$H$ and $H_{\rm c}$  with that inherited from $L^2(\Omega,\R^3)$; the
spaces~$V$ and $V_{\rm c}$ with that inherited from $H^1(\Omega,\R^3)$; and $\D(L)$ with
that inherited from $H^2(\Omega,\R^3)$. Notice
that denoting by~$\Pi$ the orthogonal projection in~$L^2(\Omega,\R^3)$
onto~$H$, it is well known that $\D(L)$ coincides with the
domain $\{u\in V| Lu\in H\}$ of the Stokes operator
$L\coloneqq-\nu\Pi\Delta$. That is the reason for the notation.

Next, fix a constant $\sigma>\frac{6}{5}$. For any pair of real numbers $a,\,b$, with $a<b$,
we introduce the Banach spaces~$\WW^{(a,\,b)|\rm wk}$ and~$\WW^{(a,\,b)|\rm st}$  of
the measurable vector functions $u=(u_1,u_2,u_3)$, defined in~$(a,\,b)\times\Omega$, satisfying
\begin{equation}\label{Wabspaces}
\begin{array}{rcl}
|u|_{\WW^{(a,\,b)|\rm wk}}&\coloneqq&\left(|u|_{L^\infty((a,\,b),\,L_{\diver}^\infty(\Omega,\,\R^3))}^2
+|\p_t u|_{L^2((a,\,b),\,L^\sigma(\Omega,\,\R^3))}^2\right)^{\frac{1}{2}}<\infty,\\
|u|_{\WW^{(a,\,b)|\rm st}}&\coloneqq&\left(|u|_{\WW^{(a,\,b)|\rm wk}}^2
+|\nabla u|_{L^2((a,\,b),\,L^3(\Omega,\,\R^9))}^2\right)^{\frac{1}{2}}<\infty.
\end{array}
\end{equation}
and also the Morrey-like spaces
\begin{equation}\label{Wspaces}
\begin{array}{rcl}
 \WW^{\rm wk}&:=&\biggl\{u\;\Bigr|\; \sup_{\tau\in[0,\,+\infty)}\bigl|u\rest{(\tau,\,\tau+1)}\bigr|_{\WW^{(\tau,\,\tau+1)|\rm wk}}
 <+\infty\biggr\},\\
 \WW^{\rm st}&:=&\biggl\{u\;\Bigr|\; \sup_{\tau\in[0,\,+\infty)}\bigl|u\rest{(\tau,\,\tau+1)}\bigr|_{\WW^{(\tau,\,\tau+1)|\rm st}}
 <+\infty\biggr\};
\end{array}
\end{equation}
endowed with the norms
${\displaystyle
|u|_{\WW^{\rm wk}}:=\sup_{\tau\ge0}|u|_{\WW^{(\tau,\,\tau+1)|\rm wk}}}$, and ${\displaystyle
|u|_{\WW^{\rm st}}:=\sup_{\tau\ge0}|u|_{\WW^{(\tau,\,\tau+1)|\rm st}}}$.

\begin{remark}
The lower bound~$\frac{6}{5}$ for~$\sigma$
is motivated from the results in~\cite{FerGueImaPuel04,Rod15-cocv}. 
\end{remark}

We recall that, in~\cite{FurGunzHou02}, 
the set of traces $u\rest \Gamma$ at the boundary $\Gamma$ of the
elements $u$ in the space
$
W((a,\,b),\,H^s_{\diver}(\Omega,\,\R^3),\,H^{s-2}(\Omega,\,\R^3))
$
is completely characterized, for each $s>\frac{1}{2}$, with
$s\notin\left\{\frac{3}{2},\,\frac{5}{2}\right\}$. Denoting
that trace space by $G^s_{\rm av}((a,\,b),\,\Gamma)$, we have that $v\mapsto v\rest{\Gamma}$ is continuous:
\[
\bigl|v\rest{\Gamma}\bigr|_{G^s_{\rm av}((a,\,b),\,\Gamma)}
\leq C_1|w|_{W((a,\,b),\,H^s_{\diver}(\Omega,\,\R^3),\,H^{s-2}(\Omega,\,\R^3))}
\]
and, there is an extension
$
E_s\colon G^s_{\rm av}((a,\,b),\,\Gamma)\to W((a,\,b),\,
H^s_{\diver}(\Omega,\,\R^3),\,H^{s-2}(\Omega,\,\R^3)),
$
which is continuous:
\[
(E_sw)\rest \Gamma=w\mbox{ and }|E_sw|_{W((a,\,b),\,H^s_{\diver}(\Omega,\,\R^3),\,H^{s-2}(\Omega,\,\R^3))}
\leq C_2|w|_{G^s_{\rm av}((a,\,b),\,\Gamma)}.
\]
We will use only the cases $s=1$ and $s=2$. From~\cite[Section~2.2]{FurGunzHou02} we know that
\[
G^s_{\rm av}((a,\,b),\,\Gamma)=G_\ttt^s((a,\,b),\,\Gamma)\oplus G_{\nnn,\rm av}^s((a,\,b),\,\Gamma)\nnn,
\]
with
\[
 \begin{array}{rcl}
 G_\ttt^1((a,\,b),\,\Gamma)&=&L^2((a,\,b),\,H^{\frac{1}{2}}(\Gamma,\,T\Gamma))\cap
H^{\frac{1}{2}}((a,\,b),\,H^{-\frac{1}{2}}(\Gamma,\,T\Gamma))\\
\hspace*{-.35em}G_{\nnn,\rm av}^1((a,\,b),\,\Gamma)&=&L^2((a,\,b),\,H^{\frac{1}{2}}_{\rm av}(\Gamma,\,\R))\cap
H^{\frac{3}{4}}((a,\,b),\,H^{-1}_{\rm av}(\Gamma,\,\R)) \\
G_\ttt^2((a,\,b),\,\Gamma)&=&L^2((a,\,b),\,H^{\frac{3}{2}}(\Gamma,\,T\Gamma))\cap
H^{\frac{3}{4}}((a,\,b),\,H^{0}(\Gamma,\,T\Gamma))\\
\hspace*{-.35em}G_{\nnn,\rm av}^2((a,\,b),\,\Gamma)&=&L^2((a,\,b),\,H^{\frac{3}{2}}_{\rm av}(\Gamma,\,\R))\cap
H^{1}((a,\,b),\,H^{-\frac{1}{2}}_{\rm av}(\Gamma,\,\R))
 \end{array}
\]
where the subscript~{``$\rm av$''} stands for the zero averaged condition, that is, $H^{r}_{\rm av}(\Gamma,\,\R)
\coloneqq\{f\in H^{r}(\Gamma,\,\R)\mid \int_\Gamma f\,\ed\Gamma=0\}$, and~$T\Gamma$ stands for the tangent bundle of the
manifold~$\Gamma$, that is, elements of~$T\Gamma$ are (tangent) vector fields in $\Gamma$.

For technical reasons we relax a little the trace spaces: we define the superspace $G^s((a,\,b),\,\Gamma)$ of $G^s_{\rm av}((a,\,b),\,\Gamma)$
by just omitting the average constraint:
\begin{align}\label{Gi_relax}
G^s((a,\,b),\,\Gamma)&\coloneqq G_\ttt^s((a,\,b),\,\Gamma)\oplus G_{\nnn}^s((a,\,b),\,\Gamma)\nnn,
\end{align}
with  (cf.~\cite[Section~2.1]{Rod15-cocv}) $G_{\nnn}^1((a,\,b),\,\Gamma)\coloneqq L^2((a,\,b),\,H^{\frac{1}{2}}(\Gamma,\,\R))\cap
H^{\frac{3}{4}}((a,\,b),\,H^{-1}(\Gamma,\,\R))$ and 
$G_{\nnn}^2((a,\,b),\,\Gamma)\coloneqq L^2((a,\,b),\,H^{\frac{3}{2}}(\Gamma,\,\R))\cap
H^{1}((a,\,b),\,H^{-\frac{1}{2}}(\Gamma,\,\R))$.

\subsection{The control space}\label{sS:ct-space}
Let us write $L^2(\Omega,\R^3)=H\oplus H^\bot$, where~$H^\bot=\{\nabla\xi\mid\xi\in H^1(\Omega,\,\R)\}$ denotes the
orthogonal complement of~$H$ in~$L^2(\Omega,\R^3)$, and denote by
$\Pi$ the orthogonal projection
\begin{equation}\label{ProjH}
 \Pi\colon L^2(\Omega,\R^3)\to H
\end{equation}
 in
$L^2(\Omega,\R^3)$ onto $H$. For each positive integer $N$, we now
define the $N$-dimensional space $H_N\subset H$ as follows: let
$\{e_i\mid i\in\mathbb N_0\}$ be an orthonormal basis in~$H$
formed by eigenfunctions of the Stokes operator $L$, whose
domain is defined by~\eqref{DL}, and let
$0<\alpha_1\leq\alpha_2\leq\dots$ be the corresponding
eigenvalues, $Le_i=\alpha_ie_i$, then put
\begin{equation}\label{HN}
H_N\coloneqq\linspan\{e_i\mid i\leq N\}\subset \D(L)\subset H
\end{equation}
and denote by $\Pi_N$ the orthogonal projection $\Pi_N\colon H\to H_N$
in $H$ onto $H_N$.

Let $\OO\subseteq\Gamma$ be a connected open subset of the boundary $\Gamma$,
localized on one side of its boundary. We suppose that $\OO$ is a
$C^\infty$-smooth manifold, either boundaryless or with $C^\infty$-smooth boundary $\p\OO$.
Let $\{\pi_i\mid i\in\mathbb N_0\}$ be an orthonormal basis
in~$L^2(\OO,\,\R)$ formed by the eigenfunctions of the
{Laplace--de~Rham} (or Laplace--Beltrami) operator  $\Delta_{\OO}$ on
the smooth manifold $\OO$, under Dirichlet boundary
conditions, $\pi_i(p)=0$ for all $p\in\p\OO$. Analogously let $\{\tau_i\mid i\in\mathbb N_0\}$ be an orthonormal
basis in~$L^2(\OO,\,T\OO)$ formed by the
vector fields that are eigenfunctions of $\Delta_{\OO}$ on
$T\OO$, also under Dirichlet boundary conditions in the case $\p\OO\ne\emptyset$,
$\tau_i(p)=0\in T_p\Gamma$ for all $p\in\p\OO$. It is known that $\pi_i$ and $\tau_i$ $(i\in\mathbb N_0)$ are smooth.
Let $0\leq\beta_1\leq\beta_2\leq\dots$, and
$0\leq\gamma_1\leq\gamma_2\leq\dots$ be the eigenvalues associated
with the systems $\{\pi_i\mid i\in\mathbb N_0\}$ and $\{\tau_i\mid
i\in\mathbb N_0\}$, respectively.

We may write $L^2(\OO,\,\R^3)$ as an orthogonal sum
$
L^2(\OO,\,\R^3)=L^2(\OO,\,\R)\nnn\oplus
L^2(\OO,\,T\OO)
$.
Notice that $\{\pi_i\nnn\mid i\in\N_0\}$ is an
orthonormal basis for $L^2(\OO,\,\R)\nnn=\{f\nnn\mid f\in L^2(\OO,\,\R)\}$, and the system
$\{\pi_i\nnn\mid i\in\N_0\}\cup \{\tau_i\mid i\in\N_0\}$ is an
orthonormal basis in the space $L^2(\OO,\,\R^3)$.

Define, for each $M\in\N_0$, the space
\begin{equation}\label{L2M}
L^2_{M}(\OO,\,\R^3)\coloneqq\linspan
\{\pi_i\nnn,\,\tau_i\mid i\in\N_0,\,i\leq M\}
\end{equation}
and, denote by $P_M^\OO$ the orthogonal projection
$P_M^\OO\colon L^2(\OO,\,\R^3)\to L^2_{M}(\OO,\,\R^3)$ in
$L^2(\OO,\,\R^3)$ onto $L^2_{M}(\OO,\,\R^3)$.

As in~\cite[Section~2.2]{Rod15-cocv}, we suppose that the control region~$\Gamma_c$ and~$\OO$ satisfy
\begin{equation}\label{GammacO}
 \overline{\Gamma_{\rm c}}=\supp(\chi)\mbox{ for some }\chi\in C^2(\Gamma,\,\R);\mbox{ and }
 \Gamma_{\rm c}\subseteq\overline{\Gamma_{\rm c}}\subseteq\OO\subseteq\Gamma.
\end{equation}

Let us define the control space
\begin{equation}\label{ct-space}
\EE_M\coloneqq \left\{\zeta\mid\zeta(t)=\Xi\kappa(t),\mbox{ and }\kappa\in H^1(\R_0,\,\R^{2M})\right\}
\end{equation}
with
\begin{equation}\label{Xi}
\begin{array}{rclll}
 \Xi\colon\R^{2M}&\to& E_M&\coloneqq& \linspan\{\chi\E_{0}^\OO P_{\chi^\bot}^\OO\pi_i\nnn,\,\chi\E_{0}^\OO\tau_i\mid i\in\N_0, i\le M\}\\
 z&\mapsto&\Xi z&\coloneqq& \sum_{i=1}^M\chi\E_{0}^\OO \left(z_iP_{\chi^\bot}^\OO(\pi_i\nnn)
+z_{M+i}\tau_i\right);
\end{array}
\end{equation}
where $\E_{0}^\OO\colon L^2(\OO,\,\R)\to L^2(\Gamma,\,\R)$  stands for the extension by zero outside the subset
$\OO$, and $ P_{\chi^\bot}^\OO\colon L^2(\OO,\,\R^3)\to\{\chi\nnn\rest{\OO}\}^\bot$ stands for the orthogonal projection in $L^2(\OO,\,\R^3)$ onto
$\{\chi\nnn\rest{\OO}\}^\bot=\{f\in L^2(\OO,\,\R^3)\mid (f,\,\chi\nnn)_{L^2(\OO,\,\R^3)}=0\}$:
\[
\begin{cases}
\E_{0}^\OO\xi\rest\OO\coloneqq\xi\\
\E_{0}^\OO\xi\rest{\Gamma\setminus\overline\OO}\coloneqq0
\end{cases}\hspace*{-1em},\quad\mbox{and}\quad P_{\chi^\bot}^\OO v\coloneqq v-\textstyle\frac{(v,\,\chi\nnn)_{L^2(\OO,\,\R^3)}}{\int_\OO\chi^2\,\ed\OO}\chi\nnn\rest\OO.
\]

\begin{remark}
 Notice that the function $\zeta=\Xi z$ satisfies
the zero-average compatibility condition: $\int_\Gamma \zeta(t)\cdot\nnn\,\ed\Gamma
=\sum_{i=1}^M z_i(t)\int_\OO P_{\chi^\bot}^\OO(\pi_i\nnn)\cdot\chi\nnn\rest\OO
\,\ed\OO=0$.
\end{remark}

In particular, observe that the controls, in~$\EE_M$, are supported in ${[0,\,+\infty)}\times\overline{\Gamma_{\rm c}}$
and take their values $\zeta(t)$ in the
finite-dimensional space~$E_M$,
for each $t\in[0,\,+\infty)$.

Let us be given a constant $\lambda>0$ and two (fixed) regular enough functions $h$ and
$\gamma$; in addition we suppose that $\hat u$ is also regular enough and solves, in $\R_0\times\Omega$,
the Navier--Stokes system~\eqref{sys-u-bdry} with $\zeta=0$, and a suitable pressure function $p_u=p_{\hat u}$.
Given $u(0)$ close enough to $\hat u(0)$,
with $(u(0)-\hat u(0))\rest\Gamma\in E_M$, our goal is to find
a (time-dependent) feedback linear controller $v\mapsto \mathbf{K}_{\hat u}^{\lambda,\,t}v\in\R^{2M}$
such that the solution of the
problem~\eqref{sys-u-bdry}, with
$\zeta=(u(0)-\hat u(0))\rest\Gamma+\Xi \int_0^t\mathbf{K}_{\hat u}^{\lambda,\,r}(u-\hat u(r))\,\ed r$,
is defined for all $t\geq0$ and converges
exponentially to~$\hat u$, with rate $\frac{\lambda}{2}$, that is,
\[
|u(t)-\hat u(t)|_{H^1_{\diver}(\Omega,\,\R^3)}^2\le
C\,\ex^{-\lambda t}|u(0)-\hat u(0)|_{H^1_{\diver}(\Omega,\,\R^3)}^2\quad\mbox{for
$t\ge0$},
\]
where $C$ is independent of $u(0)-\hat u(0)$ and time~$t$.
It will be clarified later in Section~\ref{S:nonlinear} what we mean by ``regular enough'',
``close enough'' and ``solution''.

\medskip
Notice that seeking a solution
of~\eqref{sys-u-bdry}--\eqref{ini-u} in the form $u=\hat u+v$, formally we obtain the following
equivalent problem for~$v$:
\begin{equation}\label{sys-vnonlin}
\begin{array}{rclcrcl}
 \p_t v+\BB(\hat u)v
 +\langle v\cdot\nabla\rangle v-\nu\Delta v+\nabla p_v &=&0,&\;&\diver  v & = & 0,\\
 v\rest \Gamma&=&\zeta,&\; &v(0)&=&u_0-\hat u(0),
\end{array}
\end{equation}
with $p_v=p_u-p_{\hat u}$, and $\BB(\hat u)v$ stands for $\langle \hat
u\cdot\nabla\rangle v+\langle v\cdot\nabla\rangle \hat u$. We can see
that it suffices to study the problem of stabilization of system~\eqref{sys-vnonlin} to the zero solution.
We shall start by deriving the (global) stabilization of the Oseen--Stokes system, in $\R_0\times\Omega$:
\begin{equation}\label{sys-v}
\begin{array}{rclcrcl}
 \p_t v+\BB(\hat u)v
 -\nu\Delta v+\nabla p_v&=&0,&\;& \diver v &=&0,\\
 v\rest \Gamma &=&\zeta,&\;& v(0)&=&v_0,
\end{array}
\end{equation}
to the zero solution; from which we shall derive the local result for~\eqref{sys-vnonlin}.

\subsection{Weak and strong solutions, and admissible initial conditions}\label{sS:exist-solut}
We briefly recall some notions and results from~\cite{Rod14-na,Rod15-cocv} concerning the weak and strong solutions for
the Oseen--Stokes system, in a
bounded cylinder $(a,\,b)\times\Omega$, with $a,\,b$ real numbers, $0\le a<b$.
\begin{equation}\label{sys-vg}
\begin{array}{rclcrcl}
 \p_t v+\BB(\hat u)v
 -\nu\Delta v+\nabla p_v+g&=&0,&\;& \diver v &=&0,\\
 v\rest \Gamma &=&\zeta=K\eta,&\;& v(a)&=&v_0.
\end{array}
\end{equation}

Recall the extensions $E_s$, $s\in\{1,\,2\}$, in Section~\ref{sS:funspaces}.
\begin{definition}\label{D:wsol-l}
Given $\hat u\in \WW^{(a,\,b)|\rm wk}$, $v_0\in L^2_{\diver}(\Omega,\,\R^3)$,
$g\in L^2((a,\,b),\,H^{-1}(\Omega,\,\R^3))$, and $\zeta\in G^1_{\rm av}((a,\,b),\,\Gamma)$;
we say that $v$, in the space
$W((a,\,b),\,H^1_{\diver}(\Omega,\,\R^3),\,H^{-1}(\Omega,\,\R^3))$,
is a weak solution for system~\eqref{sys-vg}, if $y\coloneqq v-E_1\zeta\in
L^2((a,\,b),\,V,\,V')$ is a weak solution for
\begin{equation}\label{sys-y-l}
\begin{array}{rclcrcl}
 \p_t y+\BB(\hat u)y-\nu\Delta
y+\nabla p_y+f&=&0,&\quad&\diver y&=&0,\\
 y\rest \Gamma &=&0,&\quad& y(a)&=&y_0
\end{array}
\end{equation}
with $f=g+\p_tE_1\zeta+\BB(\hat u)E_1\zeta-\nu\Delta E_1\zeta$, and
$y_0=v_0-E_1\zeta(a)\in H$. Here weak solution for~\eqref{sys-y-l}
is understood in the classical sense as in~\cite[Chapter~1, Sections~6.1 and~6.4]{Lions69},
\cite[Sections~2.4 and~3.2]{Temam95}, \cite[Chapter~3, Section~3]{Temam01}.
\end{definition}

\begin{definition}\label{D:ssol-l}
Given $\hat u\in \WW^{(a,\,b)|\rm st}$, $v_0\in H^1_{\diver}(\Omega,\,\R^3)$,
$g\in L^2((a,\,b),\,L^2(\Omega,\,\R^3))$, and also $\zeta\in G^2_{\rm av}((a,\,b),\,\Gamma)$;
we say that $v$, in the space
$W((a,\,b),\,H^2_{\diver}(\Omega,\,\R^3),\,L^2(\Omega,\,\R^3))$,
is a strong solution for system~\eqref{sys-vg}, if $y\coloneqq v-E_2\zeta\in
L^2((a,\,b),\,\D(L),\,H)$ is a strong solution for system~\eqref{sys-y-l}
with $f=g+\p_tE_2\zeta+\BB(\hat u)E_2\zeta-\nu\Delta E_2\zeta$, and
$y_0=v_0-E_2\zeta(a)\in V$. Again, strong solution for~\eqref{sys-y-l}
is understood in the classical sense as in~\cite[Section~2.4]{Temam95}.
\end{definition}

In the case our control~$\zeta$ is in the space $\EE_M$ a natural question is: what are
the admissible initial vector fields $v_0$, 
if we want to guarantee the existence of a weak solution?
Notice that, from~\eqref{ct-space}, $\zeta$ takes the form $\zeta=\Xi \kappa$ with $\kappa\in H^1(\R_0,\,\R^{2M})$. It is also
not hard to check that the mapping $\Xi$, in the definition of the control space~\eqref{ct-space},
maps $H^1((a,\,b),\,\R^{2M})$ into $G^2_{\rm av}((a,\,b),\,\Gamma)\subset G^1_{\rm av}((a,\,b),\,\Gamma)$ continuously, that is, our control
is of the form as in~\eqref{sys-vg} with $K=\Xi$.

The set of admissible weak initial conditions for system~\eqref{sys-vg}, with $\zeta\in\EE_M$, is given by
$\aA_{\Xi_1}=H+\HH_{\Xi_1}$, with $\HH_{\Xi_1}\coloneqq E_1\Xi H^1((a,\,b),\,\R^{2M})(a)=\{\gamma(a)\mid\gamma=
E_1\Xi \eta\mbox{ and }\eta\in H^1((a,\,b),\,\R^{2M})\}$.

Similarly, the set of admissible strong initial conditions for system~\eqref{sys-vg}, with $\zeta\in\EE_M$,
is given by $\aA_{\Xi_2}\coloneqq V+\HH_{\Xi_2}$, with
$\HH_{\Xi_2}\coloneqq E_2\Xi H^1((a,\,b),\,\R^{2M})(a)$.

Moreover $\HH_{\Xi_1}$, $\aA_{\Xi_1}$,  $\HH_{\Xi_2}$ and $\aA_{\Xi_2}$ are Hilbert
spaces, with associated range norms
\begin{equation}\label{range-norms}
\begin{array}{rcl}
 |u|_{\HH_{\Xi_i}}&\coloneqq&\inf\left\{|\eta|_{H^1((a,\,b),\,\R^{2M})}\mid 
u=E_i\Xi\eta(a),\;
\eta\in H^1((a,\,b),\,\R^{2M})\right\},\\
|u|_{\aA_{\Xi_1}}&\coloneqq&\inf\left\{|(w,\,z)|_{H\times\HH_{\Xi_1}}\mid u=w+z\mbox{ and }(w,\,z)\in H\times\HH_{\Xi_1}\right\},\\
|u|_{\aA_{\Xi_2}}&\coloneqq&\inf\left\{|(w,\,z)|_{V\times\HH_{\Xi_2}}\mid u=w+z\mbox{ and }(w,\,z)\in V\times\HH_{\Xi_2}\right\}.
\end{array}
\end{equation}

\begin{theorem}\label{T:exist-wsol-v}
If $\hat u\in\WW^{(a,\,b)|\rm wk}$, $g\in L^2((a,\,b),\,H^{-1}(\Omega,\,\R^3))$,
$v_0\in \aA_{\Xi_1}$, and $\eta\in H^1((a,\,b),\,\R^{2M})$, with $v_0-E_1\Xi \eta(a)\in H$, then there is a weak
solution $v$ in $W((a,\,b),\,H^1_{\diver}(\Omega,\,\R^3),\,H^{-1}(\Omega,\,\R^3))$
for system~\eqref{sys-vg}, with $\zeta=\Xi \eta$. Moreover $v$ is unique and depends continuously on
the given data $(v_0,\,g,\,\eta)$:
\begin{align*}
 &\quad |v|_{W((a,\,b),\,H^1_{\diver}(\Omega,\,\R^3),\,H^{-1}(\Omega,\,\R^3))}^2\\
&\leq
\overline C_{\left[|\hat u|_{\WW^{(a,\,b)|\rm wk}}\right]}
\left(|v_0|_{L^2_{\diver}(\Omega,\,\R^3)}^2+|g|_{L^2((a,\,b),\,H^{-1}(\Omega,\,\R^3))}^2+|\eta|_{H^1((a,\,b),\,\R^{2M})}^2\right).
\end{align*}
\end{theorem}

\begin{theorem}\label{T:exist-ssol-v}
If $\hat u\in\WW^{(a,\,b)|\rm st}$, $g\in L^2((a,\,b),\,L^2(\Omega,\,\R^3))$,
 $v_0\in \aA_{\Xi_2}$ and $\eta\in H^1((a,\,b),\,\R^{2M})$,  with $v_0-E_2\Xi \eta(a)\in V$, then there is a strong
solution $v$ in  $W((a,\,b),\,H^2_{\diver}(\Omega,\,\R^3),\,L^2(\Omega,\,\R^3))$
for system~\eqref{sys-vg}, with $\zeta=\Xi \eta$. Moreover $v$ is unique and depends continuously on
the given data $(v_0,\,g,\,\eta)$: 
\begin{align*}
 &\quad |v|_{W((a,\,b),\,H^2_{\diver}(\Omega,\,\R^3),\,L^2(\Omega,\,\R^3))}^2\\
 &\leq \overline C_{\left[|\hat u|_{\WW^{(a,\,b)|\rm st}}\right]}
\left(|v_0|_{H^1_{\diver}(\Omega,\,\R^3)}^2+|g|_{L^2((a,\,b),\,L^2(\Omega,\,\R^3))}^2+|\eta|_{H^1((a,\,b),\,\R^{2M})}^2\right).
\end{align*}
\end{theorem}%

\begin{remark}\label{R:sol-ind-ext}
The weak solution given in Theorem~\ref{T:exist-wsol-v} does not depend on
the extension~$E_1$. Also, the set of admissible weak initial conditions is independent of
$E_1$ (cf.~\cite[Rems.~3.2 and~3.4]{Rod14-na}).
Analogously, the strong solution given in Theorem~\ref{T:exist-ssol-v}, and the set of admissible strong initial
conditions are independent of $E_2$.
\end{remark}

\subsection{Smoothing property}\label{sS:smooth-prop}
The following Lemma will play a key role.
\begin{lemma}\label{L:smooth-prop}
Let us be given $\hat u\in\WW^{(a,\,b)|\rm st}$, $g\in L^2((a,\,b),\,L^2(\Omega,\,\R^3))$, 
$v_0\in \aA_{\Xi_1}$, and also $\eta\in H^1((a,\,b),\,\R^{2M})$, with $v_0-E_1\Xi \eta(a)\in H$;
then for the weak solution $v$ of system~\eqref{sys-vg} with $\zeta=\Xi\eta$,
we have
$(\Bigcdot-a)v\in W((a,\,b),\,H^2_{\diver}(\Omega,\,\R^3)
,\,L^2(\Omega,\,\R^3))$, and
\begin{align*}
&\quad|(\Bigcdot-a)v|_{W((a,\,b),\,H^2_{\diver}(\Omega,\,\R^3)
,\,L^2(\Omega,\,\R^3))}^2\\
&\le\overline C_{\left[|\hat u|_{\WW^{(a,\,b)|\rm st}}\right]}
\left(|v_0|_{L^2_{\diver}(\Omega,\,\R^3)}^2 +|g|_{L^2((a,\,b),\,L^2(\Omega,\,\R^3))}^2
+|\eta|_{H^1((a,\,b),\,\R^{2M})}^2\right).
\end{align*}
\end{lemma}

\begin{proof}
Since $v$ solves~\eqref{sys-vg}, it turns out that also $w=(\Bigcdot-a)v$ does, with different data:
\[
\begin{array}{rclcrcl}
 \p_t w+\BB(\hat u)w
 -\nu\Delta w+\nabla (\Bigcdot-a)p_v+(\Bigcdot-a)g-v&=&0,&\;& \diver w &=&0,\\
 w\rest \Gamma &=&\Xi(\Bigcdot-a)\eta,&\;& w(a)&=&0.
\end{array}
\]
Then, from Theorem~\ref{T:exist-ssol-v}, we can derive that the norm $|w|_{W((a,\,b),\,H^2_{\diver}(\Omega,\,\R^3),\,L^2(\Omega,\,\R^3))}^2$
is bounded by $\overline C_{\left[|\hat u|_{\WW^{(a,\,b)|\rm st}}\right]}
\left(|(\Bigcdot-a)g-v|_{L^2((a,\,b),\,L^2(\Omega,\,\R^3))}^2+|(\Bigcdot-a)\eta|_{H^1((a,\,b),\,\R^{2M})}^2\right)$;
thus Lemma~\ref{L:smooth-prop} follows from $|v|_{L^2((a,\,b),\,L^2(\Omega,\,\R^3))}^2
\le|v|_{W((a,\,b),\,H^1_{\diver}(\Omega,\,\R^3),\,H^{-1}(\Omega,\,\R^3))}^2$ and Theorem~\ref{T:exist-wsol-v}.
\end{proof}

\section{The Oseen--Stokes system: existence of a stabilizing control}\label{S:linear-exist}
Using a controllability result from~\cite{Rod15-cocv}, we shall
construct a control, in the space~$\EE_M$, exponentially stabilizing the linear system~\eqref{sys-v} to the zero solution. 
\begin{definition}
We say that $v_0$ is a weak, respectively strong, admissible initial condition for system~\eqref{sys-v}, with
$\zeta=\Xi\kappa\in\EE_M$, if it is a
weak, respectively strong, admissible initial condition for the same system in $(0,\,1)\times\Omega$
with $\zeta=\Xi\kappa\rest{(0,\,1)}\in\Xi H^1((0,\,1),\,\R^{2M})$.
\end{definition}

Let $\aA_{\Xi_1}$ be the set of admissible initial conditions for controls~$\zeta=\Xi\eta$ in~$\EE_M$, that is,
$\aA_{\Xi_1}=H+E_1\Xi H^1((0,\,1),\,\R^{2M})(0)$.

We introduce the mappings $Q_f^M\colon \R^{2M}\to\R^{2M}$, and $Q_l^M\colon \R^{2M}\to\R^{2M}$ defined by
$Q_f^My\coloneqq (z_1,\,z_2,\,\dots,\,z_{2M})$
with $z_i=y_i$ if $1\le i\le M$, and $z_i=0$ if
$M+1\le i\le 2M$; and $Q_l^M\coloneqq 1_{\R^{2M}}- Q_f^M$.
That is, $Q_f^M$ is the orthogonal projection onto
the first~$M$ coordinates, and~$Q_l^M$ the orthogonal projection onto
the last~$M$ coordinates.

In this Section we prove the following:

\begin{theorem}\label{T:ex.st-ct.lin}
Let us be given $\lambda>0$ and $\hat u\in\WW^{\rm st}$. Then there exists
$M=\overline C_{\left[|\hat u|_{\WW^{\rm st}},\lambda\right]}\geq1$ with the following property: for each $(v_0,\,\kappa^0_\tau)\in \aA_{\Xi_1}\times Q_l^M\R^{2M}$,
there exists a ``control'' vector function
$\kappa^{\hat u,\lambda}=\kappa^{\hat u,\lambda}(v_0,\,\kappa^0_\tau)\in H^1(\R_0,\,\R^{2M})$ such that the weak
solution~$v$ of system~\eqref{sys-v} in $\R_0\times\Omega$, with $\zeta=\Xi\kappa^{\hat u,\lambda}$, satisfies the inequality
\begin{equation*} 
|v(t)|_{L^2_{\diver}(\Omega,\,\R^3)}^2\le\overline C_{\left[|\hat u|_{\WW^{\rm st}},\lambda\right]}\ex^{-\lambda t}
\left(|v_0|_{L^2_{\diver}(\Omega,\,\R^3)}^2+|\kappa^0_\tau|_{\R^{2M}}^2\right), \quad t\geq0.
\end{equation*}
Moreover, for $0\le\hat\lambda<\lambda$, the mapping
$(v_0,\,\kappa^0_\tau)\mapsto \kappa^{\hat u,\lambda}(v_0,\,\kappa^0_\tau)$ is linear and satisfies:
\begin{align*} 
\bigl|\ex^{\frac{\hat\lambda}{2}\Bigcdot}\kappa^{\hat u,\lambda}(v_0,\,\kappa^0_\tau)\bigr|_{H^1(\R_0,\,\R^{2M})}^2
&\le \overline{C}_{\left[|\hat u|_{\WW^{\rm st}},\lambda,\frac{1}{(\lambda-\hat\lambda)}\right]}
\left(|v_0|_{L^2_{\diver}(\Omega,\,\R^3)}^2+|\kappa^0_\tau|_{\R^{2M}}^2\right).
\end{align*}
\end{theorem}

\subsection{Auxiliary results}
Since the trace $\kappa(0)$, at time $t=0$, is well defined
for any given $\kappa\in H^1(\R_0,\,\R^{2M})$, we can easily obtain the
explicit form of the spaces $\aA_{\Xi_1}$ and $\aA_{\Xi_2}$, of all
admissible weak and strong initial conditions, for controls in the space
$\EE_M$ defined in~\eqref{ct-space}.

Writing $v\in\aA_{\Xi_1}$ as
$v=u+(E_1\Xi\kappa)(0)$, where $u\in H$ and $\kappa\in H^1((0,\,1),\,R^{2M})$,
from the fact that~$E_1\Xi\kappa\in W((0,\,1),\,H^1(\Omega,\,\R^3),\,H^{-1}(\Omega,\,\R^3))\subset C([0,\,1],\,L^2(\Omega,\,\R^3)$,
and the continuity of the mapping $u\mapsto u\cdot\nnn$ from $L^2_{\diver}(\Omega,\,\R^3)$ into $H^{-\frac{1}{2}}(\Gamma,\,\R)$,
we obtain that $(v\cdot\nnn)\nnn=\lim_{t\searrow0}((E_1\Xi\kappa)\cdot\nnn)(t)\nnn=\lim_{t\searrow0}((E_1\Xi\kappa(t))\cdot\nnn)\nnn
=\lim_{t\searrow0}((\Xi\kappa(t))\cdot\nnn)\nnn=((\Xi\kappa(0))\cdot\nnn)\nnn$.

It follows that
\begin{equation}\label{Axi1}
\aA_{\Xi_1}=\left\{u\in L^2_{\diver}(\Omega,\,\R^3)\mid
(u\cdot\nnn)\nnn\rest\Gamma
=\Xi Q^M_f z,\mbox{ for some }z\in \R^{2M}\right\}.
\end{equation}

Analogously, we can conclude that the set of strong admissible conditions is given by
\begin{equation}\label{Axi2}
\aA_{\Xi_2}=\left\{u\in H^1_{\diver}(\Omega,\,\R^3)\mid
u\rest\Gamma
=\Xi z,
\mbox{ for some }z\in \R^{2M}\right\}.
\end{equation}

\begin{remark}\label{R:traceun}
Notice that for a given $u\in L^2_{\diver}(\Omega,\,\R^3)$, we may set
$\langle u\cdot\nnn,\,\psi\rangle_{H^{-\frac{1}{2}}(\Gamma,\,\R),\,H^{\frac{1}{2}}(\Gamma,\,\R)}
\coloneqq(u,\,\nabla \RR_\Gamma^1\psi)_{L^2(\Omega,\,\R)}$, where $\RR_\Gamma^1:H^{\frac{1}{2}}(\Gamma,\,\R)\to H^1(\Omega,\,\R)$
is a continuous linear
right inverse of the trace mapping $\Psi\mapsto \Psi\rest\Gamma$; in particular we have
$(\RR_\Gamma^1\psi)\rest\Gamma=\psi$ and the mapping
$u\mapsto u\cdot\nnn\rest\Gamma$ is linear and continuous: $\norm{u\cdot\nnn}{H^{-\frac{1}{2}}(\Gamma,\,\R)}
\le\norm{u}{H}\norm{\RR_\Gamma^1}{\LL\left(H^{\frac{1}{2}}(\Gamma,\,\R)\to H^1(\Omega,\,\R)\right)}$.
See, for example, \cite[Chapter~1, Section~1.3]{Temam01} and~\cite[Chapter~1, Section~8.2]{LioMag72-I}. Actually, writing
$\langle v\cdot\nnn,\,\psi\rangle_{H^{-\frac{1}{2}}(\Gamma,\,\R),\,H^{\frac{1}{2}}(\Gamma,\,\R)}
\coloneqq(u,\,\nabla \RR_\Gamma^1\psi)_{L^2(\Omega,\,\R)}+(\diver u,\,\RR_\Gamma^1\psi)_{L^2(\Omega,\,\R)}$ we can see that the trace
$(v\cdot\nnn)\rest\Gamma\in\LL(L^2_{\rm d}\to H^{-\frac{1}{2}}(\Gamma,\,\R))$ is well defined, with
$L^2_{\rm d}\coloneqq\{z\in L^2(\Omega,\,\R^3)\mid \diver z\in L^2(\Omega,\,\R)\}$ endowed with the norm $\norm{v}{L^2_{\rm d}}\coloneqq
(\norm{z}{L^2(\Omega,\,\R^3)}^2+\norm{\diver z}{L^2(\Omega,\,\R)}^2)^\frac{1}{2}$. Notice that~$L^2_{\diver}(\Omega,\,\R^3)$ is a closed subspace of~$L^2_{\rm d}$.
\end{remark}

\begin{corollary}\label{C:Ad-closed}
The space of admissible weak initial conditions~$\aA_{\Xi_1}$ is a closed subset of $L^2_{\diver}(\Omega,\,\R^3)$, and
the space of admissible strong initial conditions $\aA_{\Xi_2}$  is a closed subset of $H^1_{\diver}(\Omega,\,\R^3)$.
\end{corollary}

Recalling the orthogonal projection $\Pi$, in~\eqref{ProjH}, we have:
\begin{lemma}\label{L:eq.normsL2div}
The norms $|u|_{L^2_{\diver}(\Omega,\,\R^3)}$ and $\Bigl(|\Pi u|_H^2
+|u\cdot\nnn|_{H^{-\frac{1}{2}}_{\rm av}(\Gamma,\,\R)}^2\Bigr)^{\frac{1}{2}}$,
are equivalent in $L^2_{\diver}(\Omega,\,\R^3)$.
\end{lemma}
\begin{proof}
Writing $L^2(\Omega,\,\R^3)=H\oplus H^\bot$, each $v\in L^2(\Omega,\,\R^3)$ can be rewritten as
\begin{equation}\label{Pnabla}
 v=\Pi v+\nabla P_\nabla v,\quad\mbox{with }\int_\Omega P_\nabla v\,\ed\Omega=0;
\end{equation}
in this way the mapping $v\mapsto P_\nabla v\in H^1(\Omega,\,\R)$ is
well defined. Further, if $\diver v\in L^2(\Omega,\,\R)$ we can see that $P_\nabla v$ solves
the system
$\Delta P_\nabla v=\diver v,\,\nabla P_\nabla v\cdot\nnn=v\cdot\nnn$.
It follows that if $v\in L^2_{\diver}(\Omega,\,\R^3)$, then $|\nabla P_\nabla v|_{L^2(\Omega,\,\R^3)}^2
=\langle v\cdot\nnn,\,P_\nabla v\rangle_{H^{-\frac{1}{2}}(\Gamma,\,\R),\,H^{\frac{1}{2}}(\Gamma,\,\R)}$; since
$P_\nabla v$ is zero averaged in $\Omega$, there is a constant $C>0$ such that
$|\nabla P_\nabla v|_{L^2(\Omega,\,\R^3)}^2
\le C|v\cdot\nnn|_{H^{-\frac{1}{2}}(\Gamma,\,\R)}|\nabla P_\nabla v|_{L^2(\Omega,\,\R^3)}$; so
$|\nabla P_\nabla v|_{L^2(\Omega,\,\R^3)}
\le C|v\cdot\nnn|_{H^{-\frac{1}{2}}(\Gamma,\,\R)}\le C_1|v|_{L^2_{\diver}(\Omega,\,\R^3)}$.
\end{proof}

From now, for convenience, we suppose $\aA_{\Xi_1}$ and $\aA_{\Xi_2}$ endowed with the norm inherited from $L^2_{\diver}(\Omega,\,\R^3)$
and from $H^1_{\diver}(\Omega,\,\R^3)$, respectively; from Corollary~\ref{C:Ad-closed} the spaces
$\aA_{\Xi_1}$ and $\aA_{\Xi_2}$ are Hilbert spaces.

\begin{remark}
In Section~\ref{sS:exist-solut}, we have considered the spaces of admissible conditions endowed
with a suitable range norm also making them Hilbert spaces. Changing the norms now to those inherited from
$L^2_{\diver}(\Omega,\,\R^3)$
and $H^1_{\diver}(\Omega,\,\R^3)$ will not cause any trouble concerning continuity properties. Indeed, we have that
$\aA_{\Xi_1}$ and $\aA_{\Xi_2}$ endowed with the range norm are continuously embedded in $L^2_{\diver}(\Omega,\,\R^3)$ and
$H^1_{\diver}(\Omega,\,\R^3)$, respectively, (cf.~\cite[Section~2.4]{Rod15-cocv}). Thus, from
the completeness of both norms they are necessary equivalent (cf.~\cite[Corollary~2.8]{Brezis11}).
\end{remark}

Next, we define the space
\begin{equation}\label{N.sp}
\NN\coloneqq \left\{z\in\R^{2M}\biggl| \Xi z=\chi\E_0^\OO P_{\chi^\bot}^\OO
\sum_{i=1}^M (z_i\pi_i\nnn +z_{M+i}\tau_i)=0\right\}.
\end{equation}

Let $P_{\NN}\colon \R^{2M}\to\NN$ stand for the orthogonal projection in $\R^{2M}$ onto $\NN$. Denoting
the orthogonal subspace $\NN^\bot$ to $\NN$, we also denote $P_{\NN^\bot}=I_{\R^{2M}}-P_{\NN}\colon \R^{2M}\to\NN^\bot$ the orthogonal 
projection in $\R^{2M}$ onto $\NN^\bot$.
Recall also the projections $Q^M_f\colon \R^{2M}\mapsto \R^{2M}$ and $Q^M_l\colon \R^{2M}\mapsto \R^{2M}$, 
onto the first~$M$ coordinates and onto the last~$M$ coordinates, respectively.
We have the following property whose
proof is given in the Appendix, Section~\ref{sS:comut}.
\begin{equation}\label{comut}
PQ=QP\mbox{ for all }(P,\,Q)\in\{P_{\NN},\,P_{\NN^\bot}\}\times\{Q^M_f,\,Q^M_l\}.
\end{equation}
Notice that, for any $u\in\aA_{\Xi_1}$ there exists at least one $z\in\R^{2M}$ such that $(u\cdot\nnn)\nnn\rest\Gamma
=\Xi Q^M_fz=\chi\E_{0}^\OO \sum_{i=1}^M\left(z_iP_{\chi^\bot}^\OO\pi_i\nnn\right)$; it follows, using~\eqref{comut}, that the mapping
\begin{equation}\label{zun}
 u\mapsto z^{u\cdot\nnn}\coloneqq P_{\NN^\bot}Q^M_f z=Q^M_f P_{\NN^\bot}z
\end{equation}
is continuous from $\aA_{\Xi_1}$ onto $P_{\NN^\bot}Q^M_f\R^{2M}$
and $(u\cdot\nnn)\nnn\rest\Gamma=\Xi z^{u\cdot\nnn}$. 
Indeed, to check that the mapping is well defined, we set another vector $w\in\R^{2M}$ such that $(u\cdot\nnn)\nnn\rest\Gamma
=\Xi Q^M_fw$, then necessarily $\Xi Q^M_f(w-z)=0$ which means that $P_{\NN^\bot}Q^M_f(w-z)=0$. On the other hand
the continuity of the mapping $u\mapsto u\cdot\nnn$ and the fact
that both $|u\cdot\nnn|_{H^{-\frac{1}{2}}_{\rm av}(\Gamma,\,\R)}$ and $|z^{u\cdot\nnn}|_{\R^{2M}}$ are norms in the finite
dimensional space $\Xi Q^M_f\R^{2M}=\{u\cdot\nnn\mid u\in\aA_{\Xi_1}\}$ (and, so necessarily equivalent) give us
\begin{equation}\label{zun.cont}
 |z^{u\cdot\nnn}|_{\R^{2M}}\le C|u|_{L^2_{\diver}(\Omega,\,\R^3)}.
\end{equation}

\begin{lemma}\label{L:eq-normsAd1}
The norms $|u|_{L^2_{\diver}(\Omega,\,\R^3)}$ and $\left(|\Pi u|_H^2
+|z^{u\cdot\nnn}|_{\NN^\perp}^2\right)^{\frac{1}{2}}$ are equivalent in $\aA_{\Xi_1}$;
further they make $\aA_{\Xi_1}$ a Hilbert space.
\end{lemma}

\begin{proof}
By Corollary~\ref{C:Ad-closed},  $\aA_{\Xi_1}$ endowed with the norm inherited from $L^2_{\diver}(\Omega,\,\R^3)$
is a Hilbert space. The equivalence follows from Lemma~\ref{L:eq.normsL2div} and from the fact that
$|u\cdot\nnn|_{H^{-\frac{1}{2}}(\Gamma,\,\R)}$ and $|z^{u\cdot\nnn}|_{\NN^\perp}$ are equivalent norms in
$\Xi Q^M_f\R^{2M}=\Xi Q^M_f\NN^\perp$.
\end{proof}

\begin{remark}\label{R:Nn}
Notice that the space $\NN$ defined in~\eqref{N.sp} is not necessarily trivial, that is, it may contain nonzero vectors.
See the Example in Section~\ref{sS:ex-nottriv} in the Appendix.
\end{remark}

\begin{corollary}\label{C:L2=HH12}
A element $u\in L^2_{\diver}(\Omega,\,\R^3)$ can be defined by~$(\Pi u,\,u\cdot\nnn)\in H\times H^{-\frac{1}{2}}_{\rm av}(\Gamma,\,\R)$.
\end{corollary}
\begin{proof}
Given~$(h^1,\,h^2)\in H\times H^{-\frac{1}{2}}_{\rm av}(\Gamma,\,\R)$, and two elements~$u^1$ and~$u^2$ in~$L^2_{\diver}(\Omega,\,\R^3)$ such that
$\Pi u^1=h^1=\Pi u^2\quad\mbox{and}\quad u^1\cdot\nnn=h^2=u^2\cdot\nnn,$
then for~$u\coloneqq u^1-u^2$ we find $u=(1-\Pi) u=\nabla P_\nabla u$ and $u\cdot\nnn=0$, which implies $\Delta P_\nabla u=0$ and
$\nabla P_\nabla u\cdot\nnn=0$. Therefore $u=\nabla P_\nabla u=0$, that is, $u^1=u^2$.
\end{proof}

\begin{corollary}\label{C:L2=Hzun}
A element $u\in \aA_{\Xi_1}$ can be defined by~$(\Pi u,\,z^{u\cdot\nnn})\in H\times Q^M_f\NN^\perp$.
\end{corollary}
\begin{proof}
The mapping~$\Xi\colon\NN^\perp\to L^2(\Gamma,\,\R^3)$ is injective and~$\Xi z^{u\cdot\nnn}=(u\cdot\nnn)\nnn\rest\Gamma$.
\end{proof}

Now, let us be given four nonnegative constants
\[
 0\le a<b,\quad0<\varepsilon,\quad0<\delta,
\]
and two
functions~$\varphi,\,\tilde\varphi\in C^1([a,\,b],\,\R)$ such that
$\supp(\varphi)\ne\emptyset$, $\tilde\varphi(t)\geq\varepsilon$ for all $t\in\supp(\varphi)$, and $\tilde\varphi(t)=0$
for $t\in[a,\,a+\delta]\cup[b-\delta,\,b]$.
Further let $\vartheta\in C^2(\Gamma,\,\R)$ be a function such that
$\supp(\vartheta)\subseteq\overline\OO$ and $\vartheta(x)\geq\varepsilon$ for all $x\in\Gamma_{\rm c}$; see~\eqref{GammacO}.

Given a Hilbert space $X$, we define the orthogonal projection $P_M^t$ in $L^2((a,\,b),\,X)$:
\[
P_M^tf\coloneqq \sum_{n=1}^{M}\left(\int_a^bf(\tau)\sigma_n(\tau)\,\ed\tau\right)\sigma_n
\]
where the $\sigma_n$, $n\in\N_0$, are the eigenfunctions of the Dirichlet Laplacian $\Delta_t\coloneqq \p_t\p_t$ in $(a,\,b)$:
$\{\sigma_n(t)\coloneqq (\frac{2}{b-a})^{\frac{1}{2}}\sin(n\pi(\frac{t-a}{b-a}))\mid n\in\N_0\}$, and
$\{\lambda_n=-(\frac{n\pi}{b-a})^2\mid n\in\N_0\}$; $\Delta_t\sigma_n=\lambda_n\sigma_n$.

Recalling the notations from Section~\ref{sS:ct-space} and inspired by an Example in~\cite[Section~5]{Rod15-cocv}, we consider the auxiliary
``control'' space
\begin{align*}
\GG_M&\coloneqq \varphi\chi\E_{0}^\OO P_{\chi^\bot}^\OO P_M^\OO P_M^t
\tilde\varphi\vartheta G^2((a,\,b),\,\Gamma)\rest{\OO}\\
&\coloneqq \{\zeta\mid
\zeta=\varphi\chi\E_{0}^\OO P_{\chi^\bot}^\OO P_M^\OO P_M^t
(\tilde\varphi\vartheta\eta\rest{\OO})\mbox{ and }\eta\in G^2((a,\,b),\,\Gamma)\}.\notag
\end{align*}
and consider the operator
\begin{equation}\label{KO-oper}
 \eta\mapsto K^\OO_t\eta\coloneqq \varphi\chi\E_{0}^\OO P_{\chi^\bot}^\OO P_M^\OO P_M^t
(\tilde\varphi\vartheta\eta\rest{\OO}).
\end{equation}

Recalling
the space $H_N$ and the orthogonal projection $\Pi_N\colon H\to H_N$, see Section~\ref{sS:ct-space},
from~\cite[Section~5]{Rod15-cocv} we know the following controllability result: 
\begin{theorem}\label{T:ct-PiNvb=0.2}
Let us be given $\hat u\in\WW^{(a,\,b)|\rm st}$ and $N\in\mathbb N$, then there exists an integer
$M=\overline C_{\left[N,\,|\hat u|_{\WW^{(a,\,b)|\rm st}}\right]}\in\N_0$ with the following property:
for every $v_0\in H$, we can find $\eta=\eta(v_0)\in G^2((a,\,b),\,\Gamma)$, depending linearly on~$v_0$,
such that the control
$\zeta=K^\OO_t\eta=\varphi\chi\E_{0}^\OO P_{\chi^\bot}^\OO P_M^\OO P_M^t(\tilde\varphi\vartheta\eta\rest\OO)$
drives the system~\eqref{sys-v} to
a vector $v(b)\in V$ such that $\Pi_Nv(b)=0$. 
Moreover, there exists a constant
$\overline C_{\left[|\hat u|_{\WW^{(a,\,b)|\rm st}}\right]}$, depending
on~$|\hat u|_{\WW^{(a,\,b)|\rm st}}$, $\varphi$, $\tilde\varphi$, and~$b-a$, but not on the pair~$(N,\,v_0)$,
such that
\begin{equation} \label{ctPn0-bdd}
|\eta(v_0)|_{G^2((a,\,b),\,\Gamma)}^2\leq \overline C_{\left[|\hat u|_{\WW^{(a,\,b)|\rm st}}\right]}|v_0|_H^2.
\end{equation}
\end{theorem}

\subsection{Proof of Theorem \ref{T:ex.st-ct.lin}}
Let us fix a sufficiently large $N\geq1$ and let~$M$ be the integer given in Theorem~\ref{T:ct-PiNvb=0.2}. We organize the proof
into~\ref{stctest} main steps.
\begin{enumerate}[noitemsep,topsep=5pt,parsep=5pt,partopsep=0pt,leftmargin=0em]%
\renewcommand{\theenumi}{{\sf\arabic{enumi}}} 
 \renewcommand{\labelenumi}{} 
 \item \textcircled{\bf s}~Step~\theenumi:\label{st01} {\em driving the system, from $v(0)=v_0$ at time
$t=0$, to a vector $v(1)=v_1\in V$ at time $t=1$.}
Let $z^{v_0\cdot\nnn}\in P_{\NN^\bot}Q^M_f\R^{2M}$ be the vector defined as in~\eqref{zun}, and let $\phi\in C^1([0,\,1],\,\R)$
be a function taking the value $1$ in a neighborhood ${[0,\,\delta)}$ of $t=0$,
and the value $0$ in a neighborhood ${(1-\delta,\,1]}$ of $t=1$, with $\delta<\frac{1}{2}$. Let us also be given $\kappa^0_\tau\in Q^M_l\R^{2M}$.
Then, the function
$\kappa_\phi(v_0,\,\kappa^0_\tau)=\phi (z^{v_0\cdot\nnn}+\kappa^0_\tau)
=\phi(z^{v_0\cdot\nnn}_1,\,z^{v_0\cdot\nnn}_2,\,\dots,\,z^{v_0\cdot\nnn}_M,\,\kappa^0_{\tau,\,M+1},\,
\kappa^0_{\tau,\,M+2},\,\dots,\,\kappa^0_{\tau,\,2M})$ is in $C^1([0,\,1],\,\R^{2M})$.

Next we consider the system~\eqref{sys-v} in $(0,\,1)\times\Omega$, and the control $\zeta=\Xi\kappa_\phi$:
\[
\begin{array}{rclcrcl}
 \p_t v+\BB(\hat u)v
 -\nu\Delta v+\nabla p_v&=&0,&\;& \diver v &=&0,\\
 v\rest \Gamma &=&\Xi\kappa_\phi,&\;& v(0)&=&v_0;
\end{array}
\]
since $\left((v_0-E_1\Xi\kappa_\phi)\cdot\nnn\right)\nnn\rest\Gamma
=(v_0\cdot\nnn)\nnn\rest\Gamma-\Xi z^{v_0\cdot\nnn}=0$, we have 
$v_0-E_1\Xi\kappa_\phi\in H$. By
Theorem~\ref{T:exist-wsol-v}
there exists a weak
solution~$v$ satisfying the estimate $|v|_{W((0,\,1),\,H^1_{\diver}(\Omega,\,\R^3),\,H^{-1}(\Omega,\,\R^3))}^2
\le\overline C_{\left[M,\,|\hat u|_{\WW^{(0,\,1)|\rm wk}}\right]}
\left(|v_0|_{L^2_{\diver}(\Omega,\,\R^3)}^2+|\kappa_\phi|_{H^1((0,\,1),\,\R^{2M})}^2\right)$,
from which we can derive
\begin{equation}\label{stab-est01}
|v|_{W((0,\,1),\,H^1_{\diver}(\Omega,\,\R^3),\,H^{-1}(\Omega,\,\R^3))}^2
\le\overline C_{\left[M,\,|\hat u|_{\WW^{(0,\,1)|\rm wk}}\right]}
\left(|v_0|_{L^2_{\diver}(\Omega,\,\R^3)}^2+|\kappa^0_\tau|_{\R^{2M}}^2\right).
\end{equation}
Further, $v\rest\Gamma$ vanishes in a neighborhood of $t=1$, which implies that $v_1\coloneqq v(1)\in H$. Furthermore,
from Lemma~\ref{L:smooth-prop}, since
$\hat u\rest{(0,\,1)}\in \WW^{(0,\,1)|\rm st}$, we actually have $v(1)\in V\subset H$. 

\item \textcircled{\bf s}~Step~\theenumi:\label{stn1} {\em driving the system from $v(n)=v_n\in V$ at time $t=n\in\N_0$ to a
vector $v(n+1)=v_{n+1}\in V$ at time $t=n+1$, with
$|v_{n+1}|_H^2\le \ex^{-\lambda}|v_n|_H^2$.} 
Now we consider the system~\eqref{sys-v} in $(n,\,n+1)\times\Omega$, and the
control $\zeta=K^\OO_t\bar\eta^{\hat u,n}(v_n)$, where $\bar\eta^{\hat u,n}(v_n)=\eta(v_n)$ is
given in Theorem~\ref{T:ct-PiNvb=0.2}, with $(a,\,b)=(n,\,n+1)$:
\begin{equation*}
\begin{array}{rclcrcl}
 \p_t v+\BB(\hat u)v
 -\nu\Delta v+\nabla p_v&=&0,&\;& \diver v &=&0,\\
 v\rest \Gamma &=&K^\OO_t\bar\eta^{\hat u,n}(v_n),&\;& v(n)&=&v_n.
\end{array}
\end{equation*}
Now, we observe that $\sin(m\pi(t-n))=(-1)^{mn}\sin(m\pi t)$ and
$\{\sqrt{2}\sin(m\pi t)\mid m\in\N_0\}\subset H^1_0((n,\,n+1),\,\R)$
is an orthonormal basis in $L^2((n,\,n+1),\,\R)$. For time $t\in(n,\,n+1)$, the control
$K^\OO_t\bar\eta^{\hat u,n}(v_n)$ can be rewritten as
\[
\varphi\chi\E_{0}^\OO \sum_{i=1}^M\sum_{m=1}^M
\left(\eta_{i,m}^\nnn \sigma_m(t) P_{\chi^\bot}^\OO\pi_i\nnn+\eta_{i,m}^\ttt\sigma_m(t)\tau_i\right)
\]
where $\eta_{i,m}^\nnn$ and $\eta_{i,m}^\ttt$ are constants, and $\sigma_m(t)=\sqrt{2}\sin(m\pi t)$. Define 
$K^{[n]}\colon \MM_{2M\times M}\to H^1((n,\,n+1),\,H^2(\Omega,\,\R^3))$, mapping a matrix $A=[A_{j,m}]$, with real
entries $A_{j,m}\in\R$ for
$j=1,\,2,\,\dots,\, 2M$ and $m=1,\,2,\,\dots,\,M$, to
\[
K^{[n]}A\coloneqq \varphi\chi\E_{0}^\OO \sum_{i=1}^M\sum_{m=1}^M
\left(A_{i,m}\sigma_m(t) P_{\chi^\bot}^\OO\pi_i\nnn+A_{M+i,m}\sigma_m(t)\tau_i\right),
\]
and consider the space of matrices
\begin{equation}\label{N.spx}
\NN_\times\coloneqq \left\{A\in \MM_{2M\times M}\biggl| K^{[n]} A=0\right\};
\end{equation}
we suppose $\MM_{2M\times M}\sim \R^{2M^2}$  endowed with the scalar product
$(A,\,B)_\MM\coloneqq \sum_{i=1}^M\sum_{m=1}^M (A_{i,\,m}B_{i,\,m}+A_{M+i,\,m}B_{M+i,\,m})$. It follows that 
$A\mapsto |P_{\NN_\times^\bot} A|_{\MM}$ defines a norm in the range $K^\OO_t G^2((n,\,n+1),\,\Gamma)= K^{[n]}\MM_{2M\times M}$.
Since $K^\OO_t$ is linear and continuous from $\eta\in G^2((a,\,b),\,\Gamma)$ into $G^2_{\rm av}((a,\,b),\,\Gamma)$
(cf.~\cite[Proposition~5.1]{Rod15-cocv}), from the finite dimensionality of $K^\OO_t G^2((n,\,n+1),\,\Gamma)$, it follows that
$|P_{\NN_\times^\bot} A|_{\MM}\le C|\eta|_{G^2((n,\,n+1),\,\Gamma)}$, where $A$ is any matrix satisfying
$K^{[n]} A=K^\OO_t\eta$. Moreover the mapping $\eta\mapsto A^\eta\coloneqq P_{\NN_\times^\bot} A$ from $G^2((n,\,n+1),\,\Gamma)$ into
$\NN_\times^\bot$ is well defined, that is, $A^\eta$
is the unique element in $\NN_\times^\bot$ that solves $K^{[n]}A^\eta\coloneqq K^\OO_t\eta$.

As a consequence of~\eqref{ctPn0-bdd} we have that $|A^{\bar\eta^{\hat u,n}}|_\MM^2\le
\overline C_{\left[|\hat u|_{\WW^{\rm st}}\right]}|v_n|_H^2$. Defining, for each $1\le j\le 2M$, the functions
$\bar\kappa_j^{\hat u}(n,\,t)\coloneqq \sum_{m=1}^M A^{\bar\eta^{\hat u,n}}_{j,\,m}\sigma_m(n,\,t)$, we find that
$\bar\kappa^{\hat u}(n,\,\Bigcdot)=\bar\kappa^{\hat u}(n,\,\Bigcdot)(v_n)
\coloneqq (\bar\kappa_1^{\hat u}(n,\,\Bigcdot),\,\bar\kappa_2^{\hat u}(n,\,\Bigcdot),\,
\dots,\,\bar\kappa_{2M}^{\hat u}(n,\,\Bigcdot))$ is in
$C^1([n,\,n+1],\,\R^{2M})$; furthermore
\begin{equation}\label{ct-bddC1}
K^\OO_t\bar\eta^{\hat u,\,n}=\Xi\bar\kappa^{\hat u}(n,\,\Bigcdot)\quad \mbox{and}\quad
|\bar\kappa^{\hat u}(n,\,\Bigcdot)(v_n)|_{C^1([n,\,n+1],\,\R^{2M})}^2\le
\overline C_{\left[|\hat u|_{\WW^{\rm st}}\right]}|v_n|_H^2.
\end{equation}

Now, Lemma~\ref{L:smooth-prop} and the continuity of the mapping $v\mapsto v(n+1)$, from
the space $W((n,\,n+1),\,H^2_{\diver}(\Omega,\,\R^3),\,L^2(\Omega,\,\R^3))$
into $H^1_{\diver}(\Omega,\,\R^3)$, imply that
\[
|v(n+1)|_{H^1_{\diver}(\Omega,\,\R^3)}^2\le\overline C_{\left[|\hat u|_{\WW^{(n,\,n+1)|\rm st}}\right]}
\left(|v_n|_{L^2_{\diver}(\Omega,\,\R^3)}^2 +|\bar\kappa^{\hat u}(n,\,\Bigcdot)(v_n)|_{H^1((n,\,n+1),\,\R^{2M})}^2\right);
\]
on the other hand from the definition of $K^\OO_t$, in~\eqref{KO-oper}, we have that $v\rest \Gamma$ vanishes
in a neighborhood of $t=n+1$ and, since $\bar\kappa^{\hat u}(n,\,\Bigcdot)$ satisfies~\eqref{ct-bddC1}, we have that
$|v(n+1)|_{V}^2\le\overline C_{\left[|\hat u|_{\WW^{(n,\,n+1)|\rm st}}\right]}
|v_n|_{H}^2$. Now we use the fact that $\Pi_Nv(n+1)=0$ to obtain
$\alpha_N|v(n+1)|_{H}^2\le|v(n+1)|_{V}^2$, where $\alpha_N$ is the $N$th eigenvalue of the Stokes operator
(see Section~\ref{sS:ct-space}), which allow us to write
$|v(n+1)|_{H}^2\le\alpha_N^{-1}\overline C_{\left[|\hat u|_{\WW^{(n,\,n+1)|\rm st}}\right]}
|v_n|_{H}^2$; then, for big enough $N$, such that $\alpha_N
\geq\ex^\lambda\overline C_{\left[|\hat u|_{\WW^{(n,\,n+1)|\rm st}}\right]}$, we have
that $v_{n+1}\coloneqq v(n+1)$ satisfies
\begin{equation}\label{stab.n}
|v_{n+1}|_{H}^2\le\ex^{-\lambda}|v_n|_{H}^2.
\end{equation}

\item \textcircled{\bf s}~Step~\theenumi:\label{stconc} {\em concatenation; a stabilizing control.}
First of all, we fix the functions $\varphi$
and $\tilde\varphi$ (appearing in Theorem~\ref{T:ct-PiNvb=0.2}) for the interval $(a,\,b)=(1,\,2)$ and
then set $\varphi(t)\coloneqq\varphi(t-n+1)$
and $\tilde\varphi(t)\coloneqq\tilde\varphi(t-n+1)$ for $t\in(n,\,n+1)$. Since $\hat u\in\WW^{\rm st}$ (cf.~\eqref{Wspaces}),
the integer~$N$ in Step~\ref{stn1} may be taken the same in each interval
$(n,\,n+1)$, $n\in\N_0$; then,
the same holds for the integer~$M$ in Theorem~\ref{T:ct-PiNvb=0.2}, with $(a,\,b)=(n,\,n+1)$. 
Now we show that, given $(v_0,\,\kappa^0_\tau)\in\aA_{\Xi_1}\times Q^M_l\R^{2M}$, the control
\begin{equation}\label{stab.contr}
\zeta_{\hat u}^\lambda=\zeta_{\hat u}^\lambda(v_0,\,\kappa^0_\tau)
\coloneqq \left\{\begin{array}{ll}
\Xi\kappa_\phi(v_0,\,\kappa^0_\tau),&\mbox{if }t\in[0,\,1);\\
\Xi\bar\kappa^{\hat u}(n,\,\Bigcdot)(v_n),&\mbox{if }t\in[n,\,n+1),\mbox{ with }n\in\N_0;
\end{array}\right.
\end{equation}
stabilizes system~\eqref{sys-v} to the zero solution.
Here, for $n\in\N_0$, $v_n\coloneqq v(n)$ where $v$ is the solution
of the system~\eqref{sys-v} in $(0,\,n)\times\Omega$ with control $\zeta_{\hat u}^\lambda\rest{(0,\,n)\times\Gamma}$.

From~\eqref{stab-est01}, and the inequality $1\le\ex^\lambda\ex^{-\lambda t}$, for $t\in[0,\,1]$, we obtain
\begin{equation}\label{est.exp01}
|v(t)|_{L^2_{\diver}(\Omega,\,\R^3)}^2
\le\overline C_{\left[\lambda,\,|\hat u|_{\WW^{\rm st}}\right]}\ex^{-\lambda t}
\left(|v_0|_{L^2_{\diver}(\Omega,\,\R^3)}^2+|\kappa^0_\tau|_{\R^{2M}}^2\right),\mbox{ for all }t\in[0,\,1];
\end{equation}
on the other hand for $t\geq 1$ we also have,
$|v(t)|_{L^2_{\diver}(\Omega,\,\R^3)}^2\le \overline C_{\left[|\hat u|_{\WW^{\rm st}}\right]}
|v(\lfloor t\rfloor)|_{H}^2$, where~$\lfloor t\rfloor\ge1$ denotes the biggest integer that is smaller than~$t$, defined by
\begin{equation}\label{floor}
\lfloor r\rfloor\in\Z\quad\mbox{and}\quad\lfloor r\rfloor+1>r\geq\lfloor r\rfloor,\quad\mbox{for all } r\in\R.  
\end{equation}
Thus we obtain
$|v(t)|_{L^2_{\diver}(\Omega,\,\R^3)}^2\le \overline C_{\left[|\hat u|_{\WW^{\rm st}}\right]}
\ex^{-\lambda(\lfloor t\rfloor-1)}|v(1)|_{H}^2 =\overline C_{\left[|\hat u|_{\WW^{\rm st}}\right]}
\ex^{\lambda(t-\lfloor t\rfloor)}\ex^{-\lambda (t-1)}|v(1)|_{H}^2
\le \overline C_{\left[|\hat u|_{\WW^{\rm st}}\right]}
\ex^{\lambda}\ex^{-\lambda (t-1)}|v(1)|_{H}^2$. Using~\eqref{est.exp01} (with $t=1$), we can conclude that
\begin{equation}\label{est.exp}
|v(t)|_{L^2_{\diver}(\Omega,\,\R^3)}^2
\le\overline C_{\left[\lambda,\,|\hat u|_{\WW^{\rm st}}\right]}\ex^{-\lambda t}
\left(|v_0|_{L^2_{\diver}(\Omega,\,\R^3)}^2+|\kappa^0_\tau|_{\R^{2M}}^2\right),\mbox{ for all }t\in[0,\,+\infty).
\end{equation}

\item \textcircled{\bf s}~Step~\theenumi:\label{stctest} {\em control estimate.}
Defining the mapping
\begin{equation}
\kappa^{\hat u,\lambda}=\kappa^{\hat u,\lambda}(v_0,\,\kappa^0_\tau)
\coloneqq \left\{\begin{array}{ll}
\kappa_\phi(v_0,\,\kappa^0_\tau),&\mbox{if }t\in{[0,\,1)};\\
\bar\kappa^{\hat u}(n,\,\Bigcdot)(v_n),&\mbox{if }t\in{[n,\,n+1)},\mbox{ with }n\in\N_0;
\end{array}\right.
\end{equation}
we see that the control in~\eqref{stab.contr} can be rewritten as
\begin{equation}\label{stab.ctXi}
\zeta_{\hat u}^\lambda=\Xi\kappa^{\hat u,\lambda}=\Xi\kappa^{\hat u,\lambda}(v_0,\,\kappa^0_\tau).
\end{equation}
Notice that for any given positive integer $n\in\N_0$, the control function $\kappa^{\hat u,\lambda}$
vanishes in a neighborhood of~$n$. Indeed, from Step~\ref{st01}, $\kappa^{\hat u,\lambda}$ vanishes in $[1-\delta,\,1]$, and from
Step~\ref{stn1},
it also vanishes in $[n,\,n+\delta]\cup[n+1-\delta,\,n+1]$,
because $\supp(\varphi\rest{[n,\,n+1]})\subset\supp(\tilde\varphi\rest{[n,\,n+1]})
\subseteq[n+\delta,\,n+1-\delta]$. 
Let us be given $\hat\lambda\in{[0,\,\lambda)}$; then the mapping
$v_0\mapsto \ex^{\frac{\hat\lambda}{2} t}\kappa^{\hat u,\lambda}(v_0,\,\kappa^0_\tau)$ is linear and continuous, from
$\aA_{\Xi_1}\times Q^M_l\R^{2M}$ into $H^1(\R_0,\,\R^{2M})$. Indeed, the linearity follows essentially
from the linearity of system~\eqref{sys-v} and from the linearity of the mappings
$(v_0,\,\kappa^0_\tau)\mapsto\kappa_\phi(v_0,\,\kappa^0_\tau)$ and $v_n\mapsto\bar\kappa^{\hat u}(n,\,\Bigcdot)(v_n)$.
The boundedness follows by direct computations: we find
\begin{align*}
\left|\ex^{\frac{\hat\lambda}{2} \Bigcdot}\kappa^{\hat u,\lambda}\right|_{H^1(\R_0,\,\R^{2M})}^2
&=\left|\ex^{\frac{\hat\lambda}{2}\Bigcdot}
\kappa_\phi(v_0,\,\kappa^0_\tau)\right|_{H^1((0,\,1),\,\R^{2M})}^2+\sum_{n\in\mathbb N_0} \left|\ex^{\frac{\hat\lambda}{2}\Bigcdot}
\bar\kappa^{\hat u}(n,\,\Bigcdot)(v_n)\right|_{H^1((n,\,n+1),\,\R^{2M})}^2\\
&\le \overline C_{\left[\hat\lambda\right]}\left(
\left|\kappa_\phi(v_0,\,\kappa^0_\tau)\right|_{H^1((0,\,1),\,\R^{2M})}^2
+\sum_{n\in\mathbb N_0} \ex^{\hat\lambda n}
\left|\bar\kappa^{\hat u}(n,\,\Bigcdot)(v_n)\right|_{H^1((n,\,n+1),\,\R^{2M})}^2\right)\\
&\le\overline C_{\left[\lambda,\,|\hat u|_{\WW^{\rm st}}\right]}
\left(|v_0|_{L^2_{\diver}(\Omega,\,\R^3)}^2+|\kappa^0_\tau|_{\R^{2M}}^2
+\sum_{n\in\mathbb N_0} \ex^{\hat\lambda n}|v_n|_{H}^2\right)
\end{align*}
and, using~\eqref{est.exp} (with $t=n$) and the identity
$\sum_{n\in\mathbb N_0} \ex^{(\hat\lambda-\lambda) n}=\frac{\ex^{(\hat\lambda-\lambda)}}{1-\ex^{(\hat\lambda-\lambda)}}$,
it follows that
\begin{equation}\label{ex.eta-cont}
\left|\ex^{\frac{\hat\lambda}{2} \Bigcdot}\kappa_{\hat u,\lambda}(v_0,\,\kappa^0_\tau)\right|_{H^1(\R_0,\,\R^{2M})}^2
\le\overline C_{\left[(\lambda-\hat\lambda)^{-1},\,\lambda,\,|\hat u|_{\WW^{\rm st}}\right]}
\left(|v_0|_{L^2_{\diver}(\Omega,\,\R^3)}^2+|\kappa^0_\tau|_{\R^{2M}}^2\right),
\end{equation}
for any given $0\le\hat\lambda<\lambda$, which finishes the proof of Theorem \ref{T:ex.st-ct.lin}.
\hfil\qed
\end{enumerate}

\begin{corollary}\label{C:ex.H1-cont}
The solution $v=v(v_0,\,\kappa^0_\tau)$ in Theorem~\ref{T:ex.st-ct.lin} satisfies the estimate
\begin{equation}\label{ex.H1-cont}
\left|\ex^{\frac{\hat\lambda}{2} \Bigcdot}v\right|_{W(\R_0,\,H^1_{\diver}(\Omega,\,\R^3),\,H^{-1}(\Omega,\,\R^3))}^2
\le\overline C_{\left[(\lambda-\hat\lambda)^{-1},\,\lambda,\,|\hat u|_{\WW^{\rm st}}\right]}
\left(|v_0|_{L^2_{\diver}(\Omega,\,\R^3)}^2+|\kappa^0_\tau|_{\R^{2M}}^2\right).
\end{equation}
\end{corollary}
\begin{proof}
Proceeding as in the proof of Lemma~\ref{L:smooth-prop}, we start by noticing that
$w\coloneqq \ex^{\frac{\hat\lambda}{2} \Bigcdot}v$ solves system~\eqref{sys-vg}, in each interval of time $(a,\,b)\subset\R_0$,
with
$\left(w(a),\,g,\,K\eta\right)
=\left(\ex^{\frac{\hat\lambda}{2}a}v(a),\,-\frac{\hat\lambda}{2}\ex^{\frac{\hat\lambda}{2}\Bigcdot}v,\,
\Xi\ex^{\frac{\hat\lambda}{2}\Bigcdot}\kappa^{\hat u,\,\lambda}\right)$.
From Theorem~\ref{T:exist-wsol-v}, we have
\begin{align*}
&\quad\;|w|_{W((n,\,n+1),\,H^1_{\diver}(\Omega,\,\R^3),\,H^{-1}(\Omega,\,\R^3))}^2\\
&\le\overline C_{\left[|\hat u|_{\WW^{\rm wk}}\right]}
\left(\ex^{\hat\lambda n}|v(n)|_{L^2_{\diver}(\Omega,\,\R^3)}^2+
\textstyle\frac{\hat\lambda^2}{4}|\ex^{\frac{\hat\lambda}{2}\Bigcdot}v|_{L^2((n,\,n+1),\,H^{-1}(\Omega,\,\R^3))}^2
+|\ex^{\frac{\hat\lambda}{2}\Bigcdot}\kappa^{\hat u,\,\lambda}|_{H^1((n,\,n+1),\,\R^{2M})}^2\right);
\end{align*}
thus, from~\eqref{ex.eta-cont} and~\eqref{est.exp} and from the continuity of
the inclusion $L^2(\Omega,\,\R^3)\subset H^{-1}(\Omega,\,\R^3)$,
\begin{align*}
&\quad\,|w|_{W(\R_0,\,H^1_{\diver}(\Omega,\,\R^3),\,H^{-1}(\Omega,\,\R^3))}^2\\
&\le\overline C_{\left[\frac{1}{\lambda-\hat\lambda},\,\lambda,\,|\hat u|_{\WW^{\rm st}}\right]}
\left(\sum_{n\in\N}\ex^{\hat\lambda n}\ex^{-\lambda n}+\hspace*{-.2em}
\int_{\R_0}\ex^{\hat\lambda t}\ex^{-\lambda t}\,\ed t+1\right)\left(|v_0|_{L^2_{\diver}(\Omega,\,\R^3)}^2+|\kappa^0_\tau|_{\R^{2M}}^2\right),
\end{align*}
which implies~\eqref{ex.H1-cont}, because $\sum_{n\in\N}\ex^{\hat\lambda n}\ex^{-\lambda n}+
\int_{\R_0}\ex^{\hat\lambda t}\ex^{-\lambda t}\,\ed t+1
=\frac{1}{(1-\ex^{\hat\lambda-\lambda})}+\frac{1}{(\lambda-\hat\lambda)}+1$, since $\hat\lambda<\lambda$.
\end{proof}

\section{The Oseen--Stokes system: feedback stabilizing control}\label{S:linear-feed}
In this Section, we show that the finite-dimensional exponentially stabilizing control (cf.~Theorem~\ref{T:ex.st-ct.lin})
can be chosen in feedback form. In order to be more precise we will need to derive first some auxiliary results and consider a suitable
extended ``equivalent'' system.

\subsection{Some auxiliary results}
Once more we recall the projection
$\Pi$ and the space~$\NN$, see~\eqref{ProjH} and~\eqref{N.sp}.

\begin{lemma}\label{L:L2H1b-H1}
Let $v$ solve system~\eqref{sys-v}, with $\zeta=\Xi\kappa$. If
$\ex^{\frac{\lambda}{2} \Bigcdot}\Pi v\in L^2(\R_0,\,H)$ and
$\ex^{\frac{\lambda}{2} \Bigcdot}\kappa\in H^1(\R_0,\,\NN^\bot)$, then 
$\ex^{\frac{\lambda}{2} \Bigcdot}v\in W(\R_0,\,H^1_{\diver}(\Omega,\,\R^3),\,H^{-1}(\Omega,\,\R^3))$, with
\begin{align}
&\quad|\ex^{\frac{\lambda}{2} \Bigcdot}v|_{W(\R_0,\,H^1_{\diver}(\Omega,\,\R^3),\,H^{-1}(\Omega,\,\R^3))}^2\notag\\
&\le
\overline C_{\left[|\hat u|_{\WW^{\rm st}},\,\lambda\right]}\left(|v_0|_{L^2_{\diver}(\Omega,\,\R^3)}^2
+|\ex^{\frac{\lambda}{2} \Bigcdot}\Pi v|_{L^2(\R_0,\,H)}^2
+|\ex^{\frac{\lambda}{2} \Bigcdot}\kappa|_{H^1(\R_0,\,\NN^\bot)}^2\right).\label{vHtovH1}
\end{align}
\end{lemma}

\begin{proof}
We rewrite
$|\ex^{\frac{\lambda}{2} \Bigcdot}v|_{W(\R_0,\,H^1_{\diver}(\Omega,\,\R^3),\,H^{-1}(\Omega,\,\R^3))}^2$ as
the sum $|\ex^{\frac{\lambda}{2} \Bigcdot}v|_{W((0,\,1),\,H^1_{\diver}(\Omega,\,\R^3),\,H^{-1}(\Omega,\,\R^3))}^2
+|\ex^{\frac{\lambda}{2} \Bigcdot}v|_{W(\R_1,\,H^1_{\diver}(\Omega,\,\R^3),\,H^{-1}(\Omega,\,\R^3))}^2$;
from Theorem~\ref{T:exist-wsol-v} we can derive that 
\begin{align}
|\ex^{\frac{\lambda}{2} \Bigcdot}v|_{W((0,\,1),\,H^1_{\diver}(\Omega,\,\R^3),\,H^{-1}(\Omega,\,\R^3))}^2
&\le C|\ex^{\frac{\lambda}{2} \Bigcdot}|_{C^1([0,\,1],\,\R)}^2
|v|_{W((0,\,1),\,H^1_{\diver}(\Omega,\,\R^3),\,H^{-1}(\Omega,\,\R^3))}^2\notag\\
&\le\overline C_{\left[|\hat u|_{\WW^{\rm st}},\,\lambda\right]}\left(|v_0|_{L^2_{\diver}(\Omega,\,\R^3)}^2
+|\kappa|_{H^1((0,\,1),\,\NN^\bot)}^2\right)\label{ev.W1}
\end{align}
and, from Lemma~\ref{L:smooth-prop} we have that for all $t\ge1$
\[
|v(t)|_{H^1_{\diver}(\Omega,\,\R^3)}^2\le
\overline C_{\left[|\hat u|_{\WW^{\rm st}}\right]}\left(|v(t-1)|_{L^2_{\diver}(\Omega,\,\R^3)}^2
+|\kappa|_{H^1((t-1,\,t),\,\NN^\bot)}^2\right),
\]
which allow us to obtain
\begin{align*}
&\quad|\ex^{\frac{\lambda}{2} \Bigcdot}v|_{L^2(\R_1,\,H^1_{\diver}(\Omega,\,\R^3))}^2
\le\sum_{n=1}^{+\infty} \ex^{(n+1)\lambda}|v|_{L^2((n,\,n+1),\,H^1_{\diver}(\Omega,\,\R^3))}^2\\
&\le\overline C_{\left[|\hat u|_{\WW^{\rm st}}\right]}\sum_{n=1}^{+\infty}
\ex^{\lambda(n+1)}\int_{n}^{n+1} |v(t-1)|_{L^2_{\diver}(\Omega,\,\R^3)}^2
+|\kappa|_{H^1((t-1,\,t),\,\NN^\bot)}^2\,\ed t\\
&\le\overline C_{\left[|\hat u|_{\WW^{\rm st}}\right]}\sum_{n=1}^{+\infty}\left(
\ex^{2\lambda}\int_{n}^{n+1} |\ex^{\frac{\lambda}{2}(t-1)}v(t-1)|_{L^2_{\diver}(\Omega,\,\R^3)}^2\,\ed t
+\ex^{\lambda(n+1)}|\kappa|_{H^1((n-1,\,n+1),\,\NN^\bot)}^2\right)\\
&\le\overline C_{\left[|\hat u|_{\WW^{\rm st}}\right]}\ex^{2\lambda}
\Biggl(\int_{0}^{+\infty} |\ex^{\frac{\lambda}{2} t}v(t)|_{L^2_{\diver}(\Omega,\,\R^3)}^2\,\ed t\\
&\hspace*{8em}+
\sum_{n=1}^{+\infty}\left(|\ex^{\frac{\lambda}{2} \Bigcdot}\kappa|_{L^2((n-1,\,n+1),\,\NN^\bot)}^2
+|\ex^{\frac{\lambda}{2} \Bigcdot}\p_t\kappa|_{L^2((n-1,\,n+1),\,\NN^\bot)}^2\right)\Biggr).
\end{align*}
Since $\ex^{\frac{\lambda}{2} \Bigcdot}\p_t\kappa=\p_t(\ex^{\frac{\lambda}{2} \Bigcdot}\kappa)
-\frac{\lambda}{2}\ex^{\frac{\lambda}{2} \Bigcdot}\kappa$, we can derive that
\begin{align*}
|\ex^{\frac{\lambda}{2} \Bigcdot}v|_{L^2(\R_1,\,H^1_{\diver}(\Omega,\,\R^3))}^2
&\le\overline C_{\left[|\hat u|_{\WW^{\rm st}},\,\lambda\right]}
\left(|\ex^{\frac{\lambda}{2} \Bigcdot}v|_{L^2(\R_0,\,L^2_{\diver}(\Omega,\,\R^3))}^2
+\sum_{n=1}^{+\infty}|\ex^{\frac{\lambda}{2} \Bigcdot}\kappa|_{H^1((n-1,\,n+1),\,\NN^\bot)}^2\right)\\
&\le \overline C_{\left[|\hat u|_{\WW^{\rm st}},\,\lambda\right]}\left(
|\ex^{\frac{\lambda}{2} \Bigcdot}v|_{L^2(\R_0,\,L^2_{\diver}(\Omega,\,\R^3))}^2
+2|\ex^{\frac{\lambda}{2} \Bigcdot}\kappa|_{H^1(\R_0,\,\NN^\bot)}^2\right).
\end{align*}
Thus, from $\p_t(\ex^{\frac{\lambda}{2} \Bigcdot}v)=\frac{\lambda}{2}\ex^{\frac{\lambda}{2} \Bigcdot}v
+\ex^{\frac{\lambda}{2} \Bigcdot}\p_tv$, and since $v$ solves system~\eqref{sys-v}, we can obtain the estimate
$|\p_t(\ex^{\frac{\lambda}{2} \Bigcdot}v)|_{L^2(\R_1,\,H^{-1}(\Omega,\,\R^3))}^2
\le \overline C_{\left[|\hat u|_{\WW^{\rm st}},\,\lambda\right]}
|\ex^{\frac{\lambda}{2} \Bigcdot}v|_{L^2(\R_1,\,H^1_{\diver}(\Omega,\,\R^3))}^2$, which allow us to derive
\begin{align*}
|\ex^{\frac{\lambda}{2} \Bigcdot}v|_{W(\R_1,\,H^1_{\diver}(\Omega,\,\R^3),\,H^{-1}_{\diver}(\Omega,\,\R^3))}^2
&\le \overline C_{\left[|\hat u|_{\WW^{\rm st}},\,\lambda\right]}\left(
|\ex^{\frac{\lambda}{2} \Bigcdot}v|_{L^2(\R_0,\,L^2_{\diver}(\Omega,\,\R^3))}^2
+|\ex^{\frac{\lambda}{2} \Bigcdot}\kappa|_{H^1(\R_0,\,\NN^\bot)}^2\right);
\end{align*}
and then, using~\eqref{ev.W1}, we arrive to
\begin{align*}
&\quad|\ex^{\frac{\lambda}{2} \Bigcdot}v|_{W(\R_0,\,H^1_{\diver}(\Omega,\,\R^3),\,H^{-1}(\Omega,\,\R^3))}^2\\
&\le\overline C_{\left[|\hat u|_{\WW^{\rm st}},\,\lambda\right]}\hspace*{-.2em}\left(|v_0|_{L^2_{\diver}(\Omega,\,\R^3)}^2
+|\ex^{\frac{\lambda}{2} \Bigcdot}v|_{L^2(\R_0,\,L^2_{\diver}(\Omega,\,\R^3))}^2
+|\ex^{\frac{\lambda}{2} \Bigcdot}\kappa|_{H^1(\R_0,\,\NN^\bot)}^2\right).
\end{align*}
Finally, from $\kappa(t)\in\NN^\bot$, we have $Q^M_f\kappa(t)=Q^M_fP_{\NN^\bot}\kappa(t)=z^{v\cdot\nnn(t)}$;
from Lemma~\ref{L:eq-normsAd1},
it follows that $|v(t)|_{L^2_{\diver}(\Omega,\,\R^3)}^2\leq C(|\Pi v(t)|_H^2+|Q^M_fk(t)|_{\NN^\bot}^2)$, which
allow us to derive~\eqref{vHtovH1}.
\end{proof}

\begin{corollary}\label{C:L2H1b-H1}
Let $s\geq 0$ and let $v$ solve system~\eqref{sys-v}, in $\R_s\times\Omega$, with $\zeta=\Xi\kappa$ and $v(s)=v_s$. If
$\ex^{\frac{\lambda}{2} \Bigcdot}\Pi v\in L^2(\R_s,\,H)$ and
$\ex^{\frac{\lambda}{2} \Bigcdot}\kappa\in H^1(\R_s,\,\NN^\bot)$, then 
\begin{align}
&\quad|\ex^{\frac{\lambda}{2} \Bigcdot}v|_{W(\R_s,\,H^1_{\diver}(\Omega,\,\R^3),\,H^{-1}(\Omega,\,\R^3))}^2\notag\\
&\le
\overline C_{\left[|\hat u|_{\WW^{\rm st}},\,\lambda\right]}\hspace*{-.2em}\left(\ex^{\lambda s}|v_s|_{L^2_{\diver}(\Omega,\,\R^3)}^2
+|\ex^{\frac{\lambda}{2} \Bigcdot}\Pi v|_{L^2(\R_s,\,H)}^2
+|\ex^{\frac{\lambda}{2} \Bigcdot}\kappa|_{H^1(\R_s,\,\NN^\bot)}^2\hspace*{-.2em}\right)\hspace*{-.2em}.\label{vHtovH1-s}
\end{align}
\end{corollary}

\begin{proof}
 Since $v$ solve system~\eqref{sys-v}, in $\R_s\times\Omega$, we have that $w=w(r)\coloneqq v(r+s)$ solves system~\eqref{sys-v},
 in $\R_0\times\Omega$, with $w(0)=v_s$,  $\zeta=\Xi\xi(r)\coloneqq \Xi\kappa(r+s)$, and $\hat u_s(r)\coloneqq \hat u(r+s)$ in the place
 of $\hat u$.
 Since $|\hat u_s|_{\WW^{\rm st}}\leq |\hat u|_{\WW^{\rm st}}$, by Lemma~\ref{L:L2H1b-H1}, we have that
 $|\ex^{\frac{\lambda}{2} \Bigcdot}w|_{W(\R_0,\,H^1_{\diver}(\Omega,\,\R^3),\,H^{-1}(\Omega,\,\R^3))}^2$
 is bounded by $\overline C_{\left[|\hat u|_{\WW^{\rm st}},\,\lambda\right]}\left(|v_s|_{L^2_{\diver}(\Omega,\,\R^3)}^2
 +|\ex^{\frac{\lambda}{2} \Bigcdot}\Pi w|_{L^2(\R_0,\,H)}^2
 +|\ex^{\frac{\lambda}{2} \Bigcdot}\xi|_{H^1(\R_0,\,\NN^\bot)}^2\right)$, 
that is,
\begin{align}
&\quad|\ex^{\frac{\lambda}{2} (\Bigcdot-s)}v|_{W(\R_s,\,H^1_{\diver}(\Omega,\,\R^3),\,H^{-1}(\Omega,\,\R^3))}^2\label{vHtovH1-ts}\\
&\le
\overline C_{\left[|\hat u|_{\WW^{\rm st}},\,\lambda\right]}\left(|v_s|_{L^2_{\diver}(\Omega,\,\R^3)}^2
+|\ex^{\frac{\lambda}{2} (\Bigcdot-s)}\Pi v|_{L^2(\R_s,\,H)}^2
+|\ex^{\frac{\lambda}{2} (\Bigcdot-s)}\kappa|_{H^1(\R_s,\,\NN^\bot)}^2\right),\notag
\end{align}
which implies~\eqref{vHtovH1-s}.
\end{proof}

We will also need the following corollary of Theorem~\ref{T:ex.st-ct.lin} and Corollary~\ref{C:ex.H1-cont}.

\begin{corollary}\label{C:ex.st.s}
 Let us be given $s\ge0$, $\lambda>0$ and $\hat u\in\WW^{\rm st}$. Then there exists
$M=\overline C_{\left[|\hat u|_{\WW^{\rm st}},\lambda\right]}\geq1$ with the following property:
for each $(v_s,\,\kappa^s_\tau)\in \aA_{\Xi_1}\times Q_l^M\R^{2M}$,
there exists a ``control'' vector function
$\kappa^{\hat u,\lambda}=\kappa^{\hat u,\lambda}(v_s,\,\kappa^s_\tau)\in H^1(\R_s,\,\R^{2M})$ such that the weak
solution~$v$ of system~\eqref{sys-v} in $\R_s\times\Omega$, with $\zeta=\Xi\kappa^{\hat u,\lambda}$, satisfies the inequality
\begin{align}
&\quad\,\left|\ex^{\frac{\lambda}{2} \Bigcdot}v\right|_{W(\R_s,\,H^1_{\diver}(\Omega,\,\R^3),\,H^{-1}(\Omega,\,\R^3))}^2
+\bigl|\ex^{\frac{\lambda}{2}\Bigcdot}\kappa^{\hat u,\lambda}(v_s,\,\kappa^s_\tau)\bigr|_{H^1(\R_s,\,\R^{2M})}^2\notag\\
&\le\overline C_{\left[\frac{1}{\lambda},\,\lambda,\,|\hat u|_{\WW^{\rm st}}\right]}
\ex^{\lambda s}\left(|v_s|_{L^2_{\diver}(\Omega,\,\R^3)}^2+|\kappa^s_\tau|_{\R^{2M}}^2\right).\label{ex.H1-cont.s}
\end{align}
Moreover, the mapping
$(v_0,\,\kappa^0_\tau)\mapsto \kappa^{\hat u,\lambda}(v_0,\,\kappa^0_\tau)$ is linear.
\end{corollary}

\begin{proof}
 As in the proof of Corollary~\ref{C:L2H1b-H1} we change the time variable $t=r+s$ and can reduce the problem to the cylinder $\R_0\times\Omega$.
 We may take $(\lambda,\,2\lambda)$ in the place of $(\hat\lambda,\,\lambda)$ in Theorem~\ref{T:ex.st-ct.lin}, and
 $\lambda$ in the place of $\hat\lambda$ in Corollary~\ref{C:ex.H1-cont}.
 
 From Theorem~\ref{T:ex.st-ct.lin} and Corollary~\ref{C:ex.H1-cont} there is a linear
 mapping $(v_s,\,\kappa^s_\tau)\mapsto\xi(v_s,\,\kappa^s_\tau)=\xi(v_s,\,\kappa^s_\tau)(r)$, $r\in\R_0$,
 such that the solution $w=w(r)$ for system~\eqref{sys-v},
 with $w\rest\Gamma=\Xi\xi$ and with~$\hat u_s(r)=\hat u(s+r)$ in the place of $\hat u$, satisfies $(w,\,Q^M_l\xi)(0)=(v_s,\,\kappa^s_\tau)$ and
 \begin{align*}
&\quad\,\left|\ex^{\frac{\lambda}{2} \Bigcdot}w\right|_{W(\R_0,\,H^1_{\diver}(\Omega,\,\R^3),\,H^{-1}(\Omega,\,\R^3))}^2
+\bigl|\ex^{\frac{\lambda}{2}\Bigcdot}\xi(v_s,\,\kappa^s_\tau)\bigr|_{H^1(\R_0,\,\R^{2M})}^2\\
&\le\overline C_{\left[\frac{1}{\lambda},\,\lambda,\,|\hat u_s|_{\WW^{\rm st}}\right]}
\left(|v_s|_{L^2_{\diver}(\Omega,\,\R^3)}^2+|\kappa^s_\tau|_{\R^{2M}}^2\right).
\end{align*}
Then we can conclude that $(v,\,\kappa)(t)\coloneqq(w,\,\xi)(t-s)$ solves system~\eqref{sys-v} in $\R_s\times\Omega$ with $(v,\,Q^M_l\kappa)(s)=(v_s,\,\kappa^s_\tau)$
and that the estimate 
$\left|\ex^{\frac{\lambda}{2} (\Bigcdot-s)}v\right|_{W(\R_s,\,H^1_{\diver}(\Omega,\,\R^3),\,H^{-1}(\Omega,\,\R^3))}^2
+\bigl|\ex^{\frac{\lambda}{2}(\Bigcdot-s)}\kappa(v_s,\,\kappa^s_\tau)\bigr|_{H^1(\R_s,\,\R^{2M})}^2
\le\overline C_{\left[\frac{1}{\lambda},\,\lambda,\,|\hat u_s|_{\WW^{\rm st}}\right]}
\left(|v_s|_{L^2_{\diver}(\Omega,\,\R^3)}^2+|\kappa^s_\tau|_{\R^{2M}}^2\right)$ holds, which implies~\eqref{ex.H1-cont.s}. 
\end{proof}

\subsection{The extended system}\label{sS:extsyst}
In order to be able to use the dynamical programming principle, we will need to rewrite system~\eqref{sys-v} in a suitable way. We start
by observing that the mapping
\[
 \FF_\aA\colon\aA_{\Xi_1}\to H\times Q^M_f\NN^\perp,\qquad u\mapsto(\Pi u,\,z^{u\cdot\nnn})
\]
is continuous and surjective. From Corollary~\ref{C:L2=Hzun} it is also injective. Then by the Inverse Mapping
Theorem (cf.~\cite[Section~2.3, Corollary~2.7]{Brezis11}) it has a continuous inverse
$\FF_\aA^{-1}\in\LL(H\times Q^M_f\NN^\perp\to \aA_{\Xi_1}).$ Further, since~$\Xi Q^M_f\NN^\perp\subset H^\frac{3}{2}(\Omega,\,\R)\nnn$ we have that
$\FF_\aA^{-1}(0,\,Q^M_f\NN^\perp)\subset H^2_{\diver}(\Omega,\,\R^3)$, because any $u\in\FF_\aA^{-1}(0,\,Q^M_f\NN^\perp)\subset H^\perp$
satisfies $u=\nabla P_\nabla u$, $\Delta P_\nabla u=0$, and~$\nnn\cdot\nabla P_\nabla u=\nnn\cdot u=\nnn\cdot\Xi z^{u\cdot\nnn} \in H^\frac{3}{2}(\Omega,\,\R)$, which
implies~$P_\nabla u\in H^3(\Omega,\,\R)$ (e.g., see~\cite[Chapter~5, Proposition~7.7]{Taylor97}).

Now we rewrite system~\eqref{sys-v}, in $\R_s\times\Omega$
with $\zeta=\Xi\kappa$, in the extended form
\begin{subequations}\label{sys-v-ext}
\begin{align}
 \p_t v+\BB(\hat u)v
 -\nu\Delta v+\nabla p_v&=0,& \p_t\kappa &=\varkappa,\label{sys-v-ext-dyn}\\
 \diver v &=0,& v\rest \Gamma &=\Xi\kappa, \label{sys-v-ext-divbdry}\\
 v(s)&=\FF_\aA^{-1}(v_s^H,\, Q^M_f\kappa_s),& \kappa(s)&=\kappa_s,\label{sys-v-ext-ic}
\end{align}
\end{subequations}
where~$(v_s^H,\,\kappa_s,\,\varkappa)$ will be taken in~$H\times \NN^\perp\times L^2(\R_s,\,\NN^\perp)$.

\begin{remark}
Notice that the compatibility
condition~$(v(s)\cdot\nnn)\nnn=\Xi Q^M_f\kappa(s)$, which is required to guarantee the existence of
weak solutions for the ``equivalent'' system~\eqref{sys-v}, is indeed guaranteed by the condition~$v(s)=\FF_\aA^{-1}(v_s^H,\,Q^M_f\kappa_s)$.
Notice also that if we impose~$\kappa(s)$ at the boundary, then we can impose only a tangential initial condition~$\Pi v(s)\in H$ for~$v(s)$.
Essentially, the  initial ``weak''
condition for the extended system~\eqref{sys-v-ext}  is the pair~$(v_s^H,\,\kappa_s)$.
\end{remark}

Observe that $(v_s^H,\,\kappa_s)\mapsto\left(\FF_\aA^{-1}(v_s^H,\,Q^M_f\kappa_s),\,Q^M_l\kappa_s\right)$ is an isomorphism
in~$\LL(H\times \NN^\perp\to \aA_{\Xi_1}\times Q^M_l\NN^\perp)$. Form Corollary~\ref{C:ex.st.s}, it makes sense to consider the following problem:
\begin{problem}\label{Pb:Ms}
Let us be given $s\ge0$, $\lambda>0$, $\hat u\in\WW^{\rm st}$, and let $M\in\N$ be given by Corollary~\ref{C:ex.st.s}. Then
for given $(v_s^H,\,\kappa_s)\in H\times \NN^\perp$, find
the minimum of the functional
\[
M_s^\lambda(v,\,\kappa,\,\varkappa)\coloneqq \left|\ex^{\frac{\lambda}{2}\Bigcdot}\Pi v\right|_{L^2(\R_s,\,H)}^2
+\left|\ex^{\frac{\lambda}{2}\Bigcdot}\kappa\right|_{L^2(\R_s,\,\NN^\bot)}^2+\left|\ex^{\frac{\lambda}{2}\Bigcdot}\varkappa\right|_{L^2(\R_s,\,\NN^\bot)}^2
\]
on the set of functions
\[
\XX_s^{1,1}\coloneqq \left\{(v,\,\kappa,\,\varkappa)\in \ZZ^{1,\,1}_s \left|
\begin{array}{l}
\ex^{\frac{\lambda}{2}\Bigcdot}(v,\,\kappa,\,\varkappa)\in \ZZ^{1,\,1}_s\mbox{ and }
(v,\,\kappa,\,\varkappa)\mbox{ solves~\eqref{sys-v-ext-dyn}--\eqref{sys-v-ext-divbdry}}
\end{array}\right.\hspace*{-.5em}\right\}
\]
satisfying $A(v,\,\kappa,\,\varkappa)\coloneqq (\Pi v(s),\,\kappa(s))=(v_s^H,\,\kappa_s)$; where
\begin{align*}
\ZZ^{1}_s&\coloneqq W(\R_s,\,H^1_{\diver}(\Omega,\,\R^3),\,H^{-1}(\Omega,\,\R^3))\times
H^1(\R_s,\,\NN^\bot)\times L^2(\R_s,\,\NN^\bot)\\
\ZZ^{1,\,1}_s&\coloneqq \{(v,\,\kappa,\,\varkappa)\in \ZZ^{1}_s\mid v(s)\in\aA_{\Xi_1}\}.
\end{align*}
\end{problem}%

\begin{lemma} \label{L:Ms}
Problem~\ref{Pb:Ms} has a unique minimizer~$(v_s^*,\,\kappa_s^*,\,\varkappa_s^*)$.
Moreover, there exists a continuous linear self-adjoint operator $R_{\hat
u}^{\lambda,\,s}\in\LL(H\times \NN^\perp)$ such that
\begin{subequations}\label{R-opt.norm}
\begin{align}
(R_{\hat u}^{\lambda,\,s}(v_s^H,\,\kappa_s),\,(v_s^H,\,\kappa_s))_{H\times \NN^\perp}
&=M_s^\lambda(v_s^*,\,\kappa_s^*,\,\varkappa_s^*)\label{optimalcost}\\
|R_{\hat u}^{\lambda,\,s}|_{\LL(H\times \NN^\perp)}&\le
\overline C_{\left[|\hat u|_{\WW^{\rm st}},\,\lambda,\,\frac{1}{\lambda}\right]}\ex^{\lambda s}.\label{normofR}
\end{align}
\end{subequations}
\end{lemma}

\begin{proof}
We suppose $\XX_s^{1,1}$ endowed with the norm inherited from $\ZZ^{1}_s$ and we start by observing that
the set $A_{v_s^H,\,\kappa_s,\,c}\coloneqq \{(v,\,\kappa,\,\varkappa)\in\XX_s^{1,1}\mid A(v,\,\kappa,\,\varkappa)
=(v_s^H,\,\kappa_s)\mbox{ and }M_s^\lambda(v,\,\kappa,\,\varkappa)\le c\}$
is bounded, for any $(v_s^H,\,\kappa_s,\,c)\in H\times \NN^\perp\times\R_0$. Indeed, since the initial condition
$(v_s^H,\,\kappa_s)$ and initial time $t=s$ are fixed, the boundedness
follows from Corollary~\ref{C:L2H1b-H1}. 
On the other hand, from Corollary~\ref{C:ex.st.s}, the set $A_{v_s^H,\,\kappa_s}\coloneqq \{(v,\,\kappa,\,\varkappa)\in\XX_s^{1,1}\mid A(v,\,\kappa,\,\varkappa)
=(v_s^H,\,\kappa_s)\}$ is nonempty for any given $(v_s^H,\,\kappa_s)\in H\times \NN^\perp$, that is,
the mapping
$A\colon \XX_s^{1,1}\to H\times \NN^\perp$
is surjective. Further, we observe that $M_s^\lambda(v,\,\kappa,\,\varkappa)$
induces a scalar product in $\ZZ_s^{1,1}$:
\begin{align*}
&\quad\Bigl((v,\,\kappa,\,\varkappa),\,(u,\,\eta,\,\xi)\Bigr)_{M_s^\lambda}\\
&\coloneqq 
\left(\ex^{\frac{\lambda}{2}\Bigcdot}\Pi v
,\,\ex^{\frac{\lambda}{2}\Bigcdot}\Pi u\right)_{L^2(\R_s,\,V)}
+\left(\ex^{\frac{\lambda}{2}\Bigcdot}\kappa,\,\ex^{\frac{\lambda}{2}\Bigcdot}\eta\right)_{L^2(\R_s,\,\NN^\bot)}
+\left(\ex^{\frac{\lambda}{2}\Bigcdot}\varkappa,\,\ex^{\frac{\lambda}{2}\Bigcdot}\xi\right)_{L^2(\R_s,\,\NN^\bot)}.
\end{align*}
We can derive that
Problem~\ref{Pb:Ms} has
a unique minimizer~$(v_s^*,\,\kappa_s^*,\,\varkappa_s^*)=(v_s^*,\,\kappa_s^*,\,\varkappa_s^*)(v_s^H,\,\kappa_s)$, which linearly
depends on~$(v_s^H,\,\kappa_s)$
(cf.~\cite[Appendix, Lemma~A.14 and Remark~A.15]{Rod15-cocv}).

Again
from Corollary~\ref{C:ex.st.s}, we have that the mapping
\[
\left((v_s^{H,1},\,\kappa_s^{1}),\,(v_s^{H,2},\,\kappa_s^{2})\right)\mapsto
\left(\,(v_s^*,\,\kappa_s^*,\,\varkappa_s^*)(v_s^{H,1},\,\kappa_s^{1}),\,
(v_s^*,\,\kappa_s^*,\,\varkappa_s^*)(v_s^{H,2},\,\kappa_s^{2})\,\right)_{M_s^\lambda}
\]
is a symmetric continuous bilinear form on~$H\times \NN^\perp$ which is bounded by
$\overline C_{\left[|\hat u|_{\WW^{\rm st}},\,\lambda,\,\frac{1}{\lambda}\right]}\ex^{\lambda s}$ on the unit ball;
thus, the optimal cost
can be written as~\eqref{optimalcost}, where~$R_{\hat u}^{\lambda,\,s}$ is a
bounded and self-adjoint operator,  which norm satisfy~\eqref{normofR}.
\end{proof}

Next we consider another minimization problem related to Problem~\ref{Pb:Ms}.

\begin{problem}\label{Pb:Ns}
Let us be given $s>s_0\ge0$, $\lambda>0$, $\hat u\in\WW^{\rm st}$, and let $M\in\N$ be given by Corollary~\ref{C:ex.st.s}.
Given $(v_{s_0}^H,\,\kappa_{s_0})\in H\times \NN^\perp$, find the minimum of the functional
\begin{align*}
N_{s_0,\,s}^\lambda(v,\,\kappa,\,\varkappa)\coloneqq &\left|\ex^{\frac{\lambda}{2}\Bigcdot}\Pi v\right|_{L^2((s_0,\,s),\,H))}^2
+\left|\ex^{\frac{\lambda}{2}\Bigcdot}\kappa\right|_{L^2((s_0,\,s),\,\NN^\bot)}^2
+\left|\ex^{\frac{\lambda}{2}\Bigcdot}\varkappa\right|_{L^2((s_0,\,s),\,\NN^\bot)}^2\\
&+\left(R_{\hat u}^{\lambda,\,s}\left(\Pi v(s),\,\kappa(s)\right),\,\left(\Pi v(s)\,\kappa(s)\right)\right)
\end{align*}
on the set of functions
\[
(v,\,\kappa,\,\varkappa)\in\widehat\XX\coloneqq \left\{(v,\,\kappa,\,\varkappa)\in
\ZZ_{(s_0,\,s)}^{1,\,1}\left|
\begin{array}{l}
v\mbox{ solves~~\eqref{sys-v-ext-dyn}--\eqref{sys-v-ext-divbdry}}
\end{array}\right.\hspace*{-.5em}\right\}
\]
that satisfy $A(v,\,\kappa,\,\varkappa)\coloneqq (\Pi v(s_0),\,\kappa(s_0))=(v_{s_0}^H,\,\kappa_{s_0})$;
where~$\ZZ_{(s_0,\,s)}^{1,\,1}\coloneqq \ZZ_{s_0}^{1,\,1}\rest{(s_0,\,s)}.$
\end{problem}

Reasoning as in the proof of Lemma~\ref{L:Ms}, we can derive that Problem~\ref{Pb:Ns}
has a unique minimizer $(v_{s_0,\,s}^\bullet,\,\kappa_{s_0,\,s}^\bullet,\,\varkappa_{s_0,\,s}^\bullet)(v_{s_0}^H,\,\kappa_{s_0})$,
which is a linear function of~$(v_{s_0}^H,\,\kappa_{s_0})\in H\times \NN^\perp$. 
The following Lemma is the {\em dynamic programming principle\/} for Problem~\ref{Pb:Ms} (with~$s=s_0$).

\begin{lemma} \label{L:MsNs}
The minimizers of Problems~\ref{Pb:Ms} and~\ref{Pb:Ns} have the following properties: the restriction
of~$(v_{s_0}^*,\,\kappa_{s_0}^*,\,\varkappa_{s_0}^*)(v_{s_0}^H,\,\kappa_{s_0})$ to the time interval $(s_0,\,s)$ does coincide with
$(v_{s_0,\,s}^\bullet,\,\kappa_{s_0,\,s}^\bullet,\,\varkappa_{s_0,\,s}^\bullet)(v_{s_0}^H,\,\kappa_{s_0})$, and the restriction
of~$(v_{s_0}^*,\,\kappa_{s_0}^*,\,\varkappa_{s_0}^*)(v_{s_0}^H,\,\kappa_{s_0})$ to the half-line~$\R_s=(s,\,+\infty)$ does coincide with
$(v_{s}^*,\,\kappa_{s}^*,\,\varkappa_{s}^*)(\Pi v_{s_0}^*(s),\,\kappa_{s_0}^*(s))$.
\end{lemma}

\begin{proof}
From $(v_{s_0,\,s}^\bullet,\,\kappa_{s_0,\,s}^\bullet,\,\varkappa_{s_0,\,s}^\bullet)(v_{s_0}^H,\,\kappa_{s_0})\in\widehat\XX$ and
$(v_{s}^*,\,\kappa_{s}^*,\,\varkappa_{s}^*)(\Pi v_{s_0,\,s}^\bullet(s),\,\kappa_{s_0,\,s}^\bullet(s))\in \XX_{s}^{1,1}$, we necessarily have
$(v_{s_0,\,s}^\bullet(t)\cdot\nnn)\nnn=\Xi Q^M_f\kappa_{s_0,\,s}^\bullet(t)$ for any~$t\in[s_0,\,s]$, and
$(v_{s}^*(r)\cdot\nnn)\nnn=\Xi Q^M_f\kappa_{s}^*(r)$ for any~$r\ge s$.
Thus, since $\kappa_{s_0,\,s}^\bullet(s)=\kappa_{s}^*(s)$, we find that~$v_{s_0,\,s}^\bullet(s)\cdot\nnn=v_{s}^*(s)\cdot\nnn$, and
from~$\Pi v_{s_0,\,s}^\bullet(s)=\Pi v_{s}^*(s)$, it follows that
$v_{s_0,\,s}^\bullet(s)=\FF_\aA^{-1}(\Pi v_{s_0,\,s}^\bullet(s),\,z^{v_{s_0,\,s}^\bullet(s)\cdot\nnn})
=\FF_\aA^{-1}(\Pi v_{s}^*(s),\,z^{v_{s}^*(s)\cdot\nnn})=v_{s}^*(s)$.
Therefore, the concatenation
\[
(\widetilde v,\,\widetilde \kappa,\,\widetilde \varkappa)(t)\coloneqq
\left\{\begin{array}{ll}
(v_{s_0,\,s}^\bullet,\,\kappa_{s_0,\,s}^\bullet,\,\varkappa_{s_0,\,s}^\bullet)(v_{s_0}^H,\,\kappa_{s_0})(t),&\mbox{if }t\in[s_0,\,s];\\
(v_{s}^*,\,\kappa_{s}^*,\,\varkappa_{s}^*)(\Pi v_{s_0,\,s}^\bullet(s),\,\kappa_{s_0,\,s}^\bullet(s))(t),&\mbox{if }t\in\ge s;
\end{array}\right.
\]
is a function in~$\XX_{s_0}^{1,1}$.
Analogously, we can see that also the concatenation
\[
(\widehat v,\,\widehat \kappa,\,\widehat \varkappa)(t)\coloneqq
\left\{\begin{array}{ll}
(v_{s_0}^*,\,\kappa_{s_0}^*,\,\varkappa_{s_0}^*)(v_{s_0}^H,\,\kappa_{s_0})(t),&\mbox{if }t\in[s_0,\,s];\\
(v_{s}^*,\,\kappa_{s}^*,\,\varkappa_{s}^*)(\Pi v_{s_0}^*(s),\,\kappa_{s_0}^*(s))(t),&\mbox{if }t\in\ge s;
\end{array}\right.
\]
is a function in~$\XX_{s_0}^{1,1}$.

By the definition of $(v_{s_0}^*,\,\kappa_{s_0}^*,\,\varkappa_{s_0}^*)$ and $(v_{s}^*,\,\kappa_{s}^*,\,\varkappa_{s}^*)$ we can conclude that
\begin{align}
&\quad M_{s_0}^\lambda(v_{s_0}^*,\,\kappa_{s_0}^*,\,\varkappa_{s_0}^*)(v_{s_0}^H,\,\kappa_{s_0})
\le M_{s_0}^\lambda (\widehat v,\,\widehat \kappa,\,\widehat \varkappa)
= N_{s_0,\,s}^\lambda(v_{s_0}^*,\,\kappa_{s_0}^*,\,\varkappa_{s_0}^*)(v_{s_0}^H,\,\kappa_{s_0})\rest{(s_0,\,s)}\notag\\
&= \left|\ex^{\frac{\lambda}{2}t}\Pi v_{s_0}^*\right|_{L^2((s_0,\,s),\,H))}^2
+\left|\ex^{\frac{\lambda}{2}t}\kappa_{s_0}^*\right|_{L^2((s_0,\,s),\,\NN^\bot)}^2
+\left|\ex^{\frac{\lambda}{2}t}\varkappa_{s_0}^*\right|_{L^2((s_0,\,s),\,\NN^\bot)}^2\notag\\
&\quad +\left(R_{\hat u}^{\lambda,\,s}\left(\Pi v_{s_0}^*(s),\,\kappa_{s_0}^*(s)\right),\,\left(\Pi v_{s_0}^*(s)\,\kappa_{s_0}^*(s)\right)\right)
\le M_{s_0}^\lambda(v_{s_0}^*,\,\kappa_{s_0}^*,\,\varkappa_{s_0}^*)(v_{s_0}^H,\,\kappa_{s_0}).\label{dyn1}
\end{align}
From the uniqueness of the minimizer for Problem~\ref{Pb:Ms}, it follows
$(v_{s_0}^*,\,\kappa_{s_0}^*,\,\varkappa_{s_0}^*)(v_{s_0}^H,\,\kappa_{s_0})=(\widehat v,\,\widehat \kappa,\,\widehat \varkappa).$ In particular,
$(v_{s_0}^*,\,\kappa_{s_0}^*,\,\varkappa_{s_0}^*)(v_{s_0}^H,\,\kappa_{s_0})\rest{\R_s}
=(v_{s}^*,\,\kappa_{s}^*,\,\varkappa_{s}^*)(\Pi v_{s_0}^*(s),\,\kappa_{s_0}^*(s)).$

On the other side, from the definition of $(v_{s_0,\,s}^\bullet,\,\kappa_{s_0,\,s}^\bullet,\,\varkappa_{s_0,\,s}^\bullet)(v_{s_0}^H,\,\kappa_{s_0})$
and~\eqref{dyn1},
we also have
\begin{align*}
&\quad M_{s_0}^\lambda (\widetilde v,\,\widetilde \kappa,\,\widetilde \varkappa)
=N_{s_0,\,s}^\lambda(v_{s_0,\,s}^\bullet,\,\kappa_{s_0,\,s}^\bullet,\,\varkappa_{s_0,\,s}^\bullet)(v_{s_0}^H,\,\kappa_{s_0})
\le N_{s_0,\,s}^\lambda(v_{s_0}^*,\,\kappa_{s_0}^*,\,\varkappa_{s_0}^*)(v_{s_0}^H,\,\kappa_{s_0})\rest{(s_0,\,s)}\\
&=M_{s_0}^\lambda (\widehat v,\,\widehat \kappa,\,\widehat \varkappa)=M_{s_0}^\lambda(v_{s_0}^*,\,\kappa_{s_0}^*,\,\varkappa_{s_0}^*)(v_{s_0}^H,\,\kappa_{s_0})\le M_{s_0}^\lambda (\widetilde v,\,\widetilde \kappa,\,\widetilde \varkappa).
\end{align*}
Necessarily, it follows that  
$(\widetilde v,\,\widetilde \kappa,\,\widetilde \varkappa)=(v_{s_0}^*,\,\kappa_{s_0}^*,\,\varkappa_{s_0}^*)(v_{s_0}^H,\,\kappa_{s_0})$
and $(v_{s_0}^*,\,\kappa_{s_0}^*,\,\varkappa_{s_0}^*)(v_{s_0}^H,\,\kappa_{s_0})\rest{(s_0,\,s)}
=(v_{s_0,\,s}^\bullet,\,\kappa_{s_0,\,s}^\bullet,\,\varkappa_{s_0,\,s}^\bullet)(v_{s_0}^H,\,\kappa_{s_0}).$
\end{proof}

\subsection{Linear feedback stabilization for the extended system}\label{S:feedback_extsys}
In this section we prove the following Theorem~\ref{T:feedback} which says
that if we see~$\varkappa$ as our control and $(v,\,\kappa)$ as the state in~\eqref{sys-v-ext}, then $\varkappa$ can be taken in linear feedback form.

\begin{theorem} \label{T:feedback}
Given $\hat u\in\WW^{\rm st}$ and $\lambda>0$, let $M=\overline C_{\left[|\hat
u|_\WW,\lambda\right]}\in\N$ be the integer constructed in
Theorem~\ref{T:ex.st-ct.lin}. Then there is a family of
operators $\KK_{\hat u}^{\lambda,\,s}\in\LL(H\times \NN^\perp\to\NN^\perp)$ such that the
following properties hold:
\begin{enumerate}
 \renewcommand{\theenumi}{{\sf\roman{enumi}}} 
 \renewcommand{\labelenumi}{{\sf(\roman{enumi}).}} 
 \item\label{feedTcont}
The function $s\mapsto \KK_{\hat u}^{\lambda,\,s}$, $s\in{[0,\,+\infty)}$, is continuous in the weak operator topology,
and it holds~$\norm{\KK_{\hat u}^{\lambda,\,s}}{\LL(H\times \NN^\perp\to\,\NN^\perp)}
\le\overline C_{\left[|\hat u|_{\WW^{\rm st}},\,\lambda,\,\frac{1}{\lambda}\right]}$.
 \item\label{feedTest}
For any given $s_0\ge0$ and $(v_{s_0}^H,\,\kappa_{s_0})\in H\times \NN^\perp$, the solution of the system~\eqref{sys-v-ext}
with~$\varkappa=\KK_{\hat u}^{\lambda,\,t}(\Pi v(t),\,\kappa(t))$
exists, in~$\R_{s_0}\times\Omega$, and satisfies the estimate
\begin{equation}\label{estTfeed} 
\hspace*{-2.0em}\norm{\ex^{\frac{\lambda}{2}(\Bigcdot-a)}(v,\,\kappa)}{W(\R_a,\,H^1_{\diver}(\Omega,\,\R^3),\,
H^{-1}(\Omega,\,\R^3))\times H^1(\R_a,\,\NN^\perp)}^2
\le \overline C_{\left[|\hat u|_{\WW^{\rm st}},\,\lambda,\,\frac{1}{\lambda}\right]} \norm{(\Pi v(a),\,\kappa(a)}{H\times\NN^\perp}^2
\end{equation}
for all $a\ge s_0$. 
\end{enumerate}
\end{theorem}

\begin{proof}
We organize the proof into~\ref{stwot} main steps. In Step~\ref{stKKT} we use a Lagrange multiplier approach to derive two key optimality conditions
for the minimizer $(v_{s_0,\,s}^\bullet,\,\kappa_{s_0,\,s}^\bullet,\,\varkappa_{s_0,\,s}^\bullet)$ of Problem~\ref{Pb:Ns}. In Step~\ref{stPropOpti}, we use those
conditions and the dynamic programming principle to find the linear feedback rule.
In Step~\ref{stdecay} we show the
uniqueness of the solution under the feedback controller, prove the bound of the feedback operator norm, for each instant of time,
and prove estimate~\eqref{estTfeed}.
Finally in Step~\ref{stwot} we prove
the continuity of the time-dependent family of feedback operators in the weak operator topology.
%
\begin{enumerate}[noitemsep,topsep=5pt,parsep=5pt,partopsep=0pt,leftmargin=0em]%
\renewcommand{\theenumi}{{\sf\arabic{enumi}}} 
 \renewcommand{\labelenumi}{} 
\item \textcircled{\bf s}~Step~\theenumi:\label{stKKT} {\em Karush--Kuhn--Tucker Theorem.}
First we notice that, considering as usual $H$ as a pivot space, we can extend the projection $\Pi\colon L^2(\Omega,\,\R^3)\to H$
to a mapping $\Pi\colon H^{-1}(\Omega,\,\R^3)\to V'$ by simply setting
$\langle\Pi f,\,u\rangle_{V',\,V}\coloneqq \langle f,\,u\rangle_{H^{-1}(\Omega,\,\R^3),\,H^1_0(\Omega,\,\R^3)}$
for all $u\in V$ (cf.~beginning of Proof of Theorem~5.3 in~\cite{Rod15-cocv}).
Then, we define the spaces
\begin{align*}
\XX_0&\coloneqq W((s_0,\,s),\,H^1_{\diver}(\Omega,\,\R^3),\,H^{-1}(\Omega,\,\R^3))
\times H^1((s_0,\,s),\,\NN^\bot)\times L^2((s_0,\,s),\,\NN^\bot),\\
\XX_{\rm S}&\coloneqq\left\{(v,\,\kappa)\mid (v,\,\kappa,\,0)\in\XX_0,\; v(s_0)\in\aA_{\Xi_1},
\mbox{ and }v(s_0)=\FF_\aA^{-1}(\Pi v(s),\,Q^M_f\kappa(s))\right\},\\
\XX&\coloneqq \XX_{\rm S}\times L^2((s_0,\,s),\,\NN^\perp),\\
\YY&\coloneqq H\times \NN^\perp\times L^2((s_0,\,s),\,V')\times L^2((s_0,\,s),\,\NN^\perp)\times G^{1}_{{\rm av},\,0}((s_0,\,s),\,\Gamma),
\end{align*}
where we denote
\[
 G^{1}_{{\rm av},\,0}((s_0,\,s),\,\Gamma)\coloneqq \{u\in G^{1}_{{\rm av}}((s_0,\,s),\,\Gamma)\mid
u(s_0)\cdot\nnn=0\}
\]
(where we may understand $u(s_0)\cdot\nnn$ as $(E_1u)(s_0)\cdot\nnn$, with the extension~$E_1$ as in Section~\ref{sS:funspaces}).

Next we define the affine operator $F\colon \XX\to\YY$, by
\begin{equation*}
F(v,\,\kappa,\,\varkappa)\coloneqq 
\bigl(\Pi v(s_0)-v_{s_0}^H,\,\kappa(s_0)-\kappa_{s_0},\,\Pi(\p_tv+\BB(\hat u)v-\nu\Delta v),\,
\p_t\kappa-\varkappa,\,v\rest\Gamma-\Xi\kappa\bigr).
\end{equation*}
We show now that the derivative $\ed F$ of $F$, $\ed F(v,\,\kappa,\,\varkappa)=F(v,\,\kappa,\,\varkappa)+(v_{s_0}^H,\,\kappa_{s_0},\,0,\,0,\,0)$, is surjective. Indeed
let us be given $(y_1,\,y_2,\,y_3,\,y_4,\,y_5)\in\YY$.
Setting $\kappa(t)=y_2$, for all~$t\in(s_0,\,s)$, $\varkappa=-y_4$, and taking $K$ in system~\eqref{sys-vg} to be the inclusion
from $G^{1}_{{\rm av},\,0}((s_0,\,s),\,\Gamma)$ into $G^1_{{\rm av}}((s_0,\,s),\,\Gamma)$:
$\eta\mapsto \zeta=K\eta=\eta$, then
by~\cite[Theorem~2.11]{Rod15-cocv}, we
have that system~\eqref{sys-vg}, in $(s_0,\,s)\times\Omega$, has a weak solution~$v$
for the data $(v_0,\,g,\,\zeta)=(\FF_\aA^{-1}(y_1,\,Q^M_fy_2),\,y_3,\,y_5+\Xi y_2)$, because necessarily
$(\FF_\aA^{-1}(y_1,\,Q^M_fy_2)-E_1(y_5+\Xi y_2)(s_0))\cdot\nnn=-(\Xi Q^M_l y_2)\cdot\nnn=0$.

Therefore, we find 
$\ed F(v,\,\kappa,\,\varkappa)=(y_1,\,y_2,\,y_3,\,y_4,\,y_5)$
and can conclude that $\ed F$ is surjective.

We see that the minimizer of the cost $N_{s_0,\,s}^\lambda$, in Problem~\ref{Pb:Ns}, is
a minimizer in $\XX$, that is $(v_{s_0,\,s}^\bullet,\,\kappa_{s_0,\,s}^\bullet,\,\varkappa_{s_0,\,s}^\bullet)\in\XX$, and satisfies
$F(v_{s_0,\,s}^\bullet,\,\kappa_{s_0,\,s}^\bullet,\,\varkappa_{s_0,\,s}^\bullet)=0$.
By the Karush--Kuhn--Tucker Theorem (e.g., see~\cite[Theorem~A.1]{BarRodShi11}),
there exists a Lagrange multiplier 
\[
(\mu_s,\,\theta_s,\,q_s,\,\varpi_s,\,\gamma_s)\in\YY',
\]
with $\YY'=
H\times \NN^\perp\times L^2((s_0,\,s),\,V)\times L^2((s_0,\,s),\,\NN^\perp)\times G^1_{{\rm av},\,0}((s_0,\,s),\,\Gamma)'
$,
such that
\[
\ed N_{s_0,\,s}^\lambda(v_{s_0,\,s}^\bullet,\,\kappa_{s_0,\,s}^\bullet,\,\varkappa_{s_0,\,s}^\bullet)+
(\mu_s,\,\theta_s,\,q_s,\,\varpi_s,\,\gamma_s)\circ \ed F(v_{s_0,\,s}^\bullet,\,\kappa_{s_0,\,s}^\bullet,\,\varkappa_{s_0,\,s}^\bullet)=0.
\]
Hence, for all $(z,\,\varsigma)\in \XX_{\rm S}$ and all $\xi\in
L^2((s_0,\,s),\,\NN^\bot)$ we have
\begin{align}
0&=2\int_{s_0}^s\ex^{\lambda t}(\Pi v_{s_0,\,s}^\bullet,\,\Pi z)_{H}\,\ed t
+2\int_{s_0}^s\ex^{\lambda t}(\kappa_{s_0,\,s}^\bullet(t),\,\varsigma(t))_{\NN^\perp}\,\ed t\label{s.z}\\
&\quad+2\left(R_{\hat u}^{\lambda,\,s}(\Pi v_{s_0,\,s}^\bullet(s),\,\kappa_{s_0,\,s}^\bullet(s)),\,
(\Pi z(s),\,\varsigma(s))\right)_{H\times \NN^\perp}\notag\\
&\quad+\left((\Pi z(s_0),\,\varsigma(s_0)),\,(\mu_s,\,\theta_s)\right)_{H\times \NN^\perp}
+\int_{s_0}^s ( \varpi_s,\,\p_t\varsigma)_{\NN^\perp}\,\ed t\notag\\
&\quad
+\int_{s_0}^s\langle \Pi(\p_tz+\BB(\hat u)z-\nu\Delta z),\, q_s\rangle_{V',\,V}\,\ed t
+\langle \gamma_s,\,z\rest\Gamma-\Xi\varsigma\rangle_{(G^{1,\,s}_{0,\,\Gamma})',\,G^{1,\,s}_{0,\,\Gamma}}\;,\notag\\
0&=2\int_{s_0}^s\ex^{\lambda t}(\varkappa_{s_0,\,s}^\bullet,\,\xi)_{\NN^\bot}\,\ed t
+\int_{s_0}^s ( \varpi_s,\,-\xi)_{\NN^\perp}\,\ed t\;,\label{s.xi}
\end{align}
where for simplicity we denote $G^{1,\,s}_{0,\,\Gamma}\coloneqq G^1_{{\rm av},\,0}((s_0,\,s),\,\Gamma)$.

\item \textcircled{\bf s}~Step~\theenumi:\label{stPropOpti} {\em properties of the optimal triple.}
Letting $z$ run over all $z\in W((s_0,\,s),\,V,\,V')$, with $z(s_0)=z(s)=0$, and taking $\varsigma=0$, then from~\eqref{s.z} and from
$H^{-1}(\Omega,\,\R^3)=V'\oplus\{\nabla p\mid p\in L^2(\Omega,\,\R)\}$ (see, e.g.,~\cite[Chapter~1, Section~1.4]{Temam01}),
we can see that for some $p_{q_s}\in L^2((s_0,\,s),\,L^2(\Omega,\,\R))$,
\begin{equation}\label{eq.qs}
-\p_tq_s-\nu\Delta q_s+\BB^\ast(\hat u)q_s+\nabla p_{q_s}+2\ex^{\lambda t}\Pi v_{s_0,\,s}^\bullet(t)=0
\end{equation}
where $\BB^\ast(\hat u)$ is the formal adjoint to $\BB(\hat u)$: defined for given $q\in
V$ and $v\in H^{1}_{\diver}(\Omega,\,\R^3)$ by
$(v,\,\BB^\ast(\hat u)q)_{L^2(\Omega,\,\R^3)}
\coloneqq \langle \BB(\hat u)v,\,q\rangle_{H^{-1}(\Omega,\,\R^3),\,H^1_0(\Omega,\,\R^3)}$. Further, we suppose that we have fixed an appropriate
choice for the pressure function~$p_{q_s}$ (cf.~\cite[Section~3.2]{Rod15-cocv}). To fix ideas, let us choose $p_{q_s}$ satisfying $\int_\Omega p_{q_s}\,\ed\Omega=0$.

From $\Pi v_{s_0,\,s}^\bullet\in L^2((s_0,s),H)$ and
$q_s\in
L^2((s_0,s),V)$, it follows $\p_tq_s\in
L^2((s_0,s),V')$, 
in particular $q_s\in C([s_0,\,s],\,H)$ (cf.~\cite[Chapter~1, Theorem~3.1]{LioMag72-I}). Using
again~\eqref{s.z}, with arbitrary $z\in W((s_0,\,s),\,V,\,V')\subset C([s_0,\,s],\,H)$ and $\varsigma=0$, we derive that $q_s(s_0)=\mu_s$ and
\begin{equation} \label{qs1}
q_s(s)=-2 R_{\hat u,\,1}^{\lambda,\,s}(\Pi v_{s_0,\,s}^\bullet(s),\,\kappa_{s_0,\,s}^\bullet(s))\in H
\end{equation}
where, recalling that~$R_{\hat u}^{\lambda,\,s}(\Pi v_{s_0,\,s}^\bullet(s),\,\kappa_{s_0,\,s}^\bullet(s))\in H\times\NN^\perp$, we
define
\begin{equation}\label{R=R1R2}
\begin{array}{ll}
\left(R_{\hat u,\,1}^{\lambda,\,s}(\Pi v_{s_0,\,s}^\bullet(s),\,\kappa_{s_0,\,s}^\bullet(s)),\,
R_{\hat u,\,2}^{\lambda,\,s}(\Pi v_{s_0,\,s}^\bullet(s),\,\kappa_{s_0,\,s}^\bullet(s))\right)
\coloneqq R_{\hat u}^{\lambda,\,s}(\Pi v_{s_0,\,s}^\bullet(s),\,\kappa_{s_0,\,s}^\bullet(s));\\
 \left(R_{\hat u,\,1}^{\lambda,\,s}(\Pi v_{s_0,\,s}^\bullet(s),\,\kappa_{s_0,\,s}^\bullet(s)),\,
 R_{\hat u,\,2}^{\lambda,\,s}(\Pi v_{s_0,\,s}^\bullet(s),\,\kappa_{s_0,\,s}^\bullet(s))\right)\in H\times\NN^\perp.
 \end{array}
\end{equation}
Since $\Pi v_{s_0,\,s}^\bullet\in L^2((s_0,s),H)$, by standard arguments we can prove that
$q_s\in W((s_0,s-2\epsilon),\,\D(L),\,H)$ for any $0<\epsilon<\frac{s-s_0}{2}$; taking,  again in~\eqref{s.z}, arbitrary
$z$ supported in $[s_0+\epsilon,\,s-\epsilon]$ with $z\in W((s_0,\,s),\,H^1_{\diver}(\Omega,\,\R^3),\,H^{-1}(\Omega,\,\R^3))$, and $\varsigma=0$,
we obtain, using~\eqref{eq.qs},
\begin{equation}\label{qp-gamma}
0=\left(\langle\nnn\cdot\nabla\rangle q_s-p_{q_s}\nnn,\,z\rest\Gamma\right)_{L^2((s_0,\,s),\,L^2(\Gamma,\,\R^3))}
+\langle \gamma_s,\,z\rest\Gamma\rangle_{(G^{1,\,s}_{0,\,\Gamma})',\,G^{1,\,s}_{0,\,\Gamma}}.
\end{equation}

On the other hand, relation~\eqref{s.xi} implies that
\begin{equation}\label{varpikappa}
\varpi_s=2\ex^{\lambda t}\varkappa_{s_0,\,s}^\bullet,
\end{equation}
which shows us that $\varpi_s$ is independent of $s$ (or, more precisely $\varpi_s(t)$ is independent of $s$, for $t\in[s_0,\,s]$), because
by Lemma~\ref{L:MsNs} we have $(v_{s_0,\,s}^\bullet,\,\kappa_{s_0,\,s}^\bullet,\,\varkappa_{s_0,\,s}^\bullet)(v_{s_0}^H,\,\kappa_{s_0})
=(v_{s_0}^*,\,\kappa_{s_0}^*,\,\varkappa_{s_0}^*)(v_{s_0}^H,\,\kappa_{s_0})\rest{(s_0,\,s)}$.

Now taking $z=0$ and $\varsigma\in H^1_0((s_0,\,s),\,\NN^\perp)$, from~\eqref{s.z} we obtain
\begin{align}
0&=2\int_{s_0}^s\ex^{\lambda t}(\kappa_{s_0,\,s}^\bullet(t),\,\varsigma(t))_{\NN^\perp}\,\ed t
+\int_{s_0}^s ( \varpi_s,\,\p_t\varsigma)_{\NN^\perp}\,\ed t
+\langle \gamma_s,\,-\Xi\varsigma\rangle_{(G^{1,\,s}_{0,\,\Gamma})',\,G^{1,\,s}_{0,\,\Gamma}}\;\notag
\end{align}
and, introducing the adjoint $\Xi^*\colon (G^{1,\,s}_{0,\,\Gamma})'\to\NN^\bot$ of $\Xi\colon \NN^\bot\to G^{1,\,s}_{0,\,\Gamma}$, defined by
\begin{equation*}
(\Xi^*\gamma,\,k)_{\NN^\bot}\coloneqq \langle\gamma,\,\Xi k\rangle_{(G^{1,\,s}_{0,\,\Gamma})',\,G^{1,\,s}_{0,\,\Gamma}}\,,\quad\mbox{for all }k\in\NN^\bot, 
\end{equation*}
we arrive to
\begin{equation}\label{Poisson_k_1}
 \Xi^*\gamma_s=2\ex^{\lambda t}\kappa_{s_0,\,s}^\bullet
-\p_t\varpi_s
=2\ex^{\lambda t}(1-\p_t\p_t)\kappa_{s_0,\,s}^\bullet-2\lambda\ex^{\lambda t}\varkappa_{s_0,\,s}^\bullet.
\end{equation}
Then, we see that $\gamma_s$ is independent of $s$. Since the space $\{f\in G^{1,\,s}_{0,\,\Gamma}\mid \supp(f)\subset(s_0,\,s)\}\supset
\{f\in C^\infty([s_0,\,s],\,H^1(\Gamma,\,\R^3))\mid\supp(f)\subset(s_0,\,s)\}$
is dense in $L^2((s_0,\,s),\,L^2(\Gamma,\,\R^3))$, by using~\eqref{qp-gamma}
we obtain that $\langle\nnn\cdot\nabla\rangle q_s-p_{q_s}\nnn=\gamma_s$ is independent of $s$.
Here, the last equality is to be understood in~$(G^{1,\,s}_{0,\,\Gamma})'$.

Suppose now that we are given another $s_1$ with $s_0<s<s_1$. Then, the difference
$(\bar q,\,p_{\bar q})\coloneqq (q_s-q_{s_1},\,p_{q_s}-p_{q_{s_1}})$ satisfies
\begin{equation*}
 -\p_t\bar q-\nu\Delta \bar q+\BB^\ast(\hat u)\bar q+\nabla p_{\bar q}=0
\end{equation*}
in $(s_0,\,s)\times\Omega$.
From the observability inequality in~\cite[Inequality~(3.4)]{Rod15-cocv}, we obtain
\[
|\bar q(s-2\epsilon)|_H^2\leq C|p_{\bar q}\nnn
-\langle\nnn\cdot\nabla\rangle \bar q|_{G^1((s-2\epsilon,\,s),\,\Gamma)'}^2
=0\mbox{ for all }0<\epsilon\le\textstyle\frac{s-s_0}{2}.
\]
Since $\bar q\in C([s_0,\,s],\,H)$, it follows that $\bar q=0$ in $[s_0,\,s]$.
In particular, we can conclude that~$q_s$ and~$\nabla p_{q_s}$ do not depend on~$s$ and that
\begin{equation}\label{q_inds1}
-2 R_{\hat u,\,1}^{\lambda,\,s}(\Pi v_{s_0,\,s}^\bullet(s),\,\kappa_{s_0,\,s}^\bullet(s))=q_s(s)=q_{s_1}(s)\in V,\quad
\mbox{for all }s\in{[s_0,\,s_1]}
\end{equation}
which, in turn, implies that $q_s\in W((s_0,\,s),\,\D(L),\,H)$ and 
$p_{q_s}\in L^2((s_0,\,s),\,H^1(\Omega,\,\R))$ (cf.~\cite[Section~3.2]{Rod15-cocv}). In particular, if we choose~$p_{q_s}$ satisfying
$\int_\Gamma p_{q_s}\,\ed\Gamma=0$ then  the identity
$\gamma_s=\langle\nnn\cdot\nabla\rangle q_s-p_{q_s}\nnn$ is meaningful in the subset
$L^2((s_0,\,s),\,H^{\frac{1}{2}}_{\rm av}(\Gamma,\,\R^3))\subset(G^{1,\,s}_{0,\,\Gamma})'$.

Now,  we observe that the function $\eta\coloneqq \kappa_{s_0,\,s}^\bullet
+\frac{1}{s-s_0}\left((\Bigcdot-s)\kappa_{s_0,\,s}^\bullet(s_0)
-(\Bigcdot-s_0)\kappa_{s_0,\,s}^\bullet(s)\right)$ is in $H^1_0((s_0,\,s),\,\NN^\bot)$ and, from~\eqref{Poisson_k_1}, solves
the Poisson equation
\begin{align}
 (1-\Delta_t)\eta&=\textstyle\frac{1}{2}\ex^{-\lambda \Bigcdot}\Xi^*\gamma_s+\lambda\varkappa_{s_0,\,s}^\bullet
 +\frac{1}{s-s_0}\left((\Bigcdot-s)\kappa_{s_0,\,s}^\bullet(s_0)
-(\Bigcdot-s_0)\kappa_{s_0,\,s}^\bullet(s)\right),\label{poisson.eq}
\end{align}
with~$\Delta_t\coloneqq\p_t\p_t$. Since the right hand side is in $L^2((s_0,\,s),\,\NN^\bot)$, by standard arguments (for elliptic equations) it follows that
$\eta\in H^1_0((s_0,\,s),\,\NN^\bot)\cap H^2((s_0,\,s),\,\NN^\bot)$, which implies that 
$\kappa_{s_0,\,s}^\bullet\in H^2((s_0,\,s),\,\NN^\bot)$.

Whence, taking~$z=0$ and an arbitrary $\varsigma\in H^1((s_0,\,s),\,\NN^\bot)$ in~\eqref{s.z}, with~$\varsigma(s_0)=0$, 
and using~\eqref{R=R1R2} and~\eqref{Poisson_k_1} we can arrive to
$0=2R_{\hat u,\,2}^{\lambda,\,s}(\Pi v_{s_0,\,s}^\bullet(s),\,\kappa_{s_0,\,s}^\bullet(s))
+\varpi_s(s)$. Thus, from~\eqref{varpikappa}, we obtain
$\varkappa_{s_0,\,s}^\bullet(s)=-\ex^{-\lambda s}R_{\hat u,\,2}^{\lambda,\,s}(\Pi v_{s_0,\,s}^\bullet(s),\,\kappa_{s_0,\,s}^\bullet(s))$, and
Lemma~\ref{L:MsNs} leads to
\begin{equation}\label{linfeed-ext}
\varkappa_{s_0}^*(s)=\KK_{\hat u}^{\lambda,\,s}(\Pi v_{s_0}^*(s),\,\kappa_{s_0}^*(s)),\quad\mbox{with }
\KK_{\hat u}^{\lambda,\,s}\coloneqq-\ex^{-\lambda s}R_{\hat u,\,2}^{\lambda,\,s}\,.
\end{equation}

%
\item \textcircled{\bf s}~Step~\theenumi:\label{stdecay} {\em uniqueness, operator norm, and exponential decay estimates.}
We know that given $(v_{s_0}^H,\,\kappa_{s_0})\in H\times\NN^\perp$ the pair~$(v,\,\kappa)=(v_{s_0}^*,\,\kappa_{s_0}^*)(v_{s_0}^H,\,\kappa_{s_0})$
solves~\eqref{sys-v-ext}, in~$\R_{s_0}\times\Omega$, with the feedback control~\eqref{linfeed-ext}:
\begin{subequations}\label{sys-v-ext-feed}
\begin{align}
 \p_t v+\BB(\hat u)v
 -\nu\Delta v+\nabla p_v&=0,& \p_t\kappa &=\KK_{\hat u}^{\lambda,\,\Bigcdot}(\Pi v,\,\kappa),\label{sys-v-ext-feed-dyn}\\
 \diver v &=0,& v\rest \Gamma &=\Xi\kappa, \label{sys-v-ext-feed-divbdry}\\
 v(s_0)&=\FF_\aA^{-1}(v_{s_0}^H,\, Q^M_f\kappa_{s_0}),& \kappa(s_0)&=\kappa_{s_0},\label{sys-v-ext-feed-ic}
\end{align}
\end{subequations}
This solution is unique: indeed,  if $(\widetilde v,\,\widetilde \kappa)$ is another solution
with $(\Pi\widetilde v,\,\widetilde \kappa)(s_0)=(v_{s_0}^H,\,\kappa_{s_0})$,
then $(d^v,\,d^\kappa)=(v-\widetilde v,\,\kappa-\widetilde \kappa)$ solves
\begin{align*}
 \p_t d^v+\BB(\hat u)d^v
 -\nu\Delta d^v+\nabla p_{d^v}&=0,&\quad \diver d^v &=0,\hspace*{3em} d^v\rest \Gamma= \Xi d^\kappa,\\
 (d^v(s_0),\,d^\kappa(s_0))&=(0,\,0),& d^\kappa &=\int_{s_0}^t\KK_{\hat u}^{\lambda,\,r}(\Pi d^v(r),\,d^\kappa(r))\,\ed r, \\
 \end{align*}
and, from~\eqref{normofR}, \eqref{R=R1R2}, and~\eqref{linfeed-ext}, we have
\begin{equation}\label{Opnorm}
|\KK_{\hat u}^{\lambda,\,r}(\Pi d^v(r),\,d^\kappa(r))|_{\NN^\perp}^2
\le \overline C_{\left[|\hat u|_{\WW^{\rm st}},\,\lambda,\,\frac{1}{\lambda}\right]}
|(\Pi d^v(r),\,d^\kappa(r))|_{H\times\NN^\perp}^2.
\end{equation}
Then, for the function $z(t)\coloneqq
\norm{\KK_{\hat u}^{\lambda,\,t}(\Pi d^v(t),\,d^\kappa(t))}{\NN^\perp}^2$,
using Theorem~\ref{T:exist-wsol-v} and $d^v(s_0)=0$, 
\begin{align*}
z(r)
&\le\overline C_{\left[|\hat u|_{\WW^{\rm st}},\,\lambda,\,\frac{1}{\lambda}\right]}
\left(|d^v(r)|_{L^2_{\diver}(\Omega,\,\R^3)}^2+|d^\kappa(r))|_{\NN^\perp}^2\right)\\
&\le
\overline C_{\left[|\hat u|_{\WW^{\rm st}},\,\lambda,\,\frac{1}{\lambda}\right]}
\left(\left|\int_{s_0}^{\Bigcdot} \KK_{\hat u}^{\lambda,\,s}(\Pi d^v(s),\,d^\kappa(s))\,\ed s\right|_{H^1((0,\,r),\,\NN^\perp)}^2
+(r-s_0)\int_{s_0}^rz(t)\,\ed t\right)\\
&=
\overline C_{\left[|\hat u|_{\WW^{\rm st}},\,\lambda,\,\frac{1}{\lambda}\right]}
\int_{s_0}^r\left(\left|\int_{s_0}^{t} \KK_{\hat u}^{\lambda,\,s}(\Pi d^v(s),\,d^\kappa(s))\,\ed s\right|_{\NN^\perp}^2
+(1 +r-s_0)z(t)\right)\,\ed t\\
&\le
\overline C_{\left[|\hat u|_{\WW^{\rm st}},\,\lambda,\,\frac{1}{\lambda}\right]}
\int_{s_0}^r\left((t-s_0)\int_{s_0}^{t} z(s)\,\ed s
+(1 +r-s_0)z(t)\right)\,\ed t.
\end{align*}
Hence, for any $\widehat r>s_0$ and $r\in[s_0,\,\widehat r]$ we have
\begin{align*}
&\quad z(r)
\le
\int_{s_0}^r\overline C_{\left[|\hat u|_{\WW^{\rm st}},\,\lambda,\,\frac{1}{\lambda},\,\widehat r\right]}\left( z(t)
+\int_{s_0}^t z(s)\,\ed s\right)\,\ed t\notag
\end{align*}
and from the~Gronwall--Bellman inequality (see~\cite[Theorem~1]{Pachpatte73}) we obtain
$z(r)=0$. Since $\widehat r\ge s_0$ can be taken arbitrary, we can conclude that $z=0$ which implies $d^\kappa=0$.
Using again Theorem~\ref{T:exist-wsol-v}, we arrive to $d^v=0$.

The uniqueness of the solution of system~\eqref{sys-v-ext-feed}, implies that we need to prove the
estimate~\eqref{estTfeed}, for the optimal trajectory $(v_{s_0}^*,\,\kappa_{s_0}^*)(v_{s_0}^H,\,\kappa_{s_0})$ solving Problem~\ref{Pb:Ms}
(with~$s_0$ in the role of~$s$). In this case, from~\eqref{R-opt.norm} and for any $a\ge s_0$,
we already know that $M_a^\lambda(v_{s_0}^*\rest{\R_a},\,\kappa_{s_0}^*\rest{\R_a},\,\varkappa_{s_0}^*\rest{\R_a})\le
\overline C_{\left[|\hat u|_\WW,\lambda,\,\frac{1}{\lambda}\right]}\ex^{\lambda a}\norm{(\Pi v_{s_0}^*(a),\,\kappa_{s_0}^*(a))}{H\times\NN^\perp}^2$.
Hence from~\eqref{vHtovH1-ts}, 
\begin{align*}
&\quad|\ex^{\frac{\lambda}{2} (\Bigcdot-a)}v_{s_0}^*|_{W(\R_a,\,H^1_{\diver}(\Omega,\,\R^3),\,H^{-1}(\Omega,\,\R^3))}^2
+|\ex^{\frac{\lambda}{2} (\Bigcdot-a)}\kappa_{s_0}^*|_{H^1(\R_a,\,\NN^\bot)}^2\\
&\le
\overline C_{\left[|\hat u|_{\WW^{\rm st}},\,\lambda\right]}\left(|v_{s_0}^*(a)|_{L^2_{\diver}(\Omega,\,\R^3)}^2
+|\ex^{\frac{\lambda}{2} (\Bigcdot-a)}\Pi v_{s_0}^*|_{L^2(\R_a,\,H)}^2
+|\ex^{\frac{\lambda}{2} (\Bigcdot-a)}\kappa_{s_0}^*|_{H^1(\R_a,\,\NN^\bot)}^2\right)\\
&\le
\overline C_{\left[|\hat u|_{\WW^{\rm st}},\,\lambda\right]}\left(\norm{(\Pi v_{s_0}^*(a),\,Q^M_f\kappa_{s_0}^*(a))}{H\times\NN^\perp}^2
+\ex^{-\lambda a}(1+\lambda^2)M_a^\lambda(v_{s_0}^*\rest{\R_a},\,\kappa_{s_0}^*\rest{\R_a},\,\varkappa_{s_0}^*\rest{\R_a})\right)\\
&\le \overline C_{\left[|\hat u|_\WW,\lambda,\,\frac{1}{\lambda}\right]}
\norm{(\Pi v_{s_0}^*(a),\,\kappa_{s_0}^*(a))}{H\times\NN^\perp}^2,
\end{align*}
that is, estimate~\eqref{estTfeed} holds.

%
\item \textcircled{\bf s}~Step~\theenumi:\label{stwot} {\em continuity in the weak operator topology.}
By definition,~$\{\KK_{\hat u}^{\lambda,\,s}\mid s\in[0,\,+\infty)\}\subseteq\LL(H\times\NN^\perp\to\NN^\perp)$
does depend continuously 
on~$s$ in the weak operator topology if for all $(w,\,\xi)\in (H\times\NN^\perp)\times \NN^\perp$ and $s_0\in[0,\,+\infty)$,
$(\KK_{\hat u}^{\lambda,\,s}w,\,\xi)_{\NN^\perp}$ goes to $(\KK_{\hat u}^{\lambda,\,s_0}w,\,\xi)_{\NN^\perp}$
as $s$ goes to $s_0$; with $s\in[0,\,+\infty)$.
Recalling~\eqref{R=R1R2} and~\eqref{linfeed-ext}, we see that this property holds if it holds for the
family~$\{R_{\hat u}^{\lambda,\,s}\mid s\in[0,\,+\infty)\}\subseteq\LL(H\times\NN^\perp\to H\times\NN^\perp)$.

Given a pair~$(w^1,\,w^2)\in (H\times\NN^\perp)\times (H\times\NN^\perp)$, we can write 
$2(R_{\hat u}^{\lambda,\,s}w^1,\,w^2)_{H\times\NN^\perp}
=(R_{\hat u}^{\lambda,\,s}(w^1+w^2),\,w^1+w^2)_{H\times\NN^\perp}
-(R_{\hat u}^{\lambda,\,s}w^1,\,w^1)_{H\times\NN^\perp}-(R_{\hat u}^{\lambda,\,s}w^2,\,w^2)_{H\times\NN^\perp}$, and we see that
it suffices to prove that
\begin{equation} \label{wot}
(R_{\hat u}^{\lambda,\,s}w,w)_{H\times\NN^\perp}\to(R_{\hat u}^{\lambda,\,s_0}w,w)_{H\times\NN^\perp}
\quad\mbox{as $s\to s_0$, for any $w\in H\times\NN^\perp$},
\end{equation}
or equivalently (cf.~\cite[Chapter~IX, Proposition~1.3.(e)]{Conway85}) that
\begin{equation} \label{wot.seq}
(R_{\hat u}^{\lambda,\,s_0\pm\delta_n}w,w)_{H\times\NN^\perp}\to(R_{\hat u}^{\lambda,\,s_0}w,w)_{H\times\NN^\perp}\mbox{ as }n\to+\infty,
\end{equation}
for any sequence $(\delta_n)_{n\in\N}$ of real numbers, with $0< \delta_n<1$, $\delta_n\to 0$,
and any $w\in H\times\NN^\perp$ (still, with $s=s_0\pm\delta_n\geq0$).

We consider separately two cases: $s\searrow s_0$ and $s\nearrow s_0$. To shorten a little the notation we denote
$\mathbf z^*_{s_0}(w)\coloneqq(\Pi v^*_{s_0},\,\kappa^*_{s_0})(w)$, for a given $w\in H\times\NN^\perp$.

%
\noindent{\sf (a)} {\em The case $s\searrow s_0$.}
If $s=s_0+\delta_n$, we write 
\begin{align}
&\quad\left(R_{\hat u}^{\lambda,\,s_0+\delta_n}w,w\right)_{H\times\NN^\perp}\label{wot.se1}\\
&=\left(R_{\hat u}^{\lambda,\,s_0+\delta_n}(w-\mathbf z^*_{s_0}(w)(s_0+\delta_n)),w\right)_{H\times\NN^\perp}
+\left(R_{\hat u}^{\lambda,\,s_0+\delta_n}
\mathbf z^*_{s_0}(w)(s_0+\delta_n),\,w\right)_{H\times\NN^\perp}.\notag
\end{align}
Rewriting the last term in equation~\eqref{wot.se1} as
$
\left(R_{\hat u}^{\lambda,\,s_0+\delta_n}
\mathbf z_{s_0}^*(w)(s_0+\delta_n),\,w-\mathbf z_{s_0}^*(w)(s_0+\delta_n)\right)_{H\times\NN^\perp}
+\left(R_{\hat u}^{\lambda,\,s_0+\delta_n}
\mathbf z_{s_0}^*(w)(s_0+\delta_n),\,\mathbf z_{s_0}^*(w)(s_0+\delta_n)\right)_{H\times\NN^\perp}
$
which, recalling~Lemma~\ref{L:MsNs} (the dynamical programming principle), can be written as
\begin{align*}
&\left(R_{\hat u}^{\lambda,\,s_0+\delta_n}
\mathbf z_{s_0}^*(w)(s_0+\delta_n),\,w-\mathbf z_{s_0}^*(w)(s_0+\delta_n)\right)_{H\times\NN^\perp}
+\left(R_{\hat u}^{\lambda,\,s_0}
w,\,w\right)_{H\times\NN^\perp}\\
-&|\ex^{\frac{\lambda}{2} \Bigcdot}\Pi v_{s_0}^*(w)|_{L^2((s_0,\,s_0+\delta_n),\,H)}^2
-|\ex^{\frac{\lambda}{2} \Bigcdot}\kappa_{s_0}^*(w)|_{L^2((s_0,\,s_0+\delta_n),\,\NN^\bot)}^2
-|\ex^{\frac{\lambda}{2} \Bigcdot}\varkappa_{s_0}^*(w)|_{L^2((s_0,\,s_0+\delta_n),\,\NN^\bot)}^2.
\end{align*}
Using the self-adjointness of~$R_{\hat u}^{\lambda,\,s_0+\delta_n}$, we arrive to
\begin{align}
&\quad\left(R_{\hat u}^{\lambda,\,s_0+\delta_n}w,w\right)_{H\times\NN^\perp}-\left(R_{\hat u}^{\lambda,\,s_0}
w,\,w\right)_{H\times\NN^\perp}\label{wot.se2}\\
&= \left(R_{\hat u}^{\lambda,\,s_0+\delta_n}(w-\mathbf z_{s_0}^*(w)(s_0+\delta_n)),w
+\mathbf z_{s_0}^*(w)(s_0+\delta_n)\right)_{H\times\NN^\perp}\notag\\
&\quad
-|\ex^{\frac{\lambda}{2} \Bigcdot}\Pi v_{s_0}^*(w)|_{L^2((s_0,\,s_0+\delta_n),\,H)}^2
-|\ex^{\frac{\lambda}{2} \Bigcdot}\kappa_{s_0}^*(w)|_{L^2((s_0,\,s_0+\delta_n),\,\NN^\bot)}^2
-|\ex^{\frac{\lambda}{2} \Bigcdot}\varkappa_{s_0}^*(w)|_{L^2((s_0,\,s_0+\delta_n),\,\NN^\bot)}^2.\notag
\end{align}
Now, for the first term on the right-hand side of~\eqref{wot.se2}, we have
\begin{align*}
&\quad\left(R_{\hat u}^{\lambda,\,s_0+\delta_n}(w-\mathbf z_{s_0}^*(w)(s_0+\delta_n)),w
+\mathbf z_{s_0}^*(w)(s_0+\delta_n)\right)_{H\times\NN^\perp}\\
&\leq\overline C_{\left[|\hat u|_{\WW^{\rm st}},\,\lambda,\,\frac{1}{\lambda}\right]}
\ex^{\lambda(s_0+\delta_n)}
|w-\mathbf z_{s_0}^*(w)(s_0+\delta_n)|_{H\times\NN^\perp}|w
+\mathbf z_{s_0}^*(w)(s_0+\delta_n)|_{H\times\NN^\perp}
\end{align*}
and, from
the continuity of $\mathbf z_{s_0}^*(w)(t)$ in the time variable $t$, we can conclude that the this term goes to zero with $\delta_n$.
On the other hand, the remaining (last three) terms on the right hand side of~\eqref{wot.se2} also go to zero with $\delta_n$,
because the triple $(v_{s_0}^*,\,\kappa_{s_0}^*,\,\varkappa_{s_0}^*)(w)$ do not depend on $\delta_n$. Therefore we can derive
\begin{equation}\label{wot.se}
\lim_{n\to+\infty}\left(R_{\hat u}^{\lambda,\,s_0+\delta_n}w,w\right)_{H\times\NN^\perp}
=\left(R_{\hat u}^{\lambda,\,s_0}w,w\right)_{H\times\NN^\perp}. 
\end{equation}

%
\noindent{\sf (b)} {\em The case $s\nearrow s_0$.} Though we will follow the same idea, this case carries a few additional
difficulties. If $s=s_0-\delta_n$, we write 
\begin{multline*}
\left(R_{\hat u}^{\lambda,\,s_0-\delta_n}w,w\right)_{H\times\NN^\perp}
=|\ex^{\frac{\lambda}{2} \Bigcdot}\Pi v_{s_0-\delta_n}^*(w)|_{L^2((s_0-\delta_n,\,s_0),\,H)}^2
+|\ex^{\frac{\lambda}{2} \Bigcdot}\kappa_{s_0-\delta_n}^*(w)|_{L^2((s_0-\delta_n,\,s_0),\,\NN^\bot)}^2\\
+|\ex^{\frac{\lambda}{2} \Bigcdot}\varkappa_{s_0-\delta_n}^*(w)|_{L^2((s_0-\delta_n,\,s_0),\,\NN^\bot)}^2
+\left(R_{\hat u}^{\lambda,\,s_0}\mathbf z_{s_0-\delta_n}^*(w)(s_0),\mathbf z_{s_0-\delta_n}^*(w)(s_0)\right)_{H\times\NN^\perp}
\end{multline*}
and, writing $\mathbf z_{s_0-\delta_n}^*(w)(s_0)=\mathbf z_{s_0-\delta_n}^*(w)(s_0)-w+w$ in the last term, we obtain
\begin{align}
&\quad\left(R_{\hat u}^{\lambda,\,s_0-\delta_n}w,w\right)_{H\times\NN^\perp}-\left(R_{\hat u}^{\lambda,\,s_0}
w,\,w\right)_{H\times\NN^\perp}
=|\ex^{\frac{\lambda}{2} \Bigcdot}\Pi v_{s_0-\delta_n}^*(w)|_{L^2((s_0-\delta_n,\,s_0),\,H)}^2\label{wot.ne1}\\
&+|\ex^{\frac{\lambda}{2} \Bigcdot}\kappa_{s_0-\delta_n}^*(w)|_{L^2((s_0-\delta_n,\,s_0),\,\NN^\bot)}^2
+|\ex^{\frac{\lambda}{2} \Bigcdot}\varkappa_{s_0-\delta_n}^*(w)|_{L^2((s_0-\delta_n,\,s_0),\,\NN^\bot)}^2\notag\\
&+\left(R_{\hat u}^{\lambda,\,s_0}(\mathbf z_{s_0-\delta_n}^*(w)(s_0)-w),\,\mathbf z_{s_0-\delta_n}^*(w)(s_0)-w\right)_{H\times\NN^\perp}\notag\\
&+2\left(R_{\hat u}^{\lambda,\,s_0}w,\,\mathbf z_{s_0-\delta_n}^*(w)(s_0)-w\right)_{H\times\NN^\perp}\notag
\end{align}
Now by~\ref{estTfeed} we have, in particular,
\begin{align}
&\quad|\ex^{\frac{\lambda}{2} \Bigcdot}v_{s_0-\delta_n}^*(w)|_{C([s_0-\delta_n,\,s_0-\delta_n+1],\,\aA_{\Xi_1})}^2
+|\ex^{\frac{\lambda}{2} \Bigcdot}\kappa_{s_0-\delta_n}^*(w)|_{C([s_0-\delta_n,\,s_0-\delta_n+1],\,\NN^\perp)}^2\notag\\
&\le
\overline C_{\left[|\hat u|_{\WW^{\rm st}},\,\lambda,\,\frac{1}{\lambda},\,s_0\right]}
|w|_{H\times\NN^\perp}^2\label{wot.ne1a}
\end{align}
which implies, since $\delta_n<1$, 
\begin{align}
&\quad|\ex^{\frac{\lambda}{2} \Bigcdot}\Pi v_{s_0-\delta_n}^*(w)|_{L^2((s_0-\delta_n,\,s_0),\,H)}^2
+|\ex^{\frac{\lambda}{2} \Bigcdot}\kappa_{s_0-\delta_n}^*(w)|_{L^2((s_0-\delta_n,\,s_0),\,\NN^\perp)}^2\notag\\
&\le \overline C_{\left[|\hat u|_{\WW^{\rm st}},\,\lambda,\,\frac{1}{\lambda},\,s_0\right]}
|w|_{H\times\NN^\perp}^2\delta_n.\label{wot.ne1b}
\end{align}
On the other hand, from~$\varkappa_{s_0-\delta_n}^*(w)=\KK_{\hat u}^{\lambda,\,\Bigcdot}(\Pi v_{s_0-\delta_n}^*(w)
 ,\,\kappa_{s_0-\delta_n}^*(w))$ and~\eqref{wot.ne1b},
\begin{align}
 &\quad|\ex^{\frac{\lambda}{2}\Bigcdot}\varkappa_{s_0-\delta_n}^*(w)|_{L^2((s_0-\delta_n,\,s_0),\,\NN^\bot)}^2\notag\\
 &\le\overline C_{\left[|\hat u|_{\WW^{\rm st}},\,\lambda,\,\frac{1}{\lambda}\right]}
 \left(|\ex^{\frac{\lambda}{2} \Bigcdot}\Pi v_{s_0-\delta_n}^*(w)|_{L^2((s_0-\delta_n,\,s_0),\,H)}^2
+|\ex^{\frac{\lambda}{2} \Bigcdot}\kappa_{s_0-\delta_n}^*(w)|_{L^2((s_0-\delta_n,\,s_0),\,\NN^\perp)}^2\right)\notag\\
&\le \overline C_{\left[|\hat u|_{\WW^{\rm st}},\,\lambda,\,\frac{1}{\lambda},\,s_0\right]}
|w|_{H\times\NN^\perp}^2\delta_n.\label{wot.ne1c}
\end{align}
and so, from~\eqref{wot.ne1}, \eqref{wot.ne1b}, and~\eqref{wot.ne1c}, we arrive to
\begin{align}
&\quad\lim_{n\to+\infty}\left(\hspace*{-.3em}\left(R_{\hat u}^{\lambda,\,s_0-\delta_n}w,w\right)_{\hspace*{-.2em}H\times\NN^\perp}\hspace*{-.7em}
-\left(R_{\hat u}^{\lambda,\,s_0}(\mathbf z_{s_0-\delta_n}^*(w)(s_0)+w),\,\mathbf z_{s_0-\delta_n}^*(w)(s_0)-w\right)_{\hspace*{-.2em}H\times\NN^\perp}
\hspace*{-.2em}\right)\label{wot.ne2}\\
&=\left(R_{\hat u}^{\lambda,\,s_0}
w,\,w)\right)_{\hspace*{-.2em}H\times\NN^\perp}.\notag
\end{align}

Denoting $(w^H,\,w^\NN)\coloneqq w\in H\times\NN^\perp$, since $V$ is dense in $H$, there
exists a sequence $w_n=(w_n^V,\,w^\NN)_{n\in\N}$, with $w_n^V\in V$,
such that $|w_n-w|_{H\times\NN^\perp}$ goes to zero as $n$ goes to $+\infty$. Thus, writing $w=w-w_n+w_n$ and using~\eqref{wot.ne1a}
we obtain that
\begin{align*}
&\quad\;|\mathbf z_{s_0-\delta_n}^*(w)(s_0)-w|_{H\times\NN^\perp}\\
&\le  |\mathbf z_{s_0-\delta_n}^*(w-w_n)(s_0)-w+w_n|_{H\times\NN^\perp}
+ |\mathbf z_{s_0-\delta_n}^*(w_n)(s_0)-w_n|_{H\times\NN^\perp}\\
&\le \overline C_{\left[|\hat u|_{\WW^{\rm st}},\,\lambda,\,\frac{1}{\lambda},\,s_0\right]}
|w-w_n|_{H\times\NN^\perp}+|\mathbf z_{s_0-\delta_n}^*(w_n)(s_0)-w_n|_{H\times\NN^\perp},
\end{align*}
and 
\begin{align*}
&\quad \left|\left(R_{\hat u}^{\lambda,\,s_0}(\mathbf z_{s_0-\delta_n}^*(w)(s_0)+w),\,
\mathbf z_{s_0-\delta_n}^*(w)(s_0)-w\right)_{H\times\NN^\perp}\right|_\R\\
 &\le \overline C_{\left[|\hat u|_{\WW^{\rm st}},\,\lambda,\,\frac{1}{\lambda},\,s_0\right]}
 |w|_{H\times\NN^\perp}\left(|w-w_n|_{H\times\NN^\perp}+|\mathbf z_{s_0-\delta_n}^*(w_n)(s_0)-w_n|_{H\times\NN^\perp}\right),
\end{align*}
which implies that
\begin{align}
&\quad \lim_{n\to+\infty}\left|\left(R_{\hat u}^{\lambda,\,s_0}(\mathbf z_{s_0-\delta_n}^*(w)(s_0)+w),\,
\mathbf z_{s_0-\delta_n}^*(w)(s_0)-w\right)_{H\times\NN^\perp}\right|_\R\label{wot.ne3}\\
 &\le \overline C_{\left[|\hat u|_{\WW^{\rm st}},\,\lambda,\,\frac{1}{\lambda},\,s_0\right]}
 |w|_{H\times\NN^\perp}\lim_{n\to+\infty}|\mathbf z_{s_0-\delta_n}^*(w_n)(s_0)-w_n|_{H\times\NN^\perp}.\notag
\end{align}
It turns out that $v_n=v_n(\Bigcdot)\coloneqq v_{s_0-\delta_n}^*(w_n)(\Bigcdot)-w_n^V$ solves~\eqref{sys-vg}, in
$(s_0-\delta_n,\,s_0)\times\Omega$, with $g=\BB(\hat u)w_n^V-\nu\Delta w_n^V$, $K=\Xi$,
$\eta=\kappa_{s_0-\delta_n}^*(w_n)(t)$, and $v_n(s_0-\delta_n)=0$. From Theorem~\ref{T:exist-wsol-v}, and
$|g(t)|_{H^{-1}(\Omega,\,\R^3)}^2
\le C(|\hat u(t)|_{L^\infty_{\diver}(\Omega,\,\R^3)}^2+1)|w_n^V|_{H^1_{\diver}(\Omega,\,\R^3)}^2$,
we obtain that
\begin{align*}
&\quad|v_{s_0-\delta_n}^*(w_n)(s_0)-w_n^V|_{L^2_{\diver}(\Omega,\,\R^3)}^2=|v_n(s_0)|_{L^2_{\diver}(\Omega,\,\R^3)}^2\\
&\le
\overline C_{\left[|\hat u|_{\WW^{(s_0-\delta_n,\,s_0)|\rm wk}}\right]}
\left(|w_n^V|_{L^2((s_0-\delta_n,\,s_0),\,V)}^2
+|\kappa_{s_0-\delta_n}^*(w_n)|_{H^1((s_0-\delta_n,\,s_0),\,\NN^\bot)}^2\right);
\end{align*}
on the other hand we also have $\kappa_{s_0-\delta_n}^*(w_n)(\Bigcdot)-w^\NN
=\int_{s_0-\delta_n}^{\Bigcdot}\varkappa_{s_0-\delta_n}^*(w_n)(t)\,\ed t$ and
\[
\norm{\kappa_{s_0-\delta_n}^*(w_n)(s_0)-w^\NN}{\NN^\perp}^2
\le\delta_n\int_{s_0-\delta_n}^{s_0}\norm{\varkappa_{s_0-\delta_n}^*(w_n)(t)}{\NN^\perp}^2\,\ed t.
\]
Then, by~\eqref{wot.ne1b}, \eqref{wot.ne1c}, and~\eqref{wot.ne3}, we obtain
\begin{align}
&\quad\lim_{n\to+\infty}\left|\left(R_{\hat u}^{\lambda,\,s_0}(\mathbf z_{s_0-\delta_n}^*(w)(s_0)+w),\,
\mathbf z_{s_0-\delta_n}^*(w)(s_0)-w\right)_{H\times\NN^\perp}\right|_\R\notag\\
 &\le \overline C_{\left[|\hat u|_{\WW^{\rm st}},\,\lambda,\,\frac{1}{\lambda},\,s_0\right]}
 |w|_{H\times\NN^\perp}\lim_{n\to+\infty}\delta_n\left(|w_n^V|_{V}^2+|w^\NN|_{\NN^\perp}^2\right)\notag\\
& =\overline C_{\left[|\hat u|_{\WW^{\rm st}},\,\lambda,\,\frac{1}{\lambda},\,s_0\right]}
 |w|_{H\times\NN^\perp}\lim_{n\to+\infty}\delta_n|w_n^V|_{V}^2.\label{wot.ne4}
\end{align}
Recall that the inclusion $V\subset H$ is continuous. Notice that $|w_n^V|_{V}$ will go to $+\infty$ with $n$,
if $w^H\in H\setminus V$.
We need to show that the sequence $(w_n^V)_{n\in\N}$ can be chosen such that
$\delta_n|w_n^V|_{V}^2\to 0$, as $n\to+\infty$.

We know that the sequence $\Pi_{n+1}w^H\in H_{n+1}\subset V$ converges to $w^H$, in~$H$, where $\Pi_N\colon H\to H_N$ is
the orthogonal projection
in $H$ onto the space $H_N$ spanned by the first $N$ eigenfunctions of the Stokes operator~$L$ (see~\eqref{HN}) but, we
have no guarantee that $\delta_n|\Pi_{n+1}\bar w|_V^2\to 0$, that is, $\Pi_{n+1} w^H$ may be not a good choice for~$w^V_n$. 
Next we show that this issue can be overcome by somehow ``slowing down'' the convergence of $\Pi_{n+1}w^H$ to $w^H$, in~$H$:
we define the sequence $N\colon \N\to\N_0$, $n\mapsto N_n$ by
\[
\begin{cases}
N_n\coloneqq 1,\quad&\mbox{ if }\quad \frac{1}{2\alpha_{2}}<\delta_n\\
N_n\coloneqq j,\quad&\mbox{ if }\quad j\ge 2\;\mbox{ and }\;\frac{1}{(j+1)\alpha_{j+1}}<\delta_n\le \frac{1}{j\alpha_j}
\end{cases};
\]
where~$(\alpha_j)_{j\in\N_0}$ is the nondecreasing sequence of (repeated) eigenvalues of the Stokes operator (cf. Section~\ref{sS:ct-space}).
We can see that $N_n$ goes to $+\infty$ with $n$, because $\delta_n$ goes to~$0$.
Hence, setting $w_n^V\coloneqq \Pi_{N_n} w^H$, we have that
$w_n^V\to w^H$, in~$H$, and $|w_n^V|_{V}^2\le C|\nabla w_n^V|_{L^2(\Omega,\,\R^9)}^2=C(L w_n^V,\,w_n^V)_H
\le C\alpha_{N_n}|w_n^V|_H^2$ and, then $\delta_n|w_n^V|_{V}^2\le C\delta_n\alpha_{N_n}|w_n^V|_H^2
\le\frac{1}{N_n}C|w^H|_H^2$,
for all $N_n\ge 2$;
thus $\delta_n|w_n^V|_{V}^2\to0$ as $n\to+\infty$.
From~\eqref{wot.ne2} and~\eqref{wot.ne4}, we arrive to
\begin{equation}\label{wot.ne}
\lim_{n\to+\infty}\left(R_{\hat u}^{\lambda,\,s_0-\delta_n}w,w\right)_{H\times\NN^\perp}
=\left(R_{\hat u}^{\lambda,\,s_0}w,\,w)\right)_{H\times\NN^\perp}.
\end{equation}

Finally, from~\eqref{wot.se} and~\eqref{wot.ne}, we see that~\eqref{wot.seq} holds. That is, the family
$R_{\hat u}^{\lambda,\,s}$ depends continuously on $s$, in the weak operator topology.
\end{enumerate}

This ends the proof of Theorem~\ref{T:feedback}.
\quad\end{proof}

%
\subsection{Miscellaneous remarks}\label{sS:miscRems}
As we have said in the Introduction, looking for feedback finite-dimensional controllers supported in a small subset of the boundary,
is motivated by the
importance of such controllers in applications.
We give a few remarks concerning this point. 

\subsubsection{Dimension of the controller} The range $E_M$
of the controller depends only on the norm $|\hat u|_{\WW^{\rm st}}$ of the 
targeted solution~$\hat u$, and the feedback rule
depends on time. We do not address here the problem of finding an estimate for $M$, that is of crucial importance for application purposes
(e.g., numerical simulations). For internal controls, this problem has been started in~\cite{KroRod15,KroRod-ecc15} in the simpler case of the {1D}
Burgers system in a bounded interval $(0,\,L)$, and estimates on the number $M$ of needed controls is given that depend
exponentially in~$M_{\rm ref}\coloneqq(\nu^{-2}|\hat u|_{\WW}^2+\nu^{-1}\lambda)^\frac{1}{2}$
in the general case, and that are proportional to
$M_{\rm ref}$ in the case of no constraint on the support of the control, with $\WW=L^\infty(\R_0,\,L^\infty((0,\,L),\,\R))$.
However also in~\cite{KroRod15} the results of numerical simulations
suggest that it might be sufficient to take  $M$ proportional to
$M_{\rm ref}$ also in the general case. Following~\cite{KroRod15}, we see that in the particular case, the estimate
for~$M$ is derived from an inequality like $\alpha_M\ge \nu C M_{\rm ref}^2$, where $\alpha_M$ is the $M$th eigenvalue
of the Laplacian operator. Hence we may (perhaps) conjecture that
in the {$d$D} case, $d\in\{2,\,3\}$, it should be enough to take a number $M$ of internal controls proportional to
$(\nu^{-2}|\hat u|_{\WW^{\rm st}}^2+\nu^{-1}\lambda)^\frac{d}{2}$, because $\alpha_M\sim \nu C_1M^\frac{2}{d}$ (cf.~\cite{Ilyin09}).
Does the conjecture hold true? Can we derive similar estimates for the case of the boundary controls we treat here?
These questions will be addressed in future works.

Recall that, in the case of stationary~$\hat u$, for example
in~\cite{BarbuTri04,BadTakah11,RaymThev10}, we can find rather sharp estimates, though $M$ does depend on~$\hat u$ and not only in the
norm~$|\hat u|_{\WW^{\rm st}}$. The method cannot be (at least not straightforwardly) used in the nonstationary case.

\subsubsection{Lyapunov functions} Once a feedback control is constructed, it is easy to find
a time-dependent Lyapunov function for the problem in question.
Indeed, the functional
\[
\varPhi(r,w)=\int_r^\infty |\FF_\aA\sS_{r,\,t}w|_{H\times\NN^\perp}^2\,\ed t
\]
decays along the trajectories of~\eqref{sys-v-ext-feed}, where $\sS_{r,\,t}w$ denotes the solution $(v_r^*,\,\kappa_r^*)(t)$
of~\eqref{sys-v-ext-feed}, in $\R_r\times\Omega$, with the initial condition $(\Pi v_r^*,\,\kappa_r^*)(r)=w$:
we may write
\begin{align*}
\varPhi(r,(\Pi v_0^*,\,\kappa_0^*)(r))&=\int_r^\infty |\FF_\aA\sS_{r,\,\tau}(\Pi v_0^*,\,\kappa_0^*)(r)|_{H\times\NN^\perp}^2\,\ed \tau
=\int_r^\infty |\FF_\aA\sS_{0,\,\tau}(\Pi v_0^*,\,\kappa_0^*)(0)|_{H\times\NN^\perp}^2\,\ed \tau\\
&=\int_r^\infty |(\Pi v_0^*,\,\kappa_0^*)(\tau)|_{H\times\NN^\perp}^2\,\ed \tau,
\end{align*}
from which, together with~\eqref{estTfeed}, we can obtain
\begin{align*}
&\quad\left.\frac{\ed}{\ed s}\right|_{s=r}\varPhi(s,(\Pi v_0^*,\,\kappa_0^*)(s))
=-|(\Pi v_0^*,\,\kappa_0^*)(r)|_{H\times\NN^\perp}^2\\
&\le -\overline C_{\left[|\hat u|_{\WW^{\rm st}},\,\lambda,\,\frac{1}{\lambda}\right]}^{\;-1}
|\ex^{\frac{\lambda}{2}(\Bigcdot-r)}(v_0^*,\,\kappa_0^*)|_{W(\R_r,\,H^1_{\diver}(\Omega,\,\R^3),\,
H^{-1}(\Omega,\,\R^3))\times H^1(\R_r,\,\NN^\perp)}^2\\
&\le -\overline C_{\left[|\hat u|_{\WW^{\rm st}},\,\lambda,\,\frac{1}{\lambda}\right]}^{\;-1}
C|\ex^{\frac{\lambda}{2}(\Bigcdot-r)}(\Pi v_0^*,\,\kappa_0^*)|_{L^2(\R_r,\,H\times\NN^\perp))}^2
\le -\overline C_{\left[|\hat u|_{\WW^{\rm st}},\,\lambda,\,\frac{1}{\lambda}\right]}^{\;-1}
C\varPhi(r,(\Pi v_0^*,\,\kappa_0^*)(r)).
\end{align*}

Another Lyapunov function is the ``cost to go'' from time $t=r$ onwards, that is $\varPsi(r,\,w)=(R^r w,\,w)_{H\times\NN^\perp}$, where
$R^r\coloneqq R_{\hat u}^{\lambda,\,r}$ is the operator defining the optimal
cost. Indeed, from the dynamical principle we have
\begin{align*}
\varPsi(r,\,(\Pi v_0^*,\,\kappa_0^*)(r))
&=\varPsi(0,\,(\Pi v_0^*,\,\kappa_0^*)(0))\\
&\quad-\norm{\ex^{\frac{\lambda}{2}\Bigcdot}\Pi v_0^*}{L^2((0,\,r),\,H)}^2
-\norm{\ex^{\frac{\lambda}{2}\Bigcdot}\kappa_0^*}{L^2((0,\,r),\,\NN^\perp)}^2-\norm{\ex^{\frac{\lambda}{2}\Bigcdot}\p_t\kappa_0^*}{L^2((0,\,r),\,\NN^\perp)}^2
\end{align*}
which implies, using~\eqref{normofR} and~\eqref{linfeed-ext},
\begin{align*}
&\quad\left.\frac{\ed}{\ed s}\right|_{s=r}\varPsi(s,\,(\Pi v_0^*,\,\kappa_0^*)(s))
=-\ex^{\lambda r}\norm{\Pi v_0^*(r)}{H}^2
-\ex^{\lambda r}\norm{\kappa_0^*(r)}{\NN^\perp}^2
-\ex^{\lambda r}\norm{\textstyle\frac{\ed}{\ed s}\rest{s=r} \kappa_0^*(s)}{\NN^\perp}^2\\
&\le-\ex^{\lambda r}\norm{(\Pi v_0^*,\,\kappa_0^*)(r)}{H\times\NN^\perp}^2
\le -\overline C_{\left[|\hat u|_{\WW^{\rm st}},\,\lambda,\,\frac{1}{\lambda}\right]}^{\;-1}
(R^r(\Pi v_0^*,\,\kappa_0^*)(r),\,(\Pi v_0^*,\,\kappa_0^*)(r))_{H\times\NN^\perp}\\
&=-\overline C_{\left[|\hat u|_{\WW^{\rm st}},\,\lambda,\,\frac{1}{\lambda}\right]}^{\;-1}
\varPsi(r,\,(\Pi v_0^*,\,\kappa_0^*)(r)).
\end{align*}

It is, however, difficult to write down the functions $\varPhi$ and $\varPsi$ in a more explicit form.

\subsubsection{Riccati equation} In the case of internal controls it was shown
in~\cite[Remark~3.11(b)]{BarRodShi11} that the operator defining the optimal cost, and from which we can obtain the feedback law,
satisfies a suitable differential Riccati equation. 
For applications (e.g., simulations) it is important to have such equation at our disposal; 
for example, we refer to~\cite{KroRod15}
where the optimal feedback rule has been
obtained by solving (numerically) a similar differential Riccati equation,
in the simpler case of the {1D} Burgers system and internal controls.

It turns out that also in our case the operator $R\coloneqq R^t\coloneqq R_{\hat u}^{\lambda,\,t}$,
from which we can obtain the boundary feedback rule, satisfies a differential Riccati equation.

First of all we recall~\eqref{R=R1R2} and rewrite $R=(R_1,\,R_2)$, where
\[
 R_1\in\LL(H\times\NN^\perp\to H),\quad R_2\in\LL(H\times\NN^\perp\to \NN^\perp),
\]
and for any two pairs $(w^1,\,w^2)$ and $(z^1,\,z^2)$ in~$H\times\NN^\perp$, we have
\[
 (R(w^1,\,w^2),\,(z^1,\,z^2))_{H\times\NN^\perp}=(R_1(w^1,\,w^2),\,z^1)_{H}+(R_2(w^1,\,w^2),\,z^2)_{\NN^\perp}.
\]

Now observe that we can further decompose $(R_1,\,R_2)$ and define
$R_{1,1}\in \LL(H)$, $R_{1,2}\in \LL(\NN^\perp\to H)$, $R_{2,1}\in \LL(H\to \NN^\perp)$, and $R_{2,2}\in \LL(\NN^\perp)$  by
\begin{align*}
    (z^1,\,R_{1,1}w^1)_H&\coloneqq ((z^1,\,0),\,R(w^1,\,0))_{H\times\NN^\perp}\,,
    &(z^1,\,R_{1,2}w^2)_H&\coloneqq((z^1,\,0),\,R(0,\,w^2))_{H\times\NN^\perp}\,,\\
    (z^2,\,R_{2,1}w^1)_{\NN^\perp}&\coloneqq((0,\,z^2),\,R(w^1,\,0))_{H\times\NN^\perp}\,,
    &(z^2,\,R_{2,2}w^2)_{\NN^\perp}&\coloneqq((0,\,z^2),\,R(0,\,w^2))_{H\times\NN^\perp}\,.
\end{align*}
It follows that we can write
\[
R=\begin{bmatrix}
    R_{1,1}&R_{1,2}\\R_{2,1}&R_{2,2}
   \end{bmatrix},\quad (R(w^1,\,w^2),\,(z^1,\,z^2))_{H\times\NN^\perp}\eqqcolon \begin{bmatrix} z^1 & z^2\end{bmatrix}\begin{bmatrix}
    R_{1,1}&R_{1,2}\\R_{2,1}&R_{2,2}
   \end{bmatrix}\begin{bmatrix}w^1\\ w^2\end{bmatrix}
\]
with $z^1R_{1,j}w^j\coloneqq (z^1,\,R_{1,j}w^j)_H$ and $z^2R_{2,j}w^j\coloneqq (z^2,\,R_{2,j}w^j)_{\NN^\perp}$.

The operator $\RRR\coloneqq \ex^{-\lambda t}R$ satisfies, for all $t\in\R_0$
\begin{equation} \label{riccati}
\dot\RRR-\RRR\AAA- \AAA^*\RRR-\RRR\BBB\BBB^*\RRR +\CCC=0
\end{equation}
with
\[
\AAA\coloneqq\begin{bmatrix}L_{\hat u(t)}& L_{\hat u(t)}^\FF\\ 0& 0\end{bmatrix}-\frac{\lambda}{2}\begin{bmatrix}1&0\\ 0& 1\end{bmatrix};
\quad\BBB\coloneqq\begin{bmatrix}0&0\\ 0& 1\end{bmatrix};
\quad \CCC\coloneqq\begin{bmatrix}1&0\\ 0& 1\end{bmatrix};
 \]
where $L_{\hat u(t)}$ stands for 
the Oseen--Stokes operator
\[
 L_{\hat u(t)}\colon H^1_{\diver}(\Omega,\,\R^3)\to V',\quad
v\mapsto \Pi(-\nu\Delta v+\BB(\hat u(t))v),
\]
and $L_{\hat u(t)}^\FF$ is defined by
\[
 L_{\hat u(t)}^\FF\colon \NN^\perp\to H,\quad
\kappa\mapsto L_{\hat u(t)}\FF_\aA^{-1}(0,\,Q^M_f\kappa).
\]

Observe that, from Lemmas~\ref{L:Ms} and~\ref{L:MsNs}, we have
\[
(R^s(\Pi v_0^*,\,\kappa_0^*)(s),\,(\Pi v_0^*,\,\kappa_0^*)(s))_{H\times\NN^\perp}
=M^\lambda_s(v_0^*,\,\kappa_0^*,\,\p_t\kappa_0^*)
\]
from which we can derive that, formally,
\begin{align*}
&\quad\;(\p_s\rest{s=t} R^s(\Pi v_0^*,\,\kappa_0^*)(t),\,(\Pi v_0^*,\,\kappa_0^*)(t))_{H\times\NN^\perp}\\
&+(R^t\p_s\rest{s=t} (\Pi v_0^*,\,\kappa_0^*)(s),\,(\Pi v_0^*,\,\kappa_0^*)(t))_{H\times\NN^\perp}
+(R^t(\Pi v_0^*,\,\kappa_0^*)(t),\,\p_s\rest{s=t} (\Pi v_0^*,\,\kappa_0^*)(s))_{H\times\NN^\perp}\\
=&-\left|\ex^{\frac{\lambda}{2}t}\Pi v_0^*(t)\right|_{H}^2
-\left|\ex^{\frac{\lambda}{2}t}\kappa_0^*(t)\right|_{\NN^\bot}^2
-\left|\ex^{\frac{\lambda}{2}t}\varkappa_0^*(t)\right|_{\NN^\bot}^2 
\end{align*}
where~$\varkappa_0^*=\p_t\kappa_0^*$.
From~$\p_s\rest{s=t} (\Pi v_0^*,\,\kappa_0^*)(s)=(-L_{\hat u(t)} v_0^*(t),\,\varkappa_0^*(t))$, recalling~\eqref{linfeed-ext} and
using~$L_{\hat u(t)} v_0^*=L_{\hat u(t)} \FF_\aA^{-1}(\Pi v_0^*,\,Q^M_f\kappa_0^*)=L_{\hat u(t)} \Pi v_0^*
+L_{\hat u(t)}\FF_\aA^{-1}(0,\,Q^M_f\kappa_0^*)$,
in the matrix form notation we can write
\begin{align*}
&\quad\,\begin{bmatrix}\Pi v_0^*& \kappa_0^*\end{bmatrix}\dot R \begin{bmatrix}\Pi v_0^*\\ \kappa_0^*\end{bmatrix}
+\begin{bmatrix}\Pi v_0^*& \kappa_0^*\end{bmatrix}R
\begin{bmatrix}-L_{\hat u}&-L_{\hat u(t)}^\FF\\ -\ex^{-\lambda\Bigcdot}R_{2,1}& -\ex^{-\lambda\Bigcdot}R_{2,2}\end{bmatrix}
\begin{bmatrix}\Pi v_0^*\\ \kappa_0^*\end{bmatrix}\\
&\quad+\begin{bmatrix}\Pi v_0^*& \kappa_0^*\end{bmatrix}
\begin{bmatrix}-(L_{\hat u})^*&-\ex^{-\lambda\Bigcdot}R_{2,1}^*\\ -(L_{\hat u(t)}^\FF)^*& -\ex^{-\lambda\Bigcdot}R_{2,2}^*\end{bmatrix}R
\begin{bmatrix}\Pi v_0^*\\ \kappa_0^*\end{bmatrix}\\
&=-\ex^{\lambda\Bigcdot}\begin{bmatrix}\Pi v_0^*& \kappa_0^*\end{bmatrix}\begin{bmatrix}1& 0\\ 0&1\end{bmatrix}
\begin{bmatrix}\Pi v_0^*\\ \kappa_0^*\end{bmatrix}
-\ex^{-\lambda\Bigcdot}\begin{bmatrix}\Pi v_0^*& \kappa_0^*\end{bmatrix}
\begin{bmatrix}0&R_{2,1}^*\\ 0& R_{2,2}^*\end{bmatrix}
\begin{bmatrix}0& 0\\ R_{2,1}& R_{2,2}\end{bmatrix}
\begin{bmatrix}\Pi v_0^*\\ \kappa_0^*\end{bmatrix}.
\end{align*}

Therefore, we can conclude that
\begin{align}
&\quad\,\dot R -R\begin{bmatrix}L_{\hat u}&L_{\hat u(t)}^\FF\\ 0& 0\end{bmatrix}-\begin{bmatrix}(L_{\hat u})^*&0\\ (L_{\hat u(t)}^\FF)^*& 0\end{bmatrix}R
+\ex^{\lambda\Bigcdot}\begin{bmatrix}1& 0\\ 0&1\end{bmatrix}\notag\\
&=\ex^{-\lambda\Bigcdot}R\begin{bmatrix}0&0\\ R_{2,1}& R_{2,2}\end{bmatrix}
+\ex^{-\lambda\Bigcdot}\begin{bmatrix}0&R_{2,1}^*\\ 0& R_{2,2}^*\end{bmatrix}R
-\ex^{-\lambda\Bigcdot}
\begin{bmatrix}0&R_{2,1}^*\\ 0& R_{2,2}^*\end{bmatrix}
\begin{bmatrix}0& 0\\ R_{2,1}& R_{2,2}\end{bmatrix}.\label{Ricc1}
\end{align}

Using the symmetry of~$R$, we have that~$R_{2,1}^*=R_{1,2}$ and $R_{2,2}^*=R_{2,2}$, and we can derive
\begin{align*}
&\quad\ex^{-\lambda\Bigcdot}R\begin{bmatrix}0&0\\ R_{2,1}& R_{2,2}\end{bmatrix}
+\ex^{-\lambda\Bigcdot}\begin{bmatrix}0&R_{2,1}^*\\ 0& R_{2,2}^*\end{bmatrix}R
-\ex^{-\lambda\Bigcdot}
\begin{bmatrix}0&R_{2,1}^*\\ 0& R_{2,2}^*\end{bmatrix}
\begin{bmatrix}0& 0\\ R_{2,1}& R_{2,2}\end{bmatrix}\\
&=\ex^{-\lambda\Bigcdot}
\begin{bmatrix}0&R_{1,2}\\ 0& R_{2,2}\end{bmatrix}
\begin{bmatrix}0& 0\\ R_{2,1}& R_{2,2}\end{bmatrix}
=\ex^{-\lambda\Bigcdot} R\BBB\BBB R.
\end{align*}

Hence, from~\eqref{Ricc1} and $\BBB=\BBB^*$, it follows that
\[
\dot R -R\begin{bmatrix}L_{\hat u}&L_{\hat u(t)}^\FF\\ 0& 0\end{bmatrix}-\begin{bmatrix}(L_{\hat u})^*&0\\ (L_{\hat u(t)}^\FF)^*& 0\end{bmatrix}R
+\ex^{\lambda\Bigcdot}\CCC- \ex^{-\lambda\Bigcdot}R\BBB\BBB^* R
=0
\]
which implies that for~$\RRR=\ex^{-\lambda\Bigcdot}R$, using~$\dot\RRR=\ex^{-\lambda\Bigcdot}\dot R-\lambda \RRR$, we have the identity
\[
\dot \RRR -\RRR\begin{bmatrix}L_{\hat u}&L_{\hat u(t)}^\FF\\ 0& 0\end{bmatrix}-\begin{bmatrix}(L_{\hat u})^*&0\\ (L_{\hat u(t)}^\FF)^*& 0\end{bmatrix}\RRR
+\CCC-\ex^{-2\lambda\Bigcdot} R\BBB\BBB^* R+\lambda \RRR
=0,
\]
which is equivalent to~\eqref{riccati}.

Once we have $\RRR=\begin{bmatrix}\RRR_{1,1}&\RRR_{2,1}^*\\\RRR_{2,1}&\RRR_{2,2}\end{bmatrix}$,
the feedback control, in Theorem~\ref{T:feedback} (cf. system~\eqref{sys-v-ext-feed}, and~\eqref{linfeed-ext} ) is then given by
\[
\varkappa=\KK_{\hat u}^{\lambda,\,\Bigcdot}(\Pi v,\,\kappa)(\Bigcdot)=-\RRR_{2,1}\Pi v -\RRR_{2,2}\kappa.
\]

The structure of the feedback control is comparable with that proposed in~\cite[below Equation~(1.22)]{Badra09-cocv}, in the case of a stationary
targeted solution. We also remark that here, taking advantage of the finite-dimensionality of the controls,
we consider a simple dynamics~$\partial_t\kappa=\varkappa$ on the boundary, while in~\cite[Equation~(1.12)]{Badra09-cocv} 
a different dynamics is proposed
to deal with a larger class of initial conditions, but with controls that are not necessary finite-dimensional.

Solving~\eqref{riccati} numerically, say with finite element method as in~\cite{KroRod15}, can become a quite demanding problem
as the number~$n_p$ of mesh points increase. In the case of internal controls we will end up in solving
a matrix $n_p\times n_p$-dimensional problem.
On the other hand, in the boundary case and considering the extended system instead, the problem will not
get much worse because the dimension of the problem just increase by the number~$M$ of controls, and $M$ will (or, is expected to be) be much smaller than
$n_p$. We may expect that the numerical demand in solving the resulting matrix $(n_p+M)\times(n_p+M)$-dimensional problem is comparable
with that in solving a matrix~$n_p\times n_p$-dimensional problem.

Notice also that the ``less standard'' operator $\kappa\mapsto\FF_\aA^{-1}(0,\,Q^M_f\kappa)$, from $\NN^\perp$ to~$\aA_{\Xi_1}$ can be constructed by
solving $\widehat M=\dim Q^M_f\NN^\perp\le M$ Laplace equations of the form
\[
\Delta p_{\bar\kappa}=0,\qquad\nnn\cdot\nabla p_{\bar\kappa}=\nnn\cdot\Xi \bar\kappa,
\]
(for each $\bar\kappa\in Q^M_f\NN^\perp$) and then assigning $\FF_\aA^{-1}(0,\,Q^M_f\kappa)=\nabla p_{Q^M_f\kappa}$.

Concerning, the numerical solution of the Riccati equation
we refer the reader, in particular, to the works
in~\cite{BanschBennerSaakWech15,Benner06,BennerLaubMehrmann97}.
See also \cite{BanKunisch84,BurnsSachsZietsman08,KunkelMehrmann90}.

\section{Stabilization of the Navier--Stokes system}\label{S:nonlinear}
In this Section we prove that the feedback controller constructed in Section~\ref{S:linear-feed}, to stabilize the linear extended Oseen--Stokes
system~\eqref{sys-v-ext} to zero, also
stabilizes locally the corresponding extension of the nonlinear system~\eqref{sys-vnonlin} to zero. The main result of the paper is given in
Section~\ref{sS:mainTh}; before, due to some well known issues related with the existence and
uniqueness of solutions we need to recall some definitions
(cf.~\cite[Section~5]{Rod14-na}). 

\subsection{Solutions for the nonlinear systems}
Let $a,\,b\in\R$ be two real numbers with
$0\le a<b$.
\begin{definition}\label{D:wsol-sys-y-nl}
Given $f\in
L^2((a,\,b),\,H^{-1}(\Omega,\,\R^3))$, $k\in L^4((a,\,b),\,L^{6}(\Omega,\,\R^3))$, and $y_0\in H$, we say that $y\in
L^2((a,\,b),\,V)\cap L^\infty((a,\,b),\,H)$,
with $\p_ty\in L^1((a,\,b),\,V')$, is a weak solution for system
\begin{equation}\label{sys-y-nonl}
\begin{array}{rclcrcl}
\p_t y+\langle y\cdot\nabla\rangle y+\BB(k)y-\nu\Delta
y+\nabla p_y+f&=&0,&\quad&\diver y&=&0,\\
 y\rest \Gamma &=&0,&\quad& y(a)&=&y_0,
\end{array}
\end{equation}
in $(a,\,b)\times\Omega$, if it is a weak solution in the classical sense of~\cite[Chapter~3, Section~3.1]{Temam01}.
\end{definition}

For simplicity, and for $r>1,\;k>1$, we define the subspace
\begin{equation*}
\Theta^{r,k}_{(a,\,b)}
\coloneqq W((a,\,b),\,H^{1}_{\diver}(\Omega,\,\R^3),\,H^{-1}(\Omega,\,\R^3))\cap L^r((a,\,b),\,L^{k}(\Omega,\,R^3)).
\end{equation*}

\begin{definition}\label{D:wsol-sys-u-bdry}
Given $h\in
L^2((a,\,b),\,H^{-1}(\Omega,\,\R^3))$, $z\in\Theta^{4,6}_{(a,\,b)}$, and $u_0\in L^2_{\diver}(\Omega,\,\R^3)$,  we say that~$u$ satisfying
\begin{align*}
u&\in
L^2((a,\,b),\,H^1_{\diver}(\Omega,\,\R^3))\cap L^\infty((a,\,b),\,L^2_{\diver}(\Omega,\,\R^3)),\quad
\p_tu&\in L^1((a,\,b),\,H^{-1}(\Omega,\,\R^3)),
\end{align*}
is a weak solution for system~\eqref{sys-u-bdry},
in $(a,\,b)\times\Omega$, with
$\gamma+\zeta=z\rest{\Gamma}$ and $u(a)=u_0$, if $y=u-z$ is a weak solution for the system~\eqref{sys-y-nonl}
with $f=h+\p_tz+\langle z\cdot\nabla\rangle z-\nu\Delta z$, $k=z$, and
$y_0=u_0-z(a)\in H$.
\end{definition}

\begin{definition}\label{D:wsol-vnonlin}
Given $\hat u\in L^\infty((a,\,b),\,L^\infty_{\diver}(\Omega,\,\R^3))$, $z\in\Theta^{4,6}_{(a,\,b)}$,
and $v_0\in L^2_{\diver}(\Omega,\,\R^3)$, we say that~$v$ satisfying
\begin{align*}
v\in
L^2((a,\,b),\,H^1_{\diver}(\Omega,\,\R^3))\cap L^\infty((a,\,b),\,L^2_{\diver}(\Omega,\,\R^3)),\quad \p_tv\in L^1((a,\,b),\,H^{-1}(\Omega,\,\R^3))
\end{align*}
is a weak solution for system~\eqref{sys-vnonlin},
in $(a,\,b)\times\Omega$, with
$\zeta=z\rest{\Gamma}$ and $v(a)=v_0$, if $y=v-z$ is a weak solution for the system~\eqref{sys-y-nonl}
with $f=\p_tz+\langle z\cdot\nabla\rangle z-\nu\Delta z+\BB(\hat u)z$, $k=\hat u+z$ and
$y_0=v_0-z(a)\in H$.
\end{definition}

Analogously, we define the global solutions in $\R_0\times\Omega$: let us denote, for simplicity,
$\XX_{\rm loc}(\R_0)\coloneqq \{f\mid f\rest{(0,\,T)}\in\XX((0,\,T))\mbox{ for all } T>0\}$, where $\XX((0,\,T))$ is a suitable space of
functions defined in $(0,\,T)\subset\R_0$.

\begin{definition}\label{D:wsol-sys-u-bdry-glob}
Given $h\in
L^2_{\rm loc}(\R_0,\,H^{-1}(\Omega,\,\R^3))$, $z\in\Theta^{4,6}_{\R_0,\,{\rm loc}}$, and $u_0\in L^2_{\diver}(\Omega,\,\R^3)$,
we say that $u$ is a weak solution for system~\eqref{sys-u-bdry}--\eqref{ini-u},
in $\R_0\times\Omega$, if $u\rest{(0,\,T)}$ is a weak solution in $(0,\,T)\times\Omega$, for all $T>0$, for the same system
with the data $h\rest{(0,\,T)},\,k\rest{(0,\,T)}$.
\end{definition}

\begin{definition}\label{D:wsol-vnonlin-glob}
Given $\hat u\in L^\infty_{\rm loc}((0,\,T),\,L^\infty_{\diver}(\Omega,\,\R^3))$, $z\in\Theta^{4,6}_{(a,\,b),\,{\rm loc}}$,
and $v_0\in L^2_{\diver,\,{\rm loc}}(\Omega,\,\R^3)$, we say that $v$ is a weak solution for system~\eqref{sys-vnonlin},
in $\R_0\times\Omega$, if $v\rest{(0,\,T)}$ is a weak solution in $(0,\,T)\times\Omega$, for all $T>0$, for the same system
with the data $h\rest{(0,\,T)},\,k\rest{(0,\,T)}$.
\end{definition}

The existence of the solution in Definitions~\ref{D:wsol-sys-y-nl}, \ref{D:wsol-sys-u-bdry} and~\ref{D:wsol-vnonlin}
can be proven following classical arguments, and some continuity observations as mentioned
in~\cite[Remarks~5.1 and~5.2]{Rod14-na}. Then it also follow the existence of the solutions in
Definitions~\ref{D:wsol-sys-u-bdry-glob} and~\ref{D:wsol-vnonlin-glob}. Concerning the uniqueness, following an argument
as in the proof of Lemma~5.3 in\cite{Rod14-na}, we have the following:

\begin{lemma}\label{L:unisol}
The solution in Definitions~\ref{D:wsol-sys-y-nl}, \ref{D:wsol-sys-u-bdry} and~\ref{D:wsol-vnonlin}, is unique
if it is in $\Theta^{4,6}_{(a,\,b)}$.  The solution in Definitions~\ref{D:wsol-sys-u-bdry-glob}, and~\ref{D:wsol-vnonlin-glob},
is unique
if it is in $\Theta^{4,6}_{\R_0,\,{\rm loc}}$. 
\end{lemma}

\subsection{Main Theorem}\label{sS:mainTh}
To have enough regularity to deal with the nonlinear problems,
we take strong admissible initial conditions
in $\aA_{\Xi_2}= \{u\in H^1_{\diver}(\Omega,\,\R^3)\mid u\rest\Gamma=\Xi z\mbox{ for some }z\in\NN^\perp\}
\subset\aA_{\Xi_1}$. Recall that $\aA_{\Xi_2}$ is the set of admissible strong initial conditions
for the Oseen--Stokes system
(see Section~\ref{sS:exist-solut}). 

Now we rewrite system~\eqref{sys-vnonlin}, in $\R_0\times\Omega$
with $\zeta=\Xi\kappa$, in the extended form
\begin{subequations}\label{sys-vnonlin-ext}
\begin{align}
 \p_t v+\BB(\hat u)v+\langle v\cdot\nabla\rangle v
 -\nu\Delta v+\nabla p_v&=0,& \p_t\kappa &=\KK_{\hat u}^{\lambda,\Bigcdot}(\Pi v,\,\kappa),\label{sys-vnonlin-ext-dyn}\\
 \diver v &=0,& v\rest \Gamma &=\Xi\kappa, \label{sys-vnonlin-ext-divbdry}\\
 v(0)&=\FF_\aA^{-1}(v_0^H,\, Q^M_f\kappa_0),& \kappa(0)&=\kappa_0,\label{sys-vnonlin-ext-ic}
\end{align}
\end{subequations}
where~$(v_s^H,\,\kappa_s)$ will be taken in~$H\times \NN^\perp\times L^2(\R_s,\,\NN^\perp)$, and $\KK_{\hat u}^{\lambda,\Bigcdot}$
is as in Theorem~\ref{T:feedback}
(cf.~system~\eqref{sys-v-ext-feed}).

\begin{theorem}\label{T:ex.st-ct}
For given $\hat u\in\WW^{\rm st}$ and $\lambda>0$, let
$M=\overline C_{\left[|\hat u|_{\WW^{\rm st}},\lambda\right]}$ be as in Theorem~\ref{T:feedback}. Then there exists $\epsilon>0$ with the following property:
if 
\begin{equation}\label{st-ic-ext}
(v_0^H,\,\kappa_0)\in (H\cap H^1_{\diver}(\Omega,\,\R^3))\times\NN^\perp,\quad
\mbox{with }~\FF_\aA^{-1}(v_0^H,\, Q^M_f\kappa_0)\rest\Gamma=\Xi\kappa_0,
\end{equation}
 and $\norm{\FF_\aA^{-1}(v_0^H,\, Q^M_f\kappa_0)}{H^1_{\diver}(\Omega,\,\R^3)}<\epsilon$, then
there exists a weak
solution~$(v,\,\kappa)$ in the product space $W_{\rm loc}(\R_0,\,H^2_{\diver}(\Omega,\,\R^3),\,L^2(\Omega,\,\R^3))\times H^1_{\rm loc}(\R_0,\,\NN^\perp)$
 for system~\eqref{sys-vnonlin-ext}, which
is unique and
satisfies the inequality
\begin{equation} \label{expH1.main}
|v(t)|_{H^1_{\diver}(\Omega,\,\R^3)}^2
\le\overline C_{\left[|\hat u|_{\WW^{\rm st}},\,\lambda,\,\frac{1}{\lambda}\right]}\ex^{-\lambda t}
\left(|v_0|_{H^1_{\diver}(\Omega,\,\R^3)}^2\right), \quad t\ge0.
\end{equation}
\end{theorem}

Before the proof we need some auxiliary results.

\begin{lemma}\label{L:reg.a}
If $(v_0^H,\,\kappa_0)$ satisfies~\eqref{st-ic-ext}, the solution $v$ in Theorem~\ref{T:feedback} satisfies 
\begin{align*} 
\sup_{r\geq 0}|\ex^{\frac{\lambda}{2}\Bigcdot}v|_{W((r,\,r+1),\,H^2_{\diver}(\Omega,\,\R^3),\,
L^2(\Omega,\,\R^3))}^2
&\le \overline C_{\left[|\hat u|_{\WW^{\rm st}},\,\lambda,\,\frac{1}{\lambda}\right]}
|v(0)|_{H^1_{\diver}(\Omega,\,\R^3)}^2.
\end{align*}
\end{lemma}
\begin{proof}
We know that, taking $a=0$ in~\eqref{estTfeed},
\begin{equation}\label{feedest.a=0} 
\hspace*{-2.0em}\norm{\ex^{\frac{\lambda}{2}\Bigcdot}(v,\,\kappa)}{W(\R_0,\,H^1_{\diver}(\Omega,\,\R^3),\,
H^{-1}(\Omega,\,\R^3))\times H^1(\R_0,\,\NN^\perp)}^2
\le \overline C_{\left[|\hat u|_{\WW^{\rm st}},\,\lambda,\,\frac{1}{\lambda}\right]} \norm{(v_0^H,\,\kappa_0)}{H\times\NN^\perp}^2;
\end{equation}
since
$w\coloneqq \ex^{\frac{\lambda}{2}\Bigcdot}v$ solves~\eqref{sys-vg} with $g=-\frac{\lambda}{2}v$,
 $\eta=\ex^{\frac{\lambda}{2}\Bigcdot}\kappa=\ex^{\frac{\lambda}{2}\Bigcdot}\left(\kappa_0+\int_0^{\Bigcdot}\KK_{\hat u}^{\lambda,s}(\Pi v(s),\,\kappa(s))\,\ed s\right)$, and $K=\Xi$,
from
Theorem~\ref{T:exist-ssol-v} we have that
\begin{align*}
&\quad|w|_{W((r,\,r+1),\,H^2_{\diver}(\Omega,\,\R^3),\,L^2(\Omega,\,\R^3))}^2\\
&\le \overline C_{\left[|\hat u|_{\WW^{\rm st}}\right]}
\left(|w(r)|_{H^1_{\diver}(\Omega,\,\R^3)}^2 + \textstyle\frac{\lambda^2}{4}|v|_{L^2((r,\,r+1),\,L^2_{\diver}(\Omega,\,\R^3))}^2
+\left|\eta\right|_{H^1((r,\,r+1),\,\NN^\bot)}^2\right).
\end{align*}
On the other hand, by construction, in the proof of Theorem~\ref{T:feedback} we know that the solution coincides with the
minimizer of Problem~\ref{Pb:Ms}, with $s=0$; and,
from~\eqref{feedest.a=0} it follows that $|v|_{L^2((r,\,r+1),\,L^2_{\diver}(\Omega,\,\R^3))}^2
\le \overline C_{\left[|\hat u|_{\WW^{\rm st}},\,\lambda,\,,\,\frac{1}{\lambda}\right]}
\norm{(v_0^H,\,\kappa_0)}{H\times\NN^\perp}^2$. Therefore, using $\p_t\eta=\frac{\lambda}{2}\eta+\ex^{\frac{\lambda}{2}\Bigcdot}\p_t\kappa$,
\begin{align*}
&\quad|w|_{W((r,\,r+1),\,H^2_{\diver}(\Omega,\,\R^3),\,L^2(\Omega,\,\R^3))}^2\\
&\le \overline C_{\left[|\hat u|_{\WW^{\rm st}},\,\lambda,\,\frac{1}{\lambda}\right]}
\hspace*{-.2em}\left(|w(r)|_{H^1_{\diver}(\Omega,\,\R^3)}^2 + \norm{(v_0^H,\,\kappa_0)}{H\times\NN^\perp}^2
+M_0^\lambda\left(v,\,\kappa,\,\p_t\kappa\right)\right);
\end{align*}
and from Lemma~\ref{L:Ms} we have
\begin{align*}
&\quad|w|_{W((r,\,r+1),\,H^2_{\diver}(\Omega,\,\R^3),\,L^2(\Omega,\,\R^3))}^2
\le \overline C_{\left[|\hat u|_{\WW^{\rm st}},\,\lambda,\,\frac{1}{\lambda}\right]}
\left(|w(r)|_{H^1_{\diver}(\Omega,\,\R^3)}^2 + \norm{(v_0^H,\,\kappa_0)}{H\times\NN^\perp}^2\right)\\
&\le \overline C_{\left[|\hat u|_{\WW^{\rm st}},\,\lambda,\,\frac{1}{\lambda}\right]}
\left(\ex^{\lambda r}|v(r)|_{H^1_{\diver}(\Omega,\,\R^3)}^2 + \norm{(v_0^H,\,\kappa_0)}{H\times\NN^\perp}^2\right).
\end{align*}
In particular, 
$|w(\lfloor r\rfloor+1)|_{H^1_{\diver}(\Omega,\,\R^3)}^2
\le \overline C_{\left[|\hat u|_{\WW^{\rm st}},\,\lambda,\,\frac{1}{\lambda}\right]}
\left(\ex^{\lambda \lfloor r\rfloor}|v(\lfloor r\rfloor)|_{H^1_{\diver}(\Omega,\,\R^3)}^2 + \norm{(v_0^H,\,\kappa_0)}{H\times\NN^\perp}^2
\right)$, with $\lfloor r\rfloor\in\N$ as in~\eqref{floor}, and
\begin{align}
&\quad|w|_{W((r,\,r+1),\,H^2_{\diver}(\Omega,\,\R^3),\,L^2(\Omega,\,\R^3))}^2\le
|w|_{W((\lfloor r\rfloor,\,\lfloor r\rfloor+2),\,H^2_{\diver}(\Omega,\,\R^3),\,L^2(\Omega,\,\R^3))}^2\notag\\
&\le \overline C_{\left[|\hat u|_{\WW^{\rm st}},\,\lambda,\,\frac{1}{\lambda}\right]}
\left(|w(\lfloor r\rfloor)|_{H^1_{\diver}(\Omega,\,\R^3)}^2 +
|w(\lfloor r\rfloor+1)|_{H^1_{\diver}(\Omega,\,\R^3)}^2 + 2\norm{(v_0^H,\,\kappa_0)}{H\times\NN^\perp}^2\right)\notag\\
&\le \overline C_{\left[|\hat u|_{\WW^{\rm st}},\,\lambda,\,\frac{1}{\lambda}\right]}
\left(\ex^{\lambda \lfloor r\rfloor}|v(\lfloor r\rfloor)|_{H^1_{\diver}(\Omega,\,\R^3)}^2
+\norm{(v_0^H,\,\kappa_0)}{H\times\NN^\perp}^2\right).\label{sup_Ir_1}
\end{align}
If $\lfloor r\rfloor>0$, that is, if $r\ge1$ by Lemma~\ref{L:smooth-prop} we also have
\begin{align}
|v\lfloor r\rfloor|_{H^1_{\diver}(\Omega,\,\R^3)}^2&\le
\overline C_{\left[|\hat u|_{\WW^{\rm st}}\right]}\left(|v(\lfloor r\rfloor-1)|_{L^2_{\diver}(\Omega,\,\R^3)}^2
+|\kappa|_{H^1((\lfloor r\rfloor-1,\,\lfloor r\rfloor),\,\NN^\bot)}^2\right),\label{sup_Ir_2}
\end{align}
and from $\eta=\ex^{\frac{\lambda}{2}\Bigcdot}\kappa$ and
$\p_t\eta=\frac{\lambda}{2}\eta+\ex^{\frac{\lambda}{2}\Bigcdot}\p_t\kappa$, we can conclude that
\begin{align}
 &\quad|\kappa|_{H^1((\lfloor r\rfloor-1,\,\lfloor r\rfloor),\,\NN^\bot)}^2=|\kappa|_{L^2((\lfloor r\rfloor-1,\,\lfloor r\rfloor),\,\NN^\bot)}^2
+|\p_t\kappa|_{L^2((\lfloor r\rfloor-1,\,\lfloor r\rfloor),\,\NN^\bot)}^2\notag\\
&\leq \ex^{-\lambda(\lfloor r\rfloor-1)}|\eta|_{L^2((\lfloor r\rfloor-1,\,\lfloor r\rfloor),\,\NN^\bot)}^2
+|\ex^{-\frac{\lambda}{2} t}(\p_t\eta-\textstyle\frac{\lambda}{2}\eta)|_{L^2((\lfloor r\rfloor-1,\,\lfloor r\rfloor),\,\NN^\bot)}^2\notag\\
&\leq \overline C_{\left[\lambda\right]}\ex^{-\lambda(\lfloor r\rfloor-1)}|\eta|_{H^1((\lfloor r\rfloor-1,\,\lfloor r\rfloor),\,\NN^\bot)}^2
\le \overline C_{\left[\lambda\right]}\ex^{-\lambda(\lfloor r\rfloor-1)}M_0^\lambda\left(v,\,\kappa,\,\p_t\kappa\right);
\label{sup_Ir_3}
\end{align}
thus, from~\eqref{sup_Ir_1}, \eqref{sup_Ir_2}, and~\eqref{sup_Ir_3},
and recalling~\eqref{optimalcost} and~\eqref{normofR}, it follows that if $r\ge1$, then
\begin{align*}
&\quad|\ex^{\frac{\lambda}{2}\Bigcdot}v|_{W((r,\,r+1),\,H^2_{\diver}(\Omega,\,\R^3),\,L^2(\Omega,\,\R^3))}^2
=|w|_{W((r,\,r+1),\,H^2_{\diver}(\Omega,\,\R^3),\,L^2(\Omega,\,\R^3))}^2\\
&\le \overline C_{\left[|\hat u|_{\WW^{\rm st}},\,\lambda,\,\frac{1}{\lambda}\right]}
\left(\ex^{\lambda \lfloor r\rfloor}|v(\lfloor r\rfloor-1)|_{L^2_{\diver}(\Omega,\,\R^3)}^2
+\norm{(v_0^H,\,\kappa_0)}{H\times\NN^\perp}^2\right).
\end{align*}
Now,~\eqref{feedest.a=0} implies $|v(\lfloor r\rfloor-1)|_{L^2_{\diver}(\Omega,\,\R^3)}^2\le
\overline C_{\left[|\hat u|_{\WW^{\rm st}},\,\lambda,\,\frac{1}{\lambda}\right]}
\ex^{-\lambda(\lfloor r\rfloor-1)}\norm{(v_0^H,\,\kappa_0)}{H\times\NN^\perp}^2$, thus
\begin{align*}
|\ex^{\frac{\lambda}{2}\Bigcdot}v|_{W((r,\,r+1),\,H^2_{\diver}(\Omega,\,\R^3),\,L^2(\Omega,\,\R^3))}^2
\le \overline C_{\left[|\hat u|_{\WW^{\rm st}},\,\lambda,\,\frac{1}{\lambda}\right]}
\norm{(v_0^H,\,\kappa_0)}{H\times\NN^\perp}^2, \mbox{ for }r\ge1.
\end{align*}
From~\eqref{sup_Ir_1} we also have
\[
|\ex^{\frac{\lambda}{2}\Bigcdot}v|_{W((r,\,r+1),\,H^2_{\diver}(\Omega,\,\R^3),\,L^2(\Omega,\,\R^3))}^2
\le\overline C_{\left[|\hat u|_{\WW^{\rm st}},\,\lambda,\,\frac{1}{\lambda}\right]}
\left(|v_0|_{H^1_{\diver}(\Omega,\,\R^3)}^2+\norm{\kappa_0}{\NN^\perp}^2\right), \mbox{ for }r\in[0,\,1).
\]
Since~$(v_0^H,\,\kappa_0)$ satisfies~\eqref{st-ic-ext}, we have $\norm{\kappa_0}{\NN^\perp}
\le C_1\norm{\Xi\kappa_0}{H^\frac{1}{2}_{\rm av}(\Gamma,\,\R^3)}\le C_2\norm{v(0)}{H^1_{\diver}(\Omega,\,\R^3)}$ and
\[
|\ex^{\frac{\lambda}{2}\Bigcdot}v|_{W((r,\,r+1),\,H^2_{\diver}(\Omega,\,\R^3),\,L^2(\Omega,\,\R^3))}^2
\le \overline C_{\left[|\hat u|_{\WW^{\rm st}},\,\lambda,\,\frac{1}{\lambda}\right]}
\norm{v(0)}{H^1_{\diver}(\Omega,\,\R^3)}^2, \mbox{ for }r\ge0. 
\]
which ends the proof.
\end{proof}

Inspired by Theorem~\ref{T:feedback} and Lemma~\ref{L:reg.a}, we define the Banach space
\[
\ZZ^\lambda\coloneqq \left\{z\in L^2_{\rm loc}\left(\R_0,\,H^2_{\diver}(\Omega,\,\R^3)\right)
\,\Bigl|\,
\;|z|_{\ZZ^\lambda}<\infty
\right\}
\]
endowed with the norm
$|z|_{\ZZ^\lambda}\coloneqq \sup_{r\geq 0}\norm{\ex^{\frac{\lambda}{2}\Bigcdot}z}{W((r,\,r+1),\,H^2_{\diver}(\Omega,\,\R^3),\,
L^2(\Omega,\,\R^3))}$.

For given constant~$\rho>0$ and 
$(v_0^H,\,\kappa_0)$ satisfying~\eqref{st-ic-ext}, we define the subset
\[
\ZZ^\lambda_\rho\coloneqq \left\{z\in\ZZ^\lambda\mid\,z(0)=\FF_\aA^{-1}(v_0^H,\,Q^M_f\kappa_0)\mbox{ and }
|z|_{\ZZ^\lambda}^2\leq \rho|z(0)|_{H^1_{\diver}(\Omega,\,\R^3)}^2\right\},
\]
and the mapping $\Psi\colon \ZZ^\lambda_\rho\to\ZZ^\lambda_{\rm loc}$, $\bar z\mapsto z$, where~$(z,\,\kappa)$ solves
\begin{subequations}\label{sys-z=Bz}
\begin{align}
 \p_t z+\BB(\hat u)z
 -\nu\Delta z+\nabla p_{z,\,\bar z}&=-\Pi(\langle \bar z\cdot\nabla\rangle \bar z),& \p_t\kappa &=\KK_{\hat u}^{\lambda,\Bigcdot}(\Pi z,\,\kappa),\\
 \diver z &=0,& z\rest \Gamma &=\Xi\kappa, \\
 z(0)&=\FF_\aA^{-1}(v_0^H,\, Q^M_f\kappa_0),& \kappa(0)&=\kappa_0,
\end{align}
\end{subequations}

\begin{lemma}\label{L:contract}
Under the hypotheses of Theorem~\ref{T:ex.st-ct}, there exists $\rho>0$
such that the following property holds: for any $\gamma\in(0,\,1)$ one can find a
constant $\epsilon=\epsilon_\gamma>0$ such that, for any $(v_0^H,\,\kappa_0)$ satisfying~\eqref{st-ic-ext}
and~$|\FF_\aA^{-1}(v_0^H,\,Q^M_f\kappa_0)|_{H^1_{\diver}(\Omega,\,\R^3)}\le\epsilon$, the mapping~$\Psi$
takes the set~$\ZZ^\lambda_\rho$ into itself and satisfies the inequality
\begin{equation} \label{contraction}
|\Psi(\bar z_1)-\Psi(\bar z_2)|_{\ZZ^\lambda}\leq\gamma|\bar z_1-\bar z_2|_{\ZZ^\lambda}\quad
\mbox{for all}\quad \bar z_1,\,\bar z_2\in\ZZ^\lambda_\rho.
\end{equation}
\end{lemma}

\begin{proof} We divide the proof into~\ref{stcontr} main steps:
%
\begin{enumerate}[noitemsep,topsep=5pt,parsep=5pt,partopsep=0pt,leftmargin=0em]%
\renewcommand{\theenumi}{{\sf\arabic{enumi}}} 
 \renewcommand{\labelenumi}{} 
\item \textcircled{\bf s}~Step~\theenumi:\label{stprel} {\em a preliminary estimate.} Consider the system
\begin{subequations}\label{sys-z=f}
\begin{align}
 \p_t z+\BB(\hat u)z
 -\nu\Delta z+\nabla p_z&=f,& \p_t\kappa &=\KK_{\hat u}^{\lambda,\Bigcdot}(\Pi z,\,\kappa),\\
 \diver z &=0,& z\rest \Gamma &=\Xi\kappa, \\
 z(0)&=\FF_\aA^{-1}(v_0^H,\, Q^M_f\kappa_0),& \kappa(0)&=\kappa_0,
\end{align}
\end{subequations}
where $f\in L_{\rm loc}^2(\R_0,\,H)$. 
If $(z,\,\kappa)$ is the solution of system~\eqref{sys-z=f} with $f=0$,
by Lemma~\ref{L:reg.a}
\begin{equation}\label{estz0}
\sup_{r\geq 0}|\ex^{\frac{\lambda}{2}\Bigcdot}z|_{W((r,\,r+1),\,H^2_{\diver}(\Omega,\,\R^3),\,
L^2(\Omega,\,\R^3))}^2
\le \overline C_{\left[|\hat u|_{\WW^{\rm st}},\,\lambda,\,\frac{1}{\lambda}\right]}
|z(0)|_{H^1_{\diver}(\Omega,\,\R^3)}^2.
\end{equation}
We are going to derive a version of this estimate for suitable nonzero $f$;
for that we denote by $\sS_{0,\,t}^f(v_0^H,\,\kappa_0)$ the solution~$(z,\,\kappa)$
of~\eqref{sys-z=f}. In the case $f=0$, the operator $\sS_{0,\,t}^0$ is linear;
by the Duhamel formula we can write
\begin{equation} \label{duhamel}
(z,\,\kappa)(t)=\sS_{0,\,t}^f(v_0^H,\, \kappa_0)=\sS_{0,\,t}^0 (v_0^H,\, \kappa_0)+\int_0^t\sS_{s,\,t}^0 (f(s),\,0)\,\ed s
\end{equation}
where $\sS_{s,\,t}^f(v_s^H,\,\kappa_s)$ denotes the solution of the system~\eqref{sys-z=f}, with the initial time moved to $t=s$,
and the initial condition $(v_s^H,\,\kappa_s)$.
On the other hand, from~\eqref{estTfeed}, it follows in particular that
$|\ex^{\frac{\lambda}{2}(t-s)}\sS_{s,\,t}^0w|_{L^2_{\diver}(\Omega,\,\R^3)\times\NN^\perp}^2
\le\overline C_{\left[|\hat u|_{\WW^{\rm st}},\,\lambda,\,\frac{1}{\lambda}\right]}|w|_{H\times\NN^\perp}^2$;
then
\begin{align}
&\quad\,|(z,\,\kappa)(t)|_{L^2_{\diver}(\Omega,\,\R^3)\times\NN^\perp}^2
\le2\left|\sS_{0,\,t}^0 (v_0^H,\, \kappa_0)\right|^2+2\left|\int_0^t\sS_{s,\,t}^0 (f(s),\,0)\,\ed s\right|^2\label{zHf}\\
&\leq \overline C_{\left[|\hat u|_{\WW^{\rm st}},\,\lambda,\,\frac{1}{\lambda}\right]}
\ex^{-\lambda t}\left(|(v_0^H,\,\kappa_0)|_{H\times\NN^\perp}^2
+\left(\int_0^t\ex^{\frac{\lambda}{2}s}
\left|(f(s),\,0)\right|_{H\times\NN^\perp}\,\ed s\right)^{\!\!2}\;\right)\notag
\end{align}
Now we can find, again with $\lfloor t\rfloor\in\N$ as in~\eqref{floor},
\begin{align*}
&\quad\int_0^{t}\ex^{\frac{\lambda}{2}s}|f(s)|_{H}\,\ed s\leq
\sum_{k=0}^{\lfloor t\rfloor} \int_k^{k+1}\ex^{-\frac{\lambda}{2}s}\ex^{\lambda s}|f(s)|_{H}\,\ed s\\
&\le \sum_{k=0}^{\lfloor t\rfloor} \left(\int_k^{k+1}\ex^{-\lambda s}\,\ed s\right)^{\frac{1}{2}}
\left(\int_k^{k+1}\ex^{2\lambda s}|f(s)|_{H}^2\,\ed s\right)^{\frac{1}{2}}\\
&\le \sup_{\begin{subarray}{l}j\in\N\\0\le j\le \lfloor t\rfloor\end{subarray}}
\left(\int_j^{j+1}\ex^{2\lambda s}|f(s)|_{H}^2\,\ed s\right)^{\frac{1}{2}}
\sum_{k=0}^{\lfloor t\rfloor} \left(\int_k^{k+1}\ex^{-\lambda s}\,\ed s\right)^{\frac{1}{2}}
\end{align*}
and for the sum of the series, we can find
$\sum_{k=0}^{\lfloor t\rfloor} \left(\int_k^{k+1}\ex^{-\lambda s}\,\ed s\right)^{\frac{1}{2}}\le
\sum_{k=0}^\infty\left(\int_k^{k+1}\ex^{-\lambda s}\,\ed s\right)^{\frac{1}{2}}
=\left(-\textstyle\frac{1}{\lambda}(\ex^{-\lambda}-1)\right)^{\frac{1}{2}}\sum_{k=0}^\infty
\ex^{-\frac{\lambda}{2} k}
= \frac{(1-\ex^{-\lambda})^{\frac{1}{2}}}{\lambda^{\frac{1}{2}}(  1-\ex^{-\frac{\lambda}{2}})}
=\overline C_{\left[\frac{1}{\lambda}\right]}.$
Hence we obtain the inequality
$\int_0^{t}\ex^{\frac{\lambda}{2}s}|f(s)|_{H}\,\ed s
\le \overline C_{\left[\frac{1}{\lambda}\right]}
\sup_{\begin{subarray}{l}j\in\N\\0\le j\le \lfloor t\rfloor\end{subarray}}
\left(\int_j^{j+1}\ex^{2\lambda s}|f(s)|_{H}^2\,\ed s\right)^{\frac{1}{2}}$
and, recalling~\eqref{zHf},
\begin{equation} \label{supzHf}
\ex^{\lambda t}|(z,\,\kappa)(t)|_{\aA_{\Xi_1}\times\NN^\perp}^2 \le
\overline C_{\left[|\hat u|_{\WW^{\rm st}},\,\lambda,\,\frac{1}{\lambda}\right]}
\left(|(v_0^H,\,\kappa_0)|_{H\times\NN^\perp}^2
+\sup_{\begin{subarray}{l}j\in\N\\0\le j\le \lfloor t\rfloor\end{subarray}}
\int_j^{j+1}\ex^{2\lambda s}|f(s)|_{\aA_{\Xi_1}}^2\,\ed s\right)
\end{equation}
for all $t\ge0$. Now we use Lemma~\ref{L:smooth-prop} to obtain
\begin{align}
\quad|z(r+1)|_{H^1_{\diver}(\Omega,\,\R^3)}^2&\le\overline C_{\left[|\hat u|_{\WW^{\rm st}}\right]}
\left(|z(r)|_{\aA_{\Xi_1}}^2 +|f|_{L^2((r,\,r+1),\,{H})}^2\right)\label{zH1f1}\\
&\quad+\overline C_{\left[|\hat u|_{\WW^{\rm st}}\right]}\left|\kappa(r)
 +\int_r^{\Bigcdot}\KK_{\hat u}^{\lambda,\,s}(\Pi z(s),\,\kappa(s))\,\ed s\right|_{H^1((r,\,r+1),\,\NN^\bot)}^2\hspace*{-.2em}.\notag
\end{align}

Recalling the bound in~\eqref{Opnorm}, we find
\begin{align}
\quad|z(r+1)|_{H^1_{\diver}(\Omega,\,\R^3)}^2&\le\overline C_{\left[|\hat u|_{\WW^{\rm st}}\right]}
\left(|z(r)|_{\aA_{\Xi_1}}^2 +|f|_{L^2((r,\,r+1),\,{H})}^2\right)\notag\\
&\quad+\overline C_{\left[|\hat u|_{\WW^{\rm st}},\,\lambda,\,\frac{1}{\lambda}\right]}
 \sup_{t\in(r,\,r+1)}|(\Pi z,\,\kappa)(t)|_{H\times\NN^\perp}^2.\label{zsupA1}
 \end{align}
From~\eqref{supzHf}, and~\eqref{zsupA1}, we obtain
\begin{align}
&\quad|z(r+1)|_{H^1_{\diver}(\Omega,\,\R^3)}^2
\le \overline C_{\left[|\hat u|_{\WW^{\rm st}},\,\lambda,\,\frac{1}{\lambda}\right]}
\left(\sup_{t\in(r,\,r+1)}|(z,\,\kappa)(t)|_{\aA_{\Xi_1}\times\NN^\perp}^2 +|f|_{L^2((r,\,r+1),\,H)}^2\right)\label{zH1f2}\\
&\le\overline C_{\left[|\hat u|_{\WW^{\rm st}},\,\lambda,\,\frac{1}{\lambda}\right]}
\ex^{-\lambda r}\hspace*{-.3em}\left(\hspace*{-.2em}|(v_0^H,\,\kappa_0)|_{H\times\NN^\perp}^2
+\sup_{\begin{subarray}{l}k\in\N\\0\le k\le \lfloor r+1\rfloor\end{subarray}}
\int_k^{k+1}\ex^{2\lambda s}|f(s)|_{H}^2\,\ed s\hspace*{-.2em}\right)\!;
\notag
\end{align}
because the inequalities $|f|_{L^2((r,\,r+1),\,H)}^2\le \ex^{-\lambda r}\int_r^{r+1}\ex^{\lambda s}|f(s)|_{H}^2\,\ed s\le
\ex^{-\lambda r}\int_{\lfloor r\rfloor}^{\lfloor r+1\rfloor+1}\ex^{2\lambda s}|f(s)|_{H}^2\,\ed s
\le 2\ex^{-\lambda r}\sup_{\begin{subarray}{l}k\in\N\\0\le k\le \lfloor r+1\rfloor\end{subarray}}
\int_k^{k+1}\ex^{2\lambda s}|f(s)|_{H}^2\,\ed s$ do hold true.

For $t\in(0,\,1)$, from Theorem~\ref{T:exist-ssol-v}, \eqref{Opnorm}, and~\eqref{supzHf}  we can also obtain
\begin{align}
\quad|z(t)|_{H^1_{\diver}(\Omega,\,\R^3)}^2&\le\overline C_{\left[|\hat u|_{\WW^{\rm st}}\right]}
\left(|z(0)|_{H^1_{\diver}(\Omega,\,\R^3)}^2 +|f|_{L^2((0,\,1),\,H)}^2\right)\label{zH1f3}\\
&\quad+\overline C_{\left[|\hat u|_{\WW^{\rm st}}\right]}\left|\kappa(0)
 +\int_0^{\Bigcdot}\KK_{\hat u}^{\lambda,\,s}(\Pi z(s),\,\kappa(s))\,\ed s\right|_{H^1((0,\,1),\,\NN^\bot)}^2\notag\\
&\le\overline C_{\left[|\hat u|_{\WW^{\rm st}}\right]}
\left(|z(0)|_{H^1_{\diver}(\Omega,\,\R^3)}^2 +|\kappa(0)|_{\NN^\perp}^2+|f|_{L^2((0,\,1),\,{H})}^2\right)\notag
\end{align}
which, together with~\eqref{zH1f2} give for all $t\geq 0$:
\begin{align*}
|z(t)|_{H^1_{\diver}(\Omega,\,\R^3)}^2
&\le\overline C_{\left[|\hat u|_{\WW^{\rm st}},\,\lambda,\,\frac{1}{\lambda}\right]}
\ex^{-\lambda t}\left(|z(0)|_{H^1_{\diver}(\Omega,\,\R^3)}^2
+\sup_{\begin{subarray}{l}k\in\N\\0\le k\le \lfloor t\rfloor\end{subarray}}\int_k^{k+1}\ex^{2\lambda s}|f(s)|_{H}^2\,\ed s\right);
\end{align*}
where we have used $0<\ex^{\lambda r}$, $\ex^{-\lambda r}=\ex^{\lambda}\ex^{-\lambda(r+1)}$, and $1<\ex^{\lambda}\ex^{-\lambda t}$ for
$r\geq0$ and $t\in(0,\,1)$ and also the fact that $\norm{\kappa(0)}{\NN^\perp}
\le C\norm{v(0)}{H^1_{\diver}(\Omega,\,\R^3)}$ because~$(v_0^H,\,\kappa_0)$ satisfies~\eqref{st-ic-ext}. 

Using again Theorem~\ref{T:exist-ssol-v} and proceeding as above, and using~$\norm{\kappa(t)}{\NN^\perp}
\le C\norm{v(t)}{H^1_{\diver}(\Omega,\,\R^3)}$, we can now derive
\begin{align*}
&\quad|z|_{W((r,\,r+1),\,H^2_{\diver}(\Omega,\,\R^3),\,L^2(\Omega,\,\R^3)}^2
\le\overline C_{\left[|\hat u|_{\WW^{\rm st}}\right]}
\left(|z(r)|_{H^1_{\diver}(\Omega,\,\R^3)}^2+|f|_{L^2((r,\,r+1),\,H)}^2\right)\\
&\hspace*{10em}+\overline C_{\left[|\hat u|_{\WW^{\rm st}}\right]}\left|\kappa(r)
 +\int_r^{\Bigcdot}\KK_{\hat u}^{\lambda,\,s}(\Pi z(s),\,\kappa(s))\,\ed s\right|_{H^1((r,\,r+1),\,\NN^\bot)}^2\\
&\le\overline C_{\left[|\hat u|_{\WW^{\rm st}},\,\lambda,\,\frac{1}{\lambda}\right]}
\ex^{-\lambda r}\left(|z(0)|_{H^1_{\diver}(\Omega,\,\R^3)}^2+\sup_{\begin{subarray}{l}k\in\N\\0\le k\le \lfloor r\rfloor\end{subarray}}\int_k^{k+1}\ex^{2\lambda s}|f(s)|_{H}^2\,\ed s\right);
\end{align*}
which implies, since $\p_t(\ex^{\frac{\lambda}{2} \Bigcdot}z)
=\frac{\lambda}{2}\ex^{\frac{\lambda}{2} \Bigcdot}z+\ex^{\frac{\lambda}{2} \Bigcdot}\p_tz$, that
\begin{align}
&\quad\sup_{r\ge0}|\ex^{\frac{\lambda}{2} \Bigcdot}z(t)|_{W((r,\,r+1),\,H^2_{\diver}(\Omega,\,\R^3),\,L^2(\Omega,\,\R^3)}^2
\label{estzf}\\
&\le\overline C_{\left[|\hat u|_{\WW^{\rm st}},\,\lambda,\,\frac{1}{\lambda}\right]}
\left(|z(0)|_{H^1_{\diver}(\Omega,\,\R^3)}^2
+\sup_{\begin{subarray}{l}k\in\N\end{subarray}}\int_k^{k+1}\ex^{2\lambda s}|f(s)|_{H}^2\,\ed s\right)\notag
\end{align}
which is a version of~\eqref{estz0} for nonzero $f$. 

%
\item \textcircled{\bf s}~Step~\theenumi:\label{stinto} {\em $\Psi$ maps $\ZZ^\lambda_\rho$ into itself,
if $|z(0)|_{H^1_{\diver}(\Omega,\,\R^3)}$ is small.}
We will replace $f$ by $-\Pi(\langle \bar z\cdot\nabla\rangle \bar z)$ in~\eqref{estzf}. First we recall some standard
estimates for the nonlinear term: from the Agmon inequality (see~\cite[Chapter II, Section 1.4]{Temam97}),
$|u|_{L^\infty(\Omega,\,\R^3)}
\le C_1|u|_{H^1(\Omega,\,\R^3)}^{\frac{1}{2}}|u|_{H^2(\Omega,\,\R^3)}^{\frac{1}{2}}
\le C_2|u|_{H^2(\Omega,\,\R^3)}$, we can obtain
$|\langle \bar z\cdot\nabla\rangle \bar w|_{L^2(\Omega,\,\R^3)}\le
C|\bar z|_{H^2_{\diver}(\Omega,\,\R^3)}|\bar w|_{H^1_{\diver}(\Omega,\,\R^3)}$
and
\begin{align*}
\sup_{\begin{subarray}{l}k\in\N\\0\le k\le \lfloor t\rfloor\end{subarray}}\hspace*{-.2em}\int_k^{k+1}\hspace*{-.5em}\ex^{2\lambda s}
|\langle \bar z\cdot\nabla\rangle \bar z|_{L^2(\Omega,\,\R^3)}^2\,\ed s
&\le\hspace*{-.5em}
\sup_{\begin{subarray}{l}s\in [0,\,\lfloor t\rfloor+1]\end{subarray}}\hspace*{-.2em}
|\ex^{\frac{\lambda}{2} s}\bar z(s)|_{H^1_{\diver}(\Omega,\,\R^3)}^2
\sup_{\begin{subarray}{l}k\in\N\\0\le k\le \lfloor t\rfloor\end{subarray}}\hspace*{-.2em}\int_k^{k+1}\hspace*{-.5em}
|\ex^{\frac{\lambda}{2} s}\bar z|_{H^2_{\diver}(\Omega,\,\R^3)}^2\,\ed s.
\end{align*}
Thus,
inequality~\eqref{estzf} with $f=-\Pi(\langle \bar z\cdot\nabla\rangle \bar z)$ gives us
\begin{equation} \label{Psiz1}
|\Psi(\bar z)|_{\ZZ^\lambda}^2
\le\overline C_{\left[|\hat u|_{\WW^{\rm st}},\,\lambda,\,\frac{1}{\lambda}\right]}
\left(|z(0)|_{H^1_{\diver}(\Omega,\,\R^3)}^2
+|\bar z|_{\ZZ^\lambda}^4\right).
\end{equation}
If $\bar z\in \ZZ^\lambda_\rho$, then
\begin{equation} \label{Psiz2}
|\Psi(\bar z)|_{\ZZ^\lambda}^2\le \overline C_{\left[|\hat u|_{\WW^{\rm st}},\,\lambda,\,\frac{1}{\lambda}\right]}
(1+\rho^2|z(0)|_{H^1_{\diver}(\Omega,\,\R^3)}^2)|z(0)|_{H^1_{\diver}(\Omega,\,\R^3)}^2
\end{equation}
and if we set~$\rho=2\overline C$ and $\epsilon<\frac{1}{\rho}$, where
$\overline C=\overline C_{\left[|\hat u|_{\WW^{\rm st}},\,\lambda,\,\frac{1}{\lambda}\right]}$ is the
constant in~\eqref{Psiz2}, then we obtain
$\overline C(1+\rho^2\epsilon^2)\le\rho$ if $|z(0)|_{H^1_{\diver}(\Omega,\,\R^3)}\le\epsilon$,
which means that~$\Psi(\bar z)\in\ZZ^\lambda_\rho$.

%
\item \textcircled{\bf s}~Step~\theenumi:\label{stcontr} {\em $\Psi$ is a contraction, if $|z(0)|_{H^1_{\diver}(\Omega,\,\R^3)}$ is smaller.}
It remains to prove~\eqref{contraction}.
Let us take two functions $\bar z_1, \bar z_2\in \ZZ^\lambda_\rho$ and let~$(\Psi(\bar z_1),\,\kappa_1)$ and~$(\Psi(\bar z_2),\,\kappa_2)$ 
be the corresponding solutions for~\eqref{sys-z=f}. Set $e=\bar z_1-\bar z_2$ and
$(d^z,\,d^\kappa)=(\Psi(\bar z_1)- \Psi(\bar z_2),\,\kappa_1-\kappa_2)$. Then~$(d^z,\,d^\kappa)$ solves~\eqref{sys-z=f} with $(d^z,\,d^\kappa)(0)=(0,\,0)$ and
$f=\Pi(\langle \bar z_2\cdot\nabla\rangle \bar z_2)-\Pi(\langle \bar z_1\cdot\nabla\rangle \bar z_1)$.
Therefore, by inequality~\eqref{estzf}, we have
\begin{equation} \label{dPsi1}
|\Psi(\bar z_1)-\Psi(\bar z_2)|_{\ZZ^\lambda}^2\le
\overline C_{\left[|\hat u|_{\WW^{\rm st}},\,\lambda,\,\frac{1}{\lambda}\right]}
\sup_{t\ge0}\int_t^{t+1}\ex^{2\lambda s}|\langle \bar z_2\cdot\nabla\rangle \bar z_2
-\langle \bar z_1\cdot\nabla\rangle \bar z_1|_{L^2(\Omega,\,\R^3)}^2\ed s.
\end{equation}
Straightforward computations give us
\begin{align*}
&\quad|\langle \bar z_2\cdot\nabla\rangle \bar z_2
-\langle \bar z_1\cdot\nabla\rangle \bar z_1|_{L^2(\Omega,\,\R^3)}^2
=\left|-\langle e\cdot\nabla\rangle \bar z_2
-\langle \bar z_1\cdot\nabla\rangle e\right|_{L^2(\Omega,\,\R^3)}^2\\
&\le C\left(|e|_{H^2_{\diver}(\Omega,\,\R^3)}^2|\bar z_2|_{H^1_{\diver}(\Omega,\,\R^3)}^2
+|\bar z_1|_{H^2_{\diver}(\Omega,\,\R^3)}^2|e|_{H^1_{\diver}(\Omega,\,\R^3)}^2\right)^2
\end{align*}
and, from~\eqref{dPsi1}, it follows $|\Psi(\bar z_1)-\Psi(\bar z_2)|_{\ZZ^\lambda}^2
\le \overline C_{\left[|\hat u|_{\WW^{\rm st}},\,\lambda,\,\frac{1}{\lambda}\right]}
\bigl(|\bar z_1|_{\ZZ^\lambda}^2+|\bar z_2|_{\ZZ^\lambda}^2\bigr) |e|_{\ZZ^\lambda}^2$, that is,
\begin{align}
 |\Psi(\bar z_1)-\Psi(\bar z_2)|_{\ZZ^\lambda}^2
\le \overline C_{\left[|\hat u|_{\WW^{\rm st}},\,\lambda,\,\frac{1}{\lambda}\right]}
2\rho |z(0)|_{H^1_{\diver}(\Omega,\,\R^3)}^2|\bar z_1-\bar z_2|_{\ZZ^\lambda}^2\label{contra-f}
\end{align}
Choosing $\epsilon>0$ (smaller than the one chosen in Step~\ref{stinto} and) such that $2\overline C_2
\rho \epsilon^2\le\gamma^2$, where $\overline C_2
=\overline C_{\left[|\hat u|_{\WW^{\rm st}},\,\lambda,\,\frac{1}{\lambda}\right]}$ is
the constant in~\eqref{contra-f}, we see that if $|z(0)|_{H^1_{\diver}(\Omega,\,\R^3)}\le\epsilon$,
then~\eqref{contraction} holds.
\end{enumerate}
The proof of Lemma~\ref{L:contract} is complete.
\end{proof}

\begin{proof}[Proof of Theorem~\ref{T:ex.st-ct}]
Form Lemma~\ref{L:contract} and the contraction
mapping principle (cf.~\cite[Chapter~10, Theorem~6.1]{Smith83}) it follows that if
$v(0)=\FF_\aA^{-1}(v_0^H,\,Q^M_f\kappa_0)\in H^1_{\diver}(\Omega,\,\R^3)$ satysfies~\eqref{st-ic-ext} and is sufficiently small,
$|v(0)|_{H^1_{\diver}(\Omega,\,\R^3)}<\epsilon$, then there exists a unique fixed point $v\in
\ZZ^\lambda_\rho$ for~$\Psi$. It follows from the
definitions of~$\Psi$ and~$\ZZ^\lambda_\rho$ that together with a suitable function~$\kappa$, the pair~$(v,\,\kappa)$
solves the system~\eqref{sys-z=Bz}, with $\bar z=v$.

Since we can write $\langle v\cdot\nabla\rangle v=\Pi(\langle v\cdot\nabla\rangle v)+\nabla p_{\breve v}$,
for a suitable function $p_{\breve v}\in H^1(\Omega,\,\R)$, it follows that if we set
$p_v\coloneqq p_{v,\,v}+p_{\breve v}$ then we can conclude that~$(v,\,\kappa)$ solves~\eqref{sys-vnonlin-ext}.

From~\eqref{Psiz2}, taking~$\bar z=v=z$, we find
$
|v|_{\ZZ^\lambda}^2\le \overline C_{\left[|\hat u|_{\WW^{\rm st}},\,\lambda,\,\frac{1}{\lambda}\right]}
(1+\rho^2\epsilon^2)|z(0)|_{H^1_{\diver}(\Omega,\,\R^3)}^2,
$
which implies~\eqref{expH1.main}.

It remains to prove the uniqueness of~$(v,\,\kappa)$ in the space
\[
\widehat\ZZ\coloneqq W_{\rm loc}(\R_0,\,H^2_{\diver}(\Omega,\,\R^3),\,L^2(\Omega,\,\R^3))\times
H^1_{\rm loc}(\R_0,\,\NN^\perp)\supset\ZZ^\lambda.
\]
Let $(\bar v,\,\bar\kappa)\in\widehat\ZZ$ be another solution for~\eqref{sys-vnonlin-ext}, and set
$(z,\,\theta)\coloneqq(\bar v-v,\,\bar\kappa-\kappa)$.
Then $z$ solves~\eqref{sys-vnonlin} with $\hat u+v$ in the place of $\hat u$, $z(0)=0$
and~$\zeta=\Xi\theta=\Xi\int_0^{\Bigcdot}\KK_{\hat u}^{\lambda,s}(\Pi z(s),\,Q^M_f\theta(s))\,\ed s=\Xi\int_0^{\Bigcdot}\KK_{\hat u}^{\lambda,s}\FF_\aA z(s)\,\ed s$.

We can extend a given function
$g\in H^{\frac{3}{2}}_{\rm av}(\Gamma,\,\R^3)$ to the solution
$\widehat F g\in H^{2}_{\diver}(\Omega,\,\R^3)$ of the Stokes
system
\begin{equation}\label{sys-stokes}
-\Delta\widehat F g+\nabla p_g=0,\,\diver\widehat Fg=0,\,\widehat Fg\rest\Gamma=g 
\end{equation}
and the mapping $g\mapsto\widehat F g$ is continuous (cf.~\cite[Chapter~1, Proposition~2.3]{Temam01}).
Then, we can define the extension
$\widehat F\colon \Xi H^1_{\rm loc}(\R_0,\,\NN^\perp)\to H^1_{\rm loc}(\R_0,\,H^2_{\diver}(\Omega,\,\R^3))$ by
\[
\widehat F\Xi\kappa(t)\coloneqq\sum_{i=1}^M\kappa_i(t)\widehat F\chi\E_0^\OO P_{\chi^\bot}^\OO \pi_i\nnn
+\kappa_{M+i}(t)\widehat F\chi\E_0^\OO \tau_i.
\]
Notice that $\widehat F$ is injective because $\widehat Fg=0$ implies $g=\widehat Fg\rest\Gamma=0$.
Observe also that, since $\widehat F\colon \Xi\NN^\perp$ is finite-dimensional, we have
$\norm{\widehat F\Xi\kappa(t)}{H^2_{\diver}(\Omega,\,\R^3)}\le C\norm{\widehat F\Xi\kappa(t)}{\aA_{\Xi_1}}.$ 

Moreover we can see that $y\coloneqq z-\widehat F\Xi\theta$ solves system~\eqref{sys-y-nonl} with $k=\hat u+v+\widehat F\Xi\theta$, 
$y(0)=0$, and $f=\p_t\widehat F\Xi\theta+\langle \widehat F\Xi\theta\cdot\nabla\rangle \widehat F\Xi\theta
-\nu\Delta \widehat F\Xi\theta+\BB(\hat u+v)\widehat F\Xi\theta$.
 Hence, by standard arguments and~$\frac{\ed}{\ed t}\norm{(y,\,\widehat F\Xi\theta)}{H\times\aA_{\Xi_1}}^2
=2(\p_ty,\,y)_H+2(\widehat F\Xi\p_t\theta,\,\widehat F\Xi\theta)_{\aA_{\Xi_1}}$ we can find
\begin{align*}
&\quad\,\frac{\ed}{\ed t}\norm{(y,\,\widehat F\Xi\theta)}{H\times\aA_{\Xi_1}}^2\\
&\le C_1\norm{\hat u+v+\widehat F\Xi\theta}{L^\infty(\Omega,\,\R^3)}^2\norm{y}{H}^2 +C_1\norm{-f+\p_t\widehat F\Xi\theta}{V'}^2
+2(\widehat F\Xi\p_t\theta,\,-y+\widehat F\Xi\theta)_{\aA_{\Xi_1}}
\end{align*}
and from
\begin{align*}
\norm{f-\p_t\widehat F\Xi\theta}{V'}^2
&\le C_2\norm{\widehat F\Xi\theta}{\aA_{\Xi_1}}^4+C_2\norm{\widehat F\Xi\theta}{\aA_{\Xi_1}}^2+C_2\norm{\hat u+v}{L^\infty(\Omega,\,\R^3)}^2\norm{\widehat F\Xi\theta}{\aA_{\Xi_1}}^2,\\
(\widehat F\Xi\p_t\theta,\,-y+\widehat F\Xi\theta)_{\aA_{\Xi_1}}
&=(\widehat F\Xi\KK_{\hat u}^{\lambda,\Bigcdot}\FF_{\aA}(y+\widehat F\Xi\theta),\,-y+\widehat F\Xi\theta)_{\aA_{\Xi_1}}
\le C_3\norm{(y,\,\widehat F\Xi\theta)}{H\times\aA_{\Xi_1}}^2,
\end{align*}
we arrive to
\[
\frac{\ed}{\ed t}\norm{(y,\,\widehat F\Xi\theta)}{H\times\aA_{\Xi_1}}^2
\le C_4\left(\norm{\hat u+v}{L^\infty(\Omega,\,\R^3)}^2 
+\norm{\widehat F\Xi\theta}{\aA_{\Xi_1}}^2+1\right)\norm{(y,\,\widehat F\Xi\theta)}{H\times\aA_{\Xi_1}}^2.
\]
Since $(v,\,\theta)\in\widehat\ZZ$ we have that $g=C_4\left(\norm{\hat u+v}{L^\infty(\Omega,\,\R^3)}^2 
+\norm{\widehat F\Xi\theta}{\aA_{\Xi_1}}^2+1\right)$ is locally integrable and from the Gronwall lemma it follows that
\[
\norm{(y,\,\widehat F\Xi\theta)(t)}{H\times\aA_{\Xi_1}}^2
\le \ex^{\int_0^tg(s)\,\ed s}\norm{(y,\,\widehat F\Xi\theta)(0)}{H\times\aA_{\Xi_1}}^2=0. 
\]
We can conclude that $\bar v=v+z=v+y+\widehat F\Xi\theta=v$, and also that $\bar \kappa=\kappa+\theta=\kappa$, because
$\widehat F\Xi:\NN^\perp \to \widehat F\Xi\NN^\perp$ is injective. That is, the solution~$(v,\,\kappa)$ for~\eqref{sys-vnonlin-ext} is unique in~$\widehat\ZZ$.
\end{proof}

\subsection{Integral feedback rule and stabilization to trajectories}
Given $v_0\in\aA_{\Xi_2}\subset H^1_{\diver}(\Omega,\,\R^3)$, there exists a unique~$\kappa_0\in\NN^\perp$
such that~$v_0\rest\Gamma=\Xi\kappa_0$. Let~$(v,\,\kappa)$ be the solution of~\eqref{sys-vnonlin-ext}, with
$(v_0^H,\,Q^M_f\kappa_0)=(\Pi v_0,\,Q^M_f\kappa_0)=\FF_\aA v_0$.
Then we observe
that~$v$ solves~\eqref{sys-vnonlin} with
$\zeta=\Xi\kappa$ and $v(0)=\FF_\aA^{-1}(\Pi v_0,\, Q^M_f\kappa_0)=v_0$.
From Theorem~\ref{T:ex.st-ct}, we can conclude that if $\norm{v_0}{H^1_{\diver}(\Omega,\,\R^3)}$ is small enough, then the integral feedback rule
\[
v\rest\Gamma=\Xi\kappa_0+\Xi\int_0^{\Bigcdot}\KK_{\hat u}^{\lambda,s}(\Pi v,\,Q^M_f\kappa)(s)\,\ed s
=v_0\rest\Gamma+\Xi\int_0^{\Bigcdot}\KK_{\hat u}^{\lambda,s}\FF_\aA v(s)\,\ed s
\]
does exponentially stabilize system~\eqref{sys-vnonlin} to zero with rate~$\frac{\lambda}{2}$.
That is, setting~$\mathbf K_{\hat u}^{\lambda,\Bigcdot}\coloneqq\KK_{\hat u}^{\lambda,\Bigcdot}\FF_\aA$
the feedback control $\zeta=(u_0-\hat u(0))\rest\Gamma+\Xi\int_0^{\Bigcdot}\mathbf K_{\hat u}^{\lambda,s}(u(s)-\hat u(s))\,\ed s$
does exponentially stabilize system~\eqref{sys-u-bdry} to the targeted trajectory~$\hat u$ with the same
rate~$\frac{\lambda}{2}$, provided $u_0-\hat u(0)\in\aA_{\Xi_2}\subset H^1_{\diver}(\Omega,\,\R^3)$ and~$\norm{u_0-\hat u(0)}{H^1_{\diver}(\Omega,\,\R^3)}$ is small enough.

Furthermore the corresponding solution~$u$ for system~\eqref{sys-u-bdry} is unique in the affine
subspace~$\hat u+W_{\rm loc}(\R_0,\,H^2_{\diver}(\Omega,\,\R^3),\,L^2(\Omega,\,\R^3))$.

\medskip
\addcontentsline{toc}{section}{Appendix}
\setcounter{section}{1}
\setcounter{section}{0}
\setcounter{subsection}{0}
\setcounter{theorem}{0}
\gdef\thesection{\Alph{section}}%

\section{Appendix}

\subsection{Proof of Property~\eqref{comut}}\label{sS:comut} 
From the definition of~$\Xi$ in~\eqref{Xi}, we have that for any $\kappa\in\R^{2M}$,
\[
 \Xi\kappa=0 \quad\mbox{ if, and only if,}\quad (\Xi Q^M_f\kappa,\,\Xi Q^M_l\kappa)=(0,\,0),
\]
because $\Xi Q^M_f\kappa\in L^2(\Gamma,\,\R)\nnn$ and $\Xi Q^M_l\kappa\in L^2(\Gamma,\,T\Gamma)$ are orthogonal in~$L^2(\Gamma,\,\R^3)$.
This means that $\kappa\in\NN$ if, and only if, $(Q^M_f\kappa,\,Q^M_l\kappa)\in\NN\times\NN$,
which implies that
\[
 QP_{\NN}=P_{\NN}QP_{\NN}, \quad\mbox{ for }\quad  Q\in\{Q^M_f,\,Q^M_l\}.
\]
Hence, since for any given $(\xi,\,\kappa)\in\R^{2M}\times\R^{2M}$ we have
\[
(P_{\NN}QP_{\NN^\perp}\xi,\, \kappa)_{\R^{2M}}=(P_{\NN^\perp}\xi,\, QP_{\NN}\kappa)_{\R^{2M}}
=(P_{\NN^\perp}\xi,\, P_{\NN}QP_{\NN}\kappa)_{\R^{2M}}=0,
\]
it follows that $P_{\NN}QP_{\NN^\perp}=0$ and
\[\begin{array}{l}
P_{\NN}Q= P_{\NN}QP_{\NN}+P_{\NN}QP_{\NN^\perp}=QP_{\NN},\\
P_{\NN^\perp}Q= Q-P_{\NN}Q=Q-QP_{\NN}=QP_{\NN^\perp},
\end{array}\quad\mbox{for all}\quad Q\in\{Q^M_f,\,Q^M_l\},
\]
which is the statement of~\eqref{comut}.\qed

\subsection{An Example concerning Remark~\ref{R:Nn} }\label{sS:ex-nottriv} 
We illustrate the fact that the space $\NN$ defined in~\eqref{N.sp} is not necessarily trivial.\newline
\begin{minipage}{.75\textwidth}
Let $\OO=(0,\,\pi)\times\T^1$; the family $\{\pi_i\mid i\in\N_0\}$ contains the family
$\{\sigma_{n}\mid n\in\N_0\}$, with $\sigma_{n}(r,\,s)\coloneqq \frac{1}{\pi}\sin(nr)$,
$(r,\,s)\in(0,\,\pi)\times\T^1\sim(0,\,\pi)\times[0,\,2\pi)$. Define the indicator operator
$\II_{(\frac{\pi}{3},\,\frac{2\pi}{3})}\colon C(\OO,\,\R)\to L^2(\OO,\,\R)$, sending $f$ to
$\II_{(\frac{\pi}{3},\,\frac{2\pi}{3})}f$ defined by
$\II_{(\frac{\pi}{3},\,\frac{2\pi}{3})}f(r,\,s)
\coloneqq \begin{cases}f(r,\,s)&\mbox{if }r\in(\frac{\pi}{3},\,\frac{2\pi}{3})\\
   0&\mbox{if }r\in (0,\,\pi)\setminus[\frac{\pi}{3},\,\frac{2\pi}{3}]
\end{cases}.$
\end{minipage}%
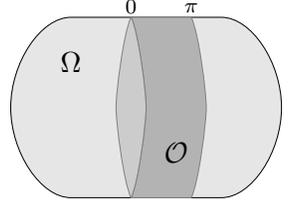
\hfill%
\begin{minipage}{.21\textwidth} 
\vspace*{.5cm}
\definecolor{bkground}{rgb}{1,1,1}
\definecolor{lightgrey9}{rgb}{.9,.9,.9}%
\definecolor{lightgrey8}{rgb}{.8,.8,.8}%
\definecolor{lightgrey7}{rgb}{.7,.7,.7}%
\definecolor{grey}{rgb}{.5,.5,.5}%
\linethickness{.2pt}
\begin{center}
\setlength{\unitlength}{2mm}
\begin{picture}(17,12)(0,-2)
\moveto(3,0)
\curveto(4,0)(8,0)(13,0) \curveto(15,0)(17,3)(17,6)
\curveto(17,9)(15,12)(13,12) \curveto(8,12)(4,12)(3,12)
\curveto(1,12)(-1,9)(-1,6)
\curveto(-1,3)(1,0)(3,0)
{\color{lightgrey9}\fillpath
}%
\moveto(3,0)
\curveto(4,0)(8,0)(13,0) \curveto(15,0)(17,3)(17,6)
\curveto(17,9)(15,12)(13,12) \curveto(8,12)(4,12)(3,12)
\curveto(1,12)(-1,9)(-1,6)
\curveto(-1,3)(1,0)(3,0)
{\color{black}\strokepath
}%
\moveto(7,0)
\curveto(7.5,1)(8,5)(8,6)
\curveto(8,7)(7.5,11)(7,12)
\curveto(7.5,12)(8,12)(11,12)
\curveto(11.5,11)(12,7)(12,6)
\curveto(12,5)(11.5,1)(11,0)
\curveto(10,0)(9,0)(7,0)
{\color{lightgrey7}\fillpath%
}%
\moveto(7,0)
\curveto(7.5,1)(8,5)(8,6)
\curveto(8,7)(7.5,11)(7,12)
\curveto(7.5,12)(8,12)(11,12)
\curveto(11.5,11)(12,7)(12,6)
\curveto(12,5)(11.5,1)(11,0)
\curveto(10,0)(9,0)(7,0)
{\color{grey}\linethickness{.2pt}\strokepath%
}%
\moveto(7,0)
\curveto(7.5,1)(8,5)(8,6)
\curveto(8,7)(7.5,11)(7,12)
\curveto(6.5,11)(6,7)(6,6)
\curveto(6,5)(6.5,1)(7,0)
{\color{lightgrey8}\fillpath%
}%
\moveto(7,0)
\curveto(7.5,1)(8,5)(8,6)
\curveto(8,7)(7.5,11)(7,12)
\curveto(6.5,11)(6,7)(6,6)
\curveto(6,5)(6.5,1)(7,0)
{\color{grey}\linethickness{.2pt}\strokepath%
}%

\put(10,3){\makebox(0,0){$\OO$}}
\put(3,9){\makebox(0,0){$\Omega$}}
\put(7,12.7){\makebox(0,0){\scriptsize$0$}}
\put(11,12.6){\makebox(0,0){\scriptsize$\pi$}}

\end{picture}
\end{center}
\end{minipage}

Now, setting the mapping
$\chi\coloneqq \II_{(\frac{\pi}{3},\,\frac{2\pi}{3})}\left(3\sigma_{3}-\sigma_9\right)$, from direct computations we
obtain the identities
$\frac{\p\chi}{\p s}=0$, $\frac{\p\chi}{\p r}
=\II_{(\frac{\pi}{3},\,\frac{2\pi}{3})}\frac{\p}{\p r}\left(3\sigma_{3}-\sigma_9\right)$, and
$\frac{\p^2\chi}{\p r^2}
=\II_{(\frac{\pi}{3},\,\frac{2\pi}{3})}\frac{\p^2}{\p r^2}\left(3\sigma_{3}-\sigma_9\right)$;
we can check that $\chi\in C^2(\OO,\,\R)$ and $\supp(\chi)=[\frac{\pi}{3},\,\frac{2\pi}{3}]\times\T^1\subset\OO$.
Now we can show that the functions
$\chi\E_0^\OO P_{\chi^\bot}^\OO(\sigma_3\nnn)$ and $\chi\E_0^\OO P_{\chi^\bot}^\OO(\sigma_9\nnn)$
are linearly dependent: we find
\begin{align*}
P_{\chi^\bot}^\OO(\sigma_3\nnn)
=\sigma_3\nnn-\frac{(\sigma_3\nnn,\,\chi\nnn)_{L^2(\OO,\,R)}}{|\chi\nnn|_{L^2(\OO,\,R)}^2}\chi\nnn
=\sigma_3\nnn-\frac{3}{10}\chi\nnn,\\
P_{\chi^\bot}^\OO(\sigma_9\nnn)
=\sigma_9\nnn-\frac{(\sigma_9\nnn,\,\chi\nnn)_{L^2(\OO,\,R)}}{|\chi\nnn|_{L^2(\OO,\,R)}^2}\chi\nnn
=\sigma_9\nnn+\frac{1}{10}\chi\nnn,
\end{align*}
from which it follows $3P_{\chi^\bot}^\OO(\sigma_3\nnn)-P_{\chi^\bot}^\OO(\sigma_9\nnn)
=3\sigma_{3}\nnn-\frac{9}{10}\chi\nnn-\sigma_9\nnn
-\frac{1}{10}\chi\nnn=0$. 
Therefore, if the functions $\chi\E_0^\OO P_{\chi^\bot}^\OO(\sigma_3\nnn)$ and
$\chi\E_0^\OO P_{\chi^\bot}^\OO(\sigma_9\nnn)$
are in the family $\{\chi\E_0^\OO P_{\chi^\bot}^\OO(\pi_i\nnn)\mid i\le M\}$, then the family is linearly dependent;
it follows that~$Q_f^M\NN\subset\NN$ contains nonzero vectors.

%
\bigskip\noindent
{\bf Acknowledgments.} The author gratefully acknowledges partial support from the Austrian Science Fund~(FWF): P~26034-N25.

%


\end{document}